\documentclass[a4paper]{amsart}

\usepackage[latin1]{inputenc}
\usepackage[T1]{fontenc}
\usepackage[english]{babel}
\usepackage{amsmath,amssymb,amsthm}
\usepackage{latexsym}
\usepackage{delarray}
\usepackage{bbm}
\usepackage{hyperref}
\usepackage[pdftex,usenames,dvipsnames]{color}
\usepackage{bbding,trfsigns}
\usepackage{wasysym}
\usepackage{datetime}
\usepackage{mathrsfs}

\usepackage{tikz}

\newtheorem{Th}{Theorem}[section]
\newtheorem{Prop}[Th]{Proposition}
\newtheorem{Lem}[Th]{Lemma}
\newtheorem{Cor}[Th]{Corollary}
\newtheorem{Rem}[Th]{Remark}

\newcommand{\N}{\mathbb{N}}
\newcommand{\R}{\mathbb{R}}
\newcommand{\Rn}{\mathbb{R}^{n}}

\newcommand{\Nn}{\mathbb{N}^{n}}

\newcommand{\G}{\Gamma}
\newcommand{\g}{\gamma}

\newcommand{\DD}{\mathcal{D}}
\newcommand{\FF}{\mathcal{F}}

\newcommand{\KK}{\mathcal{K}}
\newcommand{\LL}{\mathcal{L}}
\newcommand{\MM}{\mathcal{M}}

\newcommand{\Ss}{\mathcal{S}}

\newcommand{\ZZ}{\mathcal{Z}}

\newcommand{\abs}[1]{\left|{#1}\right|}
\newcommand{\brkt}[1]{\left({#1}\right)}
\newcommand{\set}[1]{\left\{{#1}\right\}}

\newcommand{\eps}{\varepsilon}

\title[Solvability of DNR problems]{Solvability of the Dirichlet, Neumann and the regularity problems for parabolic equations with H\"older continuous coefficients}

\author[]{Alejandro J. Castro}
\author[]{Salvador Rodr\'iguez-L\'opez}
\author[]{Wolfgang Staubach}

\address{\newline
        Alejandro J. Castro, Wolfgang Staubach \newline
        Department of  Mathematics, Uppsala University, \newline
        S-751 06 Uppsala, Sweden}
\email{alejandro.castro@math.uu.se, wolfgang.staubach@math.uu.se}

\address{\newline
        Salvador Rodr\'iguez-L\'opez \newline
        Department of Mathematics, Stockholm University,\newline
        SE - 106 91 Stockholm, Sweden}
        \email{s.rodriguez-lopez@math.su.se }
\thanks{
The first author is partially supported by Swedish Research Council Grant 621-2011-3629.
The second author is partially supported by the Spanish Government grant
 MTM2013-40985-P. The third author is partially supported by a grant from the Crafoord foundation}


\begin{document}

\maketitle

\begin{abstract}
We establish the $L^2$-solvability of Dirichlet, Neumann and regularity problems for divergence-form heat (or diffusion) equations with time-independent H\"older-continuous diffusion coefficients, on bounded Lipschitz domains in $\mathbb{R}^n$. This is achieved through the demonstration of invertibility of the relevant layer-potentials which is in turn based on Fredholm theory and a systematic transference scheme which yields suitable parabolic Rellich-type estimates.
\end{abstract}

 \tableofcontents

\section{Introduction}
In this paper, we prove the $L^2$ solvability of the Dirichlet, Neumann and Regularity problems (DNR problems for short) for divergence-form parabolic equations of the form $$\partial_t u(X,t)-\nabla_X\cdot \big(A(X) \nabla_X u(X,t)\big)= 0$$ on bounded Lipschitz domains in $\mathbb{R}^n$ ($n\geq 3$), under the assumptions that $A(X)$ is uniformly elliptic, symmetric and H\"older-continuous.\\

Let us very briefly recall some of the basic results in the field of second-order parabolic boundary value problems with time-independent diffusion coefficients on low regularity domains, which are the predecessors of the present paper. For the sake of brevity, we confine ourselves to only mention those investigations for parabolic equations which have dealt with the solvability of the boundary value problems mentioned above. In \cite{FaRi}, E. Fabes and N. Rivi\`ere proved the solvability of the $L^2$ Dirichlet and Neumann problems on bounded $C^1$ domains. This paper paved the way for subsequent developments in the field. Since the appearance of \cite{FaRi}, the investigations of the parabolic boundary value problems have been concerned with either lowering the regularity of the boundary of the domain, or lowering the regularity of the matrix $A$ appearing in the equation, or both. Another goal is to consider $L^p$ boundary value problems for various values of $p$ (i.e. $p$ other than 2). \\

But in this paper, it is the boundary value problems for time-independent domains that concern us. In \cite{FaSa} E. Fabes and S. Salsa investigated the caloric measure and the $L^p$ ($p\geq 2$) solvability of the initial-Dirichlet problem for the usual heat equation in Lipschitz cylinders. In paper \cite{BrownPhD}, R. Brown studied the $L^2$ boundary value problems and the layer potentials for the heat equation on bounded Lipschitz domains in $\mathbb{R}^n$. Next step was taken by M. Mitrea in \cite{Mi} where he proved the $L^p$ solvability (for suitable values of $p$), of DNR problems for divergence-type heat equations with smooth diffusion coefficients on compact manifolds with Lipschitz boundary. Our investigation in this paper is the continuation of these lines of studies by further pushing down the regularity of the diffusion coefficients and assuming only H\"older continuity, which together with the assumptions of ellipticity and symmetry, will yield the $L^2$ solvability of DNR problems.\\

In \cite{MiTa}, M. Mitrea and M. Taylor proved the solvability of the DNR problems for elliptic equations involving the Laplace-Beltarmi operator with H\"older-continuous metrics on Riemannian manifolds with Lipschitz boundary. Later in \cite{KeSh2}, C. Kenig and Z. Shen used the method of layer potentials to study $L^2$ boundary value problems in a bounded Lipschitz domain in $\mathbb{R}^n$, with $n\geq 3$, for a family of second-order elliptic systems with rapidly oscillating periodic coefficients. As a consequence, they also established the solvability of the DNR problems for divergence-form elliptic equations $\nabla_X\cdot \big(A(X) \nabla_X u(X)\big)= 0$ on the aforementioned domains, under the assumptions that $A(X)$ is uniformly elliptic, symmetric, periodic and H\"older-continuous. Our paper could be considered as a parabolic counterpart of \cite{KeSh2} and \cite{MiTa}.\\

We shall now briefly describe the main results of the paper and the structure of this manuscript.
We recall that a bounded domain $\Omega \subset \R^n$ is called a \emph{Lipschitz domain} (with Lipschitz constant $M>0$)
if $\partial \Omega$ can be covered by finitely many open circular cylinders whose bases have positive
distance from $\partial \Omega ,$ and corresponding to each cylinder $\ZZ \subset \Rn$ there exists:
\begin{itemize}
  \item a coordinate system $(x',x_n)$, with $x' \in \R^{n-1}$ and $x_n \in \R$,
	such that the $x_n$-axis is parallel to the axis of $\ZZ$;
  \item a function $\varphi : \R^{n-1} \longrightarrow \R$ satisfying the Lipschitz condition
	$$|\varphi(x')-\varphi(y')| \leq M |x'-y'|, \quad x',y' \in \R^{n-1},$$
	such that
\begin{equation}\label{defn: The lipshitz domain Omega}
  \qquad \qquad \Omega \cap \ZZ = \{(x',x_n) \in \ZZ : x_n > \varphi(x')\} \quad \text{and} \quad
	   \partial \Omega \cap \ZZ = \{(x',x_n) \in \ZZ : x_n = \varphi(x')\}.
\end{equation}
\end{itemize}

As mentioned earlier, we consider the \emph{parabolic divergence-type} equation
\begin{equation}\label{eq:parabolic}
	\LL_A u(X,t)
		:= \partial_t u(X,t) - \nabla_X\cdot \big(A(X) \nabla_X u(X,t)\big)
		= 0 \quad \text{in} \quad \Omega  \times (0,T),
\end{equation}
where $0<T<\infty$ and $\Omega$ is an open bounded Lipschitz domain in $\R^n$,
$n \geq 3.$ We assume that the real matrix $A(X)=(a_{ij}(X))$  verifies the following properties:
\begin{itemize}
	\item[\label{A1} \small{$(A1)$}] \textit{Independence of the time--variable}: $A=A(X)$;
	\item[\label{A2} \small{$(A2)$}] \textit{Symmetry}: $a_{ij}=a_{ji}$, \ $i,j=1, \dots, n$;
	\item[\label{A3} \small{$(A3)$}] \textit{Uniform ellipticity}:  for certain $\mu>0$,
				$$\mu |\xi|^2 \leq \sum_{i,j=1}^n a_{ij}(X) \xi_i \xi_j \leq \frac{1}{\mu} |\xi|^2, \quad X, \xi \in\R^n;$$
	\item[\label{A4} \small{$(A4)$}] \textit{H\"older regularity}: for some $\kappa>0$ and $0<\alpha \leq 1$,
		      $$|a_{ij}(X)-a_{ij}(Y)| \leq \kappa |X-Y|^\alpha, \quad X,Y \in \Rn, \quad i,j=1, \dots, n.$$
\end{itemize}

To state the aforementioned DNR problems, one defines the \emph{lateral boundary} of $\Omega \times (0,T)$ as  $S_T:=\partial \Omega \times (0,T).$ Moreover the \emph{conormal derivative} $\partial_\nu$ associated with the operator $\LL_A$ will be defined as
\begin{equation}\label{defn:normal derivative}
\partial_\nu u (Q,t)
  := \partial_{\nu_A} u (Q,t)
  := \big \langle  \nabla_Y u(Y,t)_{|_{Y=Q}}, A(Q)\, N_Q \big \rangle, \quad (Q,t) \in S_T,
\end{equation}
where $N_Q=(n_1(Q), \dots, n_n(Q))$ denotes the unit inner normal to $\partial \Omega$ at $Q$,
which is defined a.e. on $\partial \Omega$. The conormal derivative is sometimes denoted by $\partial_{\nu_{A}}$ to emphasise its dependence on the matrix $A.$ One can also define the \emph{tangential derivative} $\nabla_T$ of a function $u$ by

\begin{equation}\label{defn:tangential derivative}
  \nabla_T u(Q,t):=  \nabla_Y u(Y,t)_{|_{Y=Q}} - \big \langle  \nabla_Y u(Y,t)_{|_{Y=Q}}, N_Q \big \rangle N_Q, \quad (Q,t) \in S_T.
\end{equation}
\\

Given these preliminaries, we are interested in the solvability, in the weak sense, (see Section \ref{sec:SIBVP} for the proper statements) of the following problems: \\

 \underline{Dirichlet's problem} \hspace{5.5cm}
\underline{Neumann's problem}
$$\left\{
	\begin{array}{l}
		\LL_A u = 0 \quad \text{in} \quad \Omega \times (0,T) \\
		\quad \\
         u(X,0)= 0, \quad  X \in \Omega
        \quad \\
		u=f \in L^2(S_T) \quad \text{on } S_T \\
	\end{array} \right.
	\qquad \qquad \qquad
	\left\{
	\begin{array}{l}
		\LL_A u = 0 \quad \text{in} \quad \Omega \times (0,T) \\
         \quad \\
         u(X,0)= 0, \quad  X \in \Omega
		\quad \\
		\partial_\nu u =f \in L^2(S_T) \quad \text{on } S_T \\
	\end{array} \right.$$
	
	\quad\\
	
	\underline{Regularity problem}\\
	$$
	 \hspace{1.5cm} \left\{
	\begin{array}{l}
		\LL_A u = 0 \quad \text{in} \quad \Omega \times (0,T) \\
         \quad \\
         u(X,0)= 0, \quad  X \in \Omega
		\quad \\
		u=f \in H^{1,1/2}(S_T) \quad \text{on } S_T \\
	\end{array} \right. \hspace{9cm}$$
	
	\quad\\

To achieve our goals, in Section \ref{sec:Basic notations and tools} we introduce the notations and recall some basic harmonic analytic tools which will be used throughout this paper. In Section \ref{sec:FS} we prove quite a few new estimates for various derivatives of the fundamental solution of parabolic divergence-type operators with H\"older-continuous diffusion coefficients, and also prove the corresponding estimates for the Fourier transform of the fundamental solution in the time variable. It should be noted that although the estimates that are obtained are similar to those in the constant coefficient case, this doesn't simplify the study of the solvability of the DNR problems in our setting. Indeed even in the elliptic divergence-form case studied in \cite{KeSh2} and \cite{MiTa}, one also has the same estimates as those for the constant coefficient Laplacian, but that by no means simplifies the problem. The major difficulty in the study of low regularity elliptic and parabolic problems is to show the invertibility of the corresponding layer potentials which is a significant task for equations with rough coefficients.\\

In Section \ref{sec:PLPO} we study the parabolic single and double-layer potentials associated to the operator $\LL_A$ and establish the $L^2$ boundedness of the boundary singular integral corresponding to this operator as well as the $L^2$ boundedness of the non-tangential maximal function associated to the double-layer potential and that of the normal derivative of the single layer potential. In this section we also prove a couple of jump formulas for the aforementioned operators and also show the $L^2$ boundedness of the non-tangential maximal function associated to a fractional derivative of the single layer potential. The proofs of the $L^2$ boundedness here is somewhat simpler since it can be done using parabolic Calder\'on-Zygmund theory, as was carried out by Brown \cite{BrownPhD} in the constant coefficient case. Next, in Section \ref{sec:Fourier transformed Layer potential} we consider the Fourier transform of the equation $\LL_A u(X,t)=0$ (in time) and establish the estimates proved in Section \ref{sec:PLPO} for the Fourier-transformed equation. These estimates will be very useful for us in one of the central sections of our paper namely Section \ref{sec:Parabolic Rellich}.\\

Here we have an approach which allows us, to transfer in a systematic way, estimates for equations with smooth coefficients to those with H\"older-continuous coefficients.
Briefly, the transference method works as follows. One writes the original H\"older-continuous diffusion matrix $A$ as $B+C$ where $B= \tilde{A}+ A-A^{(r)}$ and $C=A^{(r)}-\tilde{A},$ with a smooth diffusion coefficient $\tilde{A}$ and a suitable $A^{(r)}$ such that $\Vert C\Vert_{\infty}=\Vert A^{(r)}-\tilde{A}\Vert_{\infty}$ can be made arbitrary small by choosing $r$ small enough. Then the first step is to prove Rellich estimates for the smooth part $\tilde{A}$ and then transfer those estimates to $B$. Moreover, those terms in the invertibility estimates for the operator associated with $A$ that involve $C$, can be handled using the smallness of $\Vert C\Vert_{\infty}$ and suitable $L^2$ boundedness estimates. Apart from the proofs of the $L^2-$solvability of DNR problems, this transference method is one of the main achievements of the present paper.\\

In Section \ref{sec:Invertibility} we prove the invertibility results which are the key to the solvability of the DNR problems. This is done by using all the information that we have gathered up to that point and an application of Fredholm theory. Finally in Section \ref{sec:SIBVP} we very briefly outline the solvability of the Dirichlet, Neumann and regularity problems, which is as usual, a standard consequence of the invertibility of the relevant singular integral operators.\\

\noindent\textbf{Acknowledgements.}
The authors would like to thank Kaj Nystr\"om for bringing the subject of low regularity parabolic boundary value problems to their attention and for stimulating and encouraging discussions on this topic. We are also grateful to Carlos Kenig for his encouragement and interest in our work. 
\section{Basic notations and tools}\label{sec:Basic notations and tools}

One of the conventions in the theory of boundary value problems on low-regularity domains, which we shall follow hereafter, is that interior points in the domain $\Omega$ will be denoted by $X,Y$ while those of $\partial \Omega$ will be denoted by $P,Q$. Furthermore, it is also important to warn the reader that, when we write $dP$ or $dQ$ in the integrals that are performed over the boundary, then $dP$ or $dQ$ denote the surface measures $d\sigma(P)$ or $d\sigma(Q)$.\\
We sometimes write $a\lesssim b$ as shorthand notation for $a\leq Cb$.  The constant $C$ hidden in the estimate $a\lesssim b$ can be determined by known parameters in a given situation, but in general the values of such constants are not crucial to the problem at hand. Moreover the value of $C$ may differ from line to line.\\

Now let $\Omega^+:=\Omega$ and $\Omega^- := \Rn \setminus \overline{\Omega}$. Then for some $a>0$ one defines the \emph{non-tangential approaching domains} $\g_+(P) \subset \Omega^+$ and $\g_-(P) \subset \Omega^-$ as follows:
\begin{equation}\label{gammaplus}
 \g_+(P):= \big\{X \in \Omega : |X-P| < (1+a)\, \mathrm{dist}(X,\partial \Omega)\big\},
\end{equation}

\begin{equation}\label{gammaminus}
\g_-(P):= \big\{X \in \Rn \setminus \overline{\Omega} : |X-P| < (1+a)\, \mathrm{dist}(X,\partial \Omega)\big\}.
\end{equation}
It is important to note that, for every $P,Q \in \partial \Omega$ and $X \in \g_\pm(P)$ one has
\begin{equation}\label{eq:gXQXP}
 |X-Q|
    > \mathrm{dist}(X,\partial \Omega)
    > \frac{1}{1+a} |X-P|,
\end{equation}
and
$$|X-Q|
    \geq |P-Q|-|X-P|
    > |P-Q|-(1+a)|X-Q|$$
which implies,
\begin{equation}\label{eq:gXQPQ}
 |X-Q|
    > \frac{1}{2+a} |P-Q|.
\end{equation}

These estimates will be used in Sections \ref{sec:PLPO} and \ref{sec:Fourier transformed Layer potential} in connection to the $L^2$ estimates for non-tangential maximal functions associated to various operators.

Given \eqref{gammaplus} and \eqref{gammaminus} and a function $u$, for every $(P,t) \in S_T$,
the \emph{non-tangential maximal function} $u^*$ is defined by
 \begin{equation}\label{defn:nontangentialmax}
  u^\pm_* (P,t)
	:= \sup_{X \in \g_\pm(P)} |u(X,t)|.
\end{equation}
We consider also the \emph{non-tangential limits}
\begin{equation}\label{defn:nontangentiallimitplus}
u^+(P,t)
	:= \lim_{\substack{X \to P \\ X \in \g_{+}(P)}} |u(X,t)|,
\end{equation}
\begin{equation}\label{defn:nontangentiallimitminus}
  u^-(P,t)
	:= \lim_{\substack{X \to P \\ X \in \g_{-}(P)}} |u(X,t)|.
\end{equation}

The two limits defined above are the ones that appear in the jump relations occurring in this paper, see Sections \ref{sec:PLPO} and \ref{sec:Fourier transformed Layer potential}.\\

We denote by $\MM_1$ and $\MM_{\partial \Omega}$ the \emph{Hardy-Littlewood maximal operators} on $\R$ and $\partial \Omega$ respectively, that is,
$$\MM_1(h)(t)
      = \sup_{r>0} \frac{1}{2r} \int_{|t-s|<r} |h(s)|\, ds,$$
and
$$\MM_{\partial \Omega}(g)(P)
      = \sup_{r>0} \frac{1}{r^{n-1}} \int_{\partial \Omega \cap\{|P-Q|<r\}} |g(Q)|\, dQ.$$
It is well-known that both operators are bounded in $L^2$. We write $\MM$ to refer to $\MM_1$ or $\MM_{\partial \Omega}$, indistinctly. We also recall the following well-known result which will be used in the proofs of our $L^2$ estimates in Sections \ref{sec:PLPO} and \ref{sec:Fourier transformed Layer potential}.

\begin{Lem}\label{Lem:Duo}
    Let $\phi$ be a positive, radial, decreasing and integrable function. Then
    $$\sup_{r>0} |\phi_r * F (\omega)|
	\leq \|\phi\|_1 \MM(F)(\omega),$$
    where $*$ is the usual convolution in $\R^m$, $\phi_r(\omega)=\phi(\omega/r)/r^m$ and $m=1$ or $m=n-1$.
\end{Lem}

We shall also make a repeated use of the so called \emph{Schur's lemma}

\begin{Lem}\label{Lem:Schurslemma}
Let $X,\,Y$ be two measurable spaces. Let $T$ be an integral operator with Schwartz kernel $K(x,y)$, $x\in X$, $y\in Y$
$$ T f(x)=\int_{Y} K(x,y)f(y)\,dy.$$
If
$$\int_{Y} |K(x,y)|\,dy\leq \alpha$$
for almost all $x$ and
$$\int_{X} |K(x,y)|\,dx\leq \beta $$
for almost all $y$, then $T$ extends to a bounded operator $T:L^2 (Y)\to L^2 (X)$ with the operator norm
$$ \Vert T\Vert_{L^2(Y)\to L^2 (X)} \le\sqrt{\alpha\beta}.$$
\end{Lem}

For the application of the Fredholm theory in Section \ref{sec:Invertibility} we would also need the following elementary functional analytic lemmas. We include the proofs for the convenience of the reader.
\begin{Lem}\label{Lem:compact}
  Let $\delta>0$, $\Omega$ be a bounded domain in $\R^n$ and $C$ be the operator defined by
  $$C(g)(P)
    := \int_{\partial \Omega} \frac{g(Q)}{|Q-P|^{n-1-\delta}}\, dQ, \quad P \in \partial \Omega.$$
  Then, $C$ is a compact operator in $L^2(\partial \Omega)$.
\end{Lem}

\begin{proof}
 We write
  $$C(g)(P)
    = \int_{\partial \Omega} K(Q-P) g(Q) dQ, \quad P \in \partial \Omega,$$
  where $K(Z):=|Z|^{-n+1+\delta}$. Analogously we consider, for each $\eps>0$, the operator $C_\eps$
  associated to the kernel $K_\eps(Z):=(|Z|+\eps)^{-n+1+\delta}$. Since,
  $$\int_{\partial \Omega} \int_{\partial \Omega} |K_\eps(Q-P)|^2 \, dQ \, dP
      \leq \frac{|\partial \Omega|^2}{\eps^{2(n-1-\delta)}} < \infty,$$
  for each $\eps>0$, $C_\eps$ is a Hilbert-Schmidt operator and hence it is compact. Moreover Lemma \ref{Lem:Schurslemma} and the Lebesgue dominated convergence theorem yield
  \begin{align*}
    \| C_\eps - C \|_{L^2(\partial \Omega) \to L^2(\partial \Omega)}
	\leq \| K_\eps - K \|_{L^1(\partial \Omega)}
	\longrightarrow 0, \quad \text{as } \eps \to 0.
  \end{align*}
  Therefore $C$ is a compact operator.
\end{proof}

In the estimates for the difference of the parabolic single layer potentials associated to two different diffusion coefficients, the following equality from the theory of Markov chains is useful.

\begin{Prop}\label{Chapman-Kolmogorov}$($Chapman-Kolmogorov formula$)$
Let $r$, $s$ and $t$ be real numbers with $r<s<t$ and let $\lambda>0$. Then
\begin{equation}\label{chapman-kolmogorov}
  \int_{\R^n} e^{\frac{-\lambda |w-v|^2}{2(t-s)}}\, e^{\frac{-\lambda |v-u|^2}{2(s-r)}}\, dv=(2\pi )^{-n/2}\Big(\frac{t-s}{\lambda}\Big)^{n/2}\Big(\frac{s-r}{\lambda}\Big)^{n/2}\Big(\frac{t-r}{\lambda}\Big)^{-n/2} e^{\frac{-\lambda |w-u|^2}{2(t-r)}}.
\end{equation}
\end{Prop}
See \cite[Proposition 3.2.3]{LapJohn}  for a proof of the Chapman-Kolmogorov formula.\\

We shall also need the following elementary functional analytic lemma, which is useful in connection to the invertibility of the boundary singular integral:

\begin{Lem}\label{surjectivity of Fredholm}
Let $T$ be a bounded linear operator from a Hilbert space $H$ into itself.
Furthermore, assume that $T$ is injective, has closed range, and that
$T^* - T $ is a compact operator. Then $T$ is also surjective.
\end{Lem}

\begin{proof}
  Since $T$ is a bounded, injective and has closed range, it is well-known that $\mathrm{ind}\,(T+K)=\mathrm{ind}\,(T)$ for all compact operators $K: H\to H$, where $\mathrm{ind}\,T:=\mathrm{dim\, Ker}\,T-\mathrm{dim\, Coker}\,T,$ denotes the \emph{Fredholm index} of $T.$ Therefore, since $\mathrm{ind}(T)= -\mathrm{ind}(T^*)$ and
$\mathrm{ind}\,(T^*)=\mathrm{ind}\,(T + T^*-T) =\mathrm{ind}\,(T),$ due to the compactness assumption on $T^*-T$, we also have that $\mathrm{ind}\,(T)=-\mathrm{ind}\,(T).$ Therefore $\mathrm{ind}\,(T)=0,$ which together with the injectivity of $T$, yields the surjectivity of $T.$\\
\end{proof}

Lemmas \ref{Lem:Atilde} and \ref{Lem:Ar} below are taken from \cite{KeSh2} and will be used in the proof of the Parabolic Rellich estimates in Section \ref{sec:Parabolic Rellich}.

Let $r_0:=diam \ \Omega <\infty$. We choose a cube $\mathcal{Q}_\Omega \subset \Rn$ such that $\Omega \subset \mathcal{Q}_\Omega$.
We call $2\mathcal{Q}_\Omega$ to the cube with the same centre as $\mathcal{Q}_\Omega$ but with the double size.

\begin{Lem} {\cite[Lemma 7.1]{KeSh2}}\label{Lem:Atilde}
  Given a matrix $A$ satisfying properties \hyperref[A1]{\small{$(A1)$}} -- \hyperref[A4]{\small{$(A4)$}},
  there exists $\tilde{A}\in C^\infty(2 \mathcal{Q}_\Omega \setminus \partial \Omega)$,
  such that $\tilde{A}=A$ on $\partial \Omega .$ Moreover
  \hyperref[A1]{\small{$(A1)$}} -- \hyperref[A4]{\small{$(A4)$}} hold for $\tilde{A}$ with a certain H\"older exponent $\alpha_0 \in (0, \alpha]$ and
  \begin{equation}\label{gradient of A}
   |\nabla \tilde{A}(X)|
        \lesssim \frac{1}{\mathrm{dist}(X,\partial \Omega)^{1-\alpha_0}}, \quad X \in 2 \mathcal{Q}_\Omega \setminus \partial \Omega.
  \end{equation}
\end{Lem}

\begin{Lem}{\cite[Lemma 7.2]{KeSh2}}\label{Lem:Ar}
  Let $A$ be a matrix satisfying properties \hyperref[A1]{\small{$(A1)$}} -- \hyperref[A4]{\small{$(A4)$}}.
  Fix $\theta \in C^\infty_c(-2r_0,2r_0)$ such that $0 \leq \theta \leq 1$ and $\theta \equiv 1$ on $(-r_0,r_0)$.
  Define, for each $0<r\leq 1$,
  \begin{equation*}\label{eq:defArho}
   A^{(r)}(X)
      := \theta \Big( \frac{\mathrm{dist}(X,\partial \Omega)}{r} \Big) A(X)
	+ \Big[1-\theta \Big( \frac{\mathrm{dist}(X,\partial \Omega)}{r} \Big)\Big] \tilde{A} (X), \quad X \in 2 \mathcal{Q}_\Omega,
  \end{equation*}
  where $\tilde{A}$ is the matrix given in $\mathrm{Lemma}$ $\ref{Lem:Atilde}.$ Then $A^{(r)}$ satisfies properties
  \hyperref[A1]{\small{$(A1)$}} -- \hyperref[A4]{\small{$(A4)$}} with the same H\"older exponent $\alpha_0 \in (0, \alpha]$ as in $\mathrm{Lemma}$ $\ref{Lem:Atilde}.$ Moreover $\|A^{(r)} - \tilde{A}\|_\infty \lesssim r^{\alpha_0}.$\\
\end{Lem}

In the investigation of  the solvability of the regularity problem, we would need to deal with fractional Sobolev spaces. The fractional derivatives are defined as follows:

Let $f \in C^{\infty}(-\infty, T)$ and $f(t) = 0$ for $t < 0.$  Then letting $\Gamma(\sigma)$ denote Euler's gamma function, one defines \emph{the fractional integrals} and \emph{fractional derivatives} of $f$ via
\begin{equation*}%
  I_{\sigma} f(t) = \frac{1}{\Gamma(\sigma)}\int_{0}^{t} \frac{f(s)}{(t-s)^{1-\sigma}}\, ds\quad \mathrm{for}\quad 0 < \sigma\leq 1
\end{equation*}
and
\begin{equation}\label{defn:fractional derivative}
  D^{\sigma}_t f(t) = \begin{cases}
    \partial_{t} I_{1-\sigma} f(t), \quad &\mathrm{for}\quad 0 < \sigma< 1\\
    \partial_{t} f(t),\quad &\mathrm{for}\quad \sigma=1.
  \end{cases}
\end{equation}

Furthermore, for $\sigma_1$, $\sigma_2$ in $(0,1)$ and $\sigma_1 + \sigma_2 \leq 1$ one has the following identities for the fractional integrals and derivatives

\begin{equation*}
  I_{\sigma_1}(I_{\sigma_2}(f))= I_{\sigma_1 + \sigma_2} f,
\end{equation*}
\begin{equation*}
  D^{\sigma_1}_t(D^{\sigma_2}_t(f))= D^{\sigma_1 + \sigma_2}_t f.
\end{equation*}

Note that one also has
\begin{equation}\label{eq:fractionalfouriertransform}
    \widehat{D^{\sigma}_t f}(\tau)
        = \frac{(2\pi)^\sigma}{\sqrt{2}}\, (1+i \,\mathrm{sign}(\tau))\, |\tau|^\sigma\, \widehat{f}(\tau),
\end{equation}
where $\widehat{f}(\tau)=\int_{-\infty}^\infty e^{-i \tau t } f(t)\, dt.$ This will be useful in connection to the $L^2$ estimates for the fractional derivative of the single layer potential.

Now given $S_T$ as in the introduction section of this paper, \emph{the fractional Sobolev space} $H^{1, 1/2}(S_T )$ is the closure of space $\{v;\,  v=u|_{S_{T}},\, u\in C^{\infty}_{c}(\mathbb{R}^{n+1}), \, u(X,t)=0 \,\,\textrm{for}\,\, t<0\}$ with respect to the norm
\begin{equation*}
\Vert v\Vert_{H^{1, 1/2}(S_{T})}:=\left\{ \int_{0}^{T}\int_{\partial\Omega} (|\nabla_T v|^2+ |D_{t}^{1/2} v|^2 +|v|^2)\, dP\, dt\right\}^{\frac{1}{2}},
\end{equation*}
where $D^{1/2}_t$ is the fractional derivative defined using \eqref{defn:fractional derivative} and $\nabla_T$ is the tangential derivative defined in \eqref{defn:tangential derivative}.

The following estimate involving fractional derivatives, which has been taken from Brown's thesis \cite{BrownPhD},  will play an important role in the proof of parabolic Rellich estimates in Subsection \ref{subsec:rellich for parab}.

\begin{Lem}\label{Browns fractional}
  Let $f,g \in C^\infty(-\infty,T)$ and $f(t)=g(t)=0$ for $t<0$. Then,
  $$\Big| \int_0^T D_t^{1/4}(f)(t)\,g(t)\, dt \Big|
        \lesssim \Big( \int_0^T |f(t)|^2 \, dt \Big)^{1/2} \Big( \int_0^T |D_t^{1/4}(g)(t)|^2 \,dt \Big)^{1/2}.$$
\end{Lem}

We conclude this section by pointing out that in what follows, due to the elementary and standard nature of the arguments and lack of space, we will follow the common practice of refraining from comments on justifications of the legitimacy of interchanging the order of integrations and that of differentiations and integrations. Certainly, all these operations can be fully justified in each case under consideration by a careful glance at the relevant proofs.

\section{Estimates for the fundamental solution of $\LL_A$ with a H\"older-continuous matrix}\label{sec:FS}

Let $\G$ and $\G^*$ be the fundamental solutions in $\Rn$ for the operators $\LL_A$ and $\LL_A^*$ respectively, that is
$$\LL_A \G(X,t;Y,s) = \delta(X-Y) \delta(t-s), \qquad \LL_A^{*} \G^*(Y,s;X,t) = \delta(X-Y) \delta(t-s),$$
where $\delta$ denotes the Dirac's delta function and $\LL_A^*$ is the adjoint of $\LL_A$. Note that, if $A$ is symmetric,
$\LL_A^* u= - \partial_t - \nabla\cdot\big(A \nabla u \big) $. Also one has
 \begin{equation}\label{eq:sym}
  \G(X,t;Y,s)
      = \G^*(Y,s;X,t), \quad \quad X,Y \in \Rn, \ t\neq s,
\end{equation}
see e.g. \cite[Lemma 3.5]{CDK}.
Moreover, in the case of time--independent matrix $A$, we have that
$$\G(X,t;Y,s)=\G(X,t-s;Y,0), \quad X,Y \in \Rn, \ t\neq s.$$
From now on, we simply write the three-argument function $\G(X,Y,t-s)$ to refer to the above quantity, when there is no cause for confusion.

Recall the relation
\begin{equation*}%
 \int_0^\infty \G(X,Y,t)\, dt
      = \widetilde{\G}(X,Y), \quad X,Y \in \Rn,
\end{equation*}
where $\widetilde{\G}$ represents the fundamental solution to the elliptic operator $\nabla\cdot(A \nabla \cdot),$ see e.g. \cite[p. 895]{Aro}.

Next we collect some pointwise estimates for the fundamental solution, which shall play a basic role for the estimates of the forthcoming sections.
Note that the constants appearing in the estimates below will depend on various combinations of $n$, $\mu$, $\kappa$ and $\alpha$, see the introduction section for the definitions of these latter constants.\\

The following two lemmas are well-known for the fundamental solutions of divergence-type operators under much weaker conditions than those stated here, but since the regularities lower than H\"older-continuity don't concern us in this paper, we confine ourselves to the statements below.

\begin{Lem}[\cite{Aro}]\label{Lem:Aronson}
  Assume that \hyperref[A3]{\small{$(A3)$}} holds.
  Then for every $X,Y \in \Rn$ and $t,s>0$ one has that
  $$|\G(X,t;Y,s)|
	\lesssim \frac{e^{-c|X-Y|^2/(t-s)}}{(t-s)^{n/2}} \chi_{(s,\infty)}(t).$$
\end{Lem}

\begin{Lem}[{\cite[Property 10, p. 163]{PE}}]\label{Lem:Gt}
  Assume that \hyperref[A1]{\small{$(A1)$}} and \hyperref[A3]{\small{$(A3)$}} hold.
  Then, for every $\ell \in \N$, $X,Y \in \Rn$ and $t>0$ one has
  $$|\partial_t^\ell \G(X,Y,t)|
	\lesssim \frac{e^{-c|X-Y|^2/t}}{t^{(n+2\ell)/2}} \chi_{(0,\infty)}(t).$$
\end{Lem}
Later on, in proving the $L^2$ estimates we would also need estimates on the spatial derivatives of the fundamental solution.
\begin{Lem}\label{Lem:GXYm}
  Assume that \hyperref[A1]{\small{$(A1)$}} -- \hyperref[A4]{\small{$(A4)$}} hold.
  Then for every $m \in \Nn$ such that $|m| \leq 2$; $X,Y \in \Rn$ and $t>0$ we have
  $$ |\partial_X^m \G(X,Y,t)| + | \partial_Y^m \G(X,Y,t)|
	\lesssim \frac{e^{-c|X-Y|^2/t}}{t^{(n+|m|)/2}}\chi_{(0,\infty)}(t),$$
  where $\partial_X^m = \partial_{x_1}^{m_1} \cdot \dots \cdot \partial_{x_n}^{m_n}$ if $m=(m_1, \dots, m_n)\in \Nn$ and $|m|=m_1 + \dots + m_n$.
\end{Lem}

\begin{proof}
    The bound for $X$-derivatives can be found in \cite[eq. (13.1), p. 376]{LSU}.  To obtain the corresponding estimate for the $Y$ derivative, define $\overline{\G}(Y,X,t):=\G^*(Y,X,-t)=\G(X,Y,t)$ and observe that it satisfies
$\mathcal{L}_A\overline{\G}(Y,X,t)=\delta(X-Y)\delta(t)$. Therefore
    \begin{align*}
	|\partial_Y^m \G(X,Y,t)|
	    & = |\partial_Y^m \G^*(Y,X,-t)|= |\partial_Y^m \overline{\G}(Y,X,t)|
	       \lesssim \frac{e^{-c|Y-X|^2/t}}{t^{(n+|m|)/2}}\chi_{(0,\infty)}(t).
    \end{align*}
\end{proof}
The following few lemmas are entirely new and will be useful in the later sections.
\begin{Lem}\label{Lem:GXYt}
  Assume that \hyperref[A1]{\small{$(A1)$}} -- \hyperref[A4]{\small{$(A4)$}} hold.
  Then, for every $\ell \in \N$; $X,Y \in \Rn$ and $t>0$ we have that
  \begin{itemize}
      \item[$(i)$]   $\displaystyle | \partial_X \partial_Y \G(X,Y,t)|
					  \lesssim   \frac{e^{-c|X-Y|^2/t}}{t^{(n+2)/2}}\chi_{(0,\infty)}(t)$,
 \item[$(ii)$]  $\displaystyle | \partial_X \partial_t^\ell  \G(X,Y,t)| + | \partial_Y \partial_t^\ell  \G(X,Y,t)|
					  \lesssim\frac{e^{-c|X-Y|^2/t}}{t^{(n+2\ell +1)/2}}\chi_{(0,\infty)}(t)$,
  \end{itemize}
  where $\partial_X= \partial_{x_j}$, for some $j=1, \dots, n$.
\end{Lem}

\begin{proof}
 We only prove $(i)$ in detail and briefly comment on the proof of $(ii)$. Fix $X,Y \in \Rn$, $0<s<t$. Using the well-known semigroup property of the fundamental solution (see e.g. (2.41) in \cite{CDK}), namely
\[
	\G(X,Y,t-s)=\int_{\R^{n}} \G(X,Z,t-r)\,\G(Z,Y,r-s)\, d Z,
\]
for any $r\in (s,t)$, we have
\[
	 \partial_X \partial_Y \G(X,Y,t-s)=\int_{\R^{n}}  \partial_X \G(X,Z,t-r)\,\partial_Y\G(Z,Y,r-s)\, d Z.
\]
Thus  the estimates in Lemma \ref{Lem:GXYm} and Proposition \ref{Chapman-Kolmogorov} yield
\[
	\left|\partial_X \partial_Y \G(X,Y,t-s)\right| \lesssim ((t-r)(r-s) (t-s)^n)^{-1/2} e^{-c{|X-Y|^2}/{(t-s)}}.
\]
Taking $r=(t+s)/2$  we obtain $(i)$.

To prove  $(ii)$, we observe that $\partial_t ( \G(X,Y,t-s))=-\partial_s (\G(X,Y,t-s))$, so following the same line of argument as in  $(i)$ using lemmas \ref{Lem:Gt} and \ref{Lem:GXYt}  yield  $(ii)$.

\end{proof}

For the fractional derivative defined as \eqref{defn:fractional derivative}, we have the following estimates:

\begin{Lem}\label{Lem:GDt}
  Assume that \hyperref[A1]{\small{$(A1)$}} -- \hyperref[A4]{\small{$(A4)$}} hold.
  Then, for every $X,Y \in \Rn$ and $t>0$ we have that
  \begin{itemize}
      \item[$(i)$]   $\displaystyle | D_t^{1/2} \G(X,Y,t)|
					  \lesssim \frac{e^{-c|X-Y|^2/t}}{t^{3/2}|X-Y|^{n-2}} \chi_{(0,\infty)}(t),$
      \quad \\
      \item[$(ii)$]  $\displaystyle | \partial_X D_t^{1/2} \G(X,Y,t)| + | \partial_Y D_t^{1/2}  \G(X,Y,t)|
					  \lesssim \frac{e^{-c|X-Y|^2/t}}{t^{3/2}|X-Y|^{n-1}} \chi_{(0,\infty)}(t),$
      \quad \\
      \item[$(iii)$]  $\displaystyle | \partial_t D_t^{1/2} \G(X,Y,t)|
					  \lesssim \frac{e^{-c|X-Y|^2/t}}{t^{5/2}|X-Y|^{n-2}} \chi_{(0,\infty)}(t).$
  \end{itemize}
\end{Lem}

\begin{proof}
   To prove $(i)$  it is enough to assume that $X\neq Y$, which according to Lemma \ref{Lem:GXYt} part $(ii)$ yields the continuity of $\Gamma (X,Y, t)$ for $X\neq Y$ and $t\in (0,\infty)$. Therefore using the continuity of $\Gamma (X,Y, t)$, the definition of the fractional derivative (with $\Gamma(1/2)=\sqrt{\pi}$), and integration by parts, we obtain
   \begin{align*}
     \sqrt{\pi} D_t^{1/2} \G(X,Y,t)
     	& =\partial_t \int_{0}^{t}\frac{\G(X,Y,t-s)}{\sqrt{s}}\, ds= \frac{\G(X,Y,0)}{\sqrt{t}}+ \int_{0}^{t}\frac{\partial_t\G(X,Y,t-s)}{\sqrt{s}}\, ds \\ 			& = \frac{\G(X,Y,0)}{\sqrt{t}}+ \int_{0}^{t/2}\frac{\partial_s\G(X,Y,s)}{\sqrt{t-s}}\, ds + \int_{t/2}^{t}\frac{\partial_s\G(X,Y,s)}{\sqrt{t-s}}\, ds \\
		&=\frac{\G(X,Y,t/2)}{\sqrt{t/2}}+ \int_{0}^{t/2}\frac{\G(X,Y,s)}{2(t-s)^{3/2}}\, ds + \int_{t/2}^{t}\frac{\partial_s\G(X,Y,s)}{\sqrt{t-s}}\, ds .
   \end{align*}

   Now using Lemma \ref{Lem:Gt} to estimate each of the three terms above yields $(i)$.
   The proofs of $(ii)$ and $(iii)$ differ marginally from that of $(i),$ however in proof of $(ii)$, instead of using Lemma \ref{Lem:Gt}, one has to use Lemma \ref{Lem:GXYt} $(ii).$ The details are left to the reader.
\end{proof}

In our transference scheme which would enable us to transfer invertibility of layer potential operators associated to smooth diffusion coefficients to the invertibility of non-smooth layer potentials, the following simple lemma is very useful.

\begin{Lem}\label{difference of fundsolutions to parabolic}
  Let $A_1$ and $A_2$ be two diffusion coefficients, with the corresponding fundamental solutions $\Gamma_{A_1}( X, Y, t-s)$ and $\Gamma_{A_2}( X, Y, t-s)$. Then the following equality holds for the difference of fundamental solutions:
\begin{align*}
&  \Gamma_{A_1}( X, Y, t-s)-\Gamma_{A_2}( X, Y, t-s)\\
& \qquad = \sum_{i,j=1}^{n} \int_{0}^{\infty} \int_{\mathbb{R}^n} \partial_{z_i} \Gamma_{A_1}( X, Z, u)\,\partial_{z_j} \Gamma_{A_2}( Z, Y, t-s-u)\,(A_1 (Z) -A_2 (Z))\, dZ\, du.
  \end{align*}
\end{Lem}

\begin{proof}
 Integration by parts yields
\begin{align*}
   &\sum_{i,j=1}^{n} \int_{0}^{\infty} \int_{\mathbb{R}^n} \partial_{z_i} \Gamma_{A_1}( X, Z, u)\,\partial_{z_j} \Gamma_{A_2}( Z, Y, t-s-u)\,(A_1 (Z) -A_2 (Z))\, dZ\, du\\
  & \qquad  =  \sum_{i,j=1}^{n}\Big(\int_{0}^{\infty} \int_{\mathbb{R}^n} \partial_{z_i}( A_2 (Z)\partial_{z_j}\Gamma_{A_2}( Z, Y, t-s-u)) \Gamma_{A_1}( X, Z, u)\,\, dZ\, du\\
  & \qquad \quad - \int_{0}^{\infty} \int_{\mathbb{R}^n} \partial_{z_j}(A_1 (Z)\partial_{z_i}\Gamma_{A_1}( Z, Y, u) ) \Gamma_{A_1}(Z, Y, t-s-u)\,\, dZ\, du\Big).
   \end{align*}
		
Now the claimed equality follows by also observing that
\begin{equation*}
\begin{split}
  &\partial_t \Gamma_{A_2}( Z,Y, t-s-u) - \nabla\cdot \big(A_2(Z) \nabla \Gamma_{A_2}( Z, Y,  t-s-u)\big)= \delta( Y-Z) \,\delta(t-s-u),\\
  &\partial_u \Gamma_{A_1}( X, Z, u) - \nabla\cdot \big(A_1 (Z) \nabla \Gamma_{A_1}( X, Z,u)\big)= \delta( X-Z) \,\delta(u),
  \end{split}
\end{equation*}
and that
$$\partial_u \int_{\mathbb{R}^n} \Gamma_{A_1}( X, Z, u)\, dZ
			= \partial_t \int_{\mathbb{R}^n} \Gamma_{A_1}( Z, Y, t-s-u)\, dZ =0,$$
since the integrals that are being differentiated are both equal to $1$ regardless of the time variable.
\end{proof}

In \cite{FaRi}, the problem of the invertibility of boundary singular integrals was handled by utilising the time independence of the Laplacian in the heat equation and performing a Fourier transformation in the time variable. This is an approach which we also adapt here and it has numerous advantages. However, it behoves us then to get suitable estimates for the fundamental solution of the Fourier-transformed operator. To this end, we define the truncated Fourier transform of a function $h$ as
  $$\widehat{h}(\tau)
      := \FF_t(h)(\tau)
      := \int_0^\infty e^{-i \tau t } h(t)\, dt.$$

If \hyperref[A1]{\small{$(A1)$}} is satisfied, we can take the Fourier transform in time in
\eqref{eq:parabolic} and get the new equation
\begin{equation}\label{eq:FourierEq}
  \widehat{\LL}_A \widehat{u}(X,\tau)
    :=-i \tau \widehat{u}(X,\tau) - \nabla_X\cdot\big(A(X) \nabla_X \widehat{u}(X,\tau)\big)
     =0, \quad X \in \Omega,
\end{equation}
for each $\tau$. This way, the parabolic equation becomes an elliptic equation
depending on the parameter $\tau$, which we assume to be fixed hereafter. Moreover, it is clear that
$$\widehat{\G}(X,Y,\tau)
    = \int_0^\infty e^{-i \tau t} \G(X,Y,t)\, dt, \quad X,Y \in \Rn,$$
is the fundamental solution of \eqref{eq:FourierEq}. The following lemmas establish estimates for $\widehat{\G}(X,Y,\tau)$.

\begin{Lem}\label{Lem:GFourier}
  Assume that \hyperref[A1]{\small{$(A1)$}} and \hyperref[A3]{\small{$(A3)$}} hold.
  Then, for every $N \in \N$; $X,Y \in \Rn$ we have that
  $$|\widehat{\G}(X,Y,\tau)|
	\lesssim \frac{\min\{1,(|\tau| |X-Y|^2)^{-N}\}}{|X-Y|^{n-2}},\quad \tau\neq 0.$$
\end{Lem}

\begin{proof}
  An integration by parts and Lemma \ref{Lem:Gt} lead to
  \begin{align*}
      &( |\tau||(X-Y)|^2)^{N} |\widehat{\G}(X,Y,\tau)|
	 = |X-Y|^2 \Big| \int_0^\infty \partial_s^N \Big(e^{-i \tau |X-Y|^2 s} \Big)\G(X,Y,|X-Y|^2s)\, ds \Big| \\
      & \qquad \leq |X-Y|^2  \int_0^\infty |\partial_s^N  \G(X,Y,|X-Y|^2s)|\, ds
	  = |X-Y|^{2+2N}  \int_0^\infty |(\partial_s^N  \G)(X,Y,|X-Y|^2s)|\, ds \\
      & \qquad \lesssim |X-Y|^{2+2N}  \int_0^\infty \frac{e^{-c/s}}{(|X-Y|^2s)^{n/2+N}}\, ds
	  \lesssim \frac{1}{|X-Y|^{n-2}}.
  \end{align*}

\end{proof}
For various derivatives of $\widehat{\G}$ one also has the following estimates:
\begin{Lem}\label{Lem:GXYFourier}
  Assume that \hyperref[A1]{\small{$(A1)$}} -- \hyperref[A4]{\small{$(A4)$}} hold.
  Then for every $q>0$, $m \in \Nn$ such that $|m| \leq 2$ and $X,Y \in \Rn$ we have that
  \begin{itemize}
      \item[$(i)$]   $\displaystyle  |\partial_X^m \widehat{\G}(X,Y,\tau)| + | \partial_Y^m \widehat{\G}(X,Y,\tau)|
					  \lesssim \frac{1}{|X-Y|^{n-2+m}},$
      \quad \\
      \item[$(ii)$]  $\displaystyle | \partial_X \partial_Y \widehat{\G}(X,Y,\tau)|
					  \lesssim \frac{1}{|X-Y|^{n}},$
      \quad \\
      \item[$(iii)$]  $\displaystyle |\partial_X \widehat{\G}(X,Y,\tau)| + |\partial_Y \widehat{\G}(X,Y,\tau)|
					  \lesssim \frac{|\tau|^{-q}}{ |X-Y|^{n-1+2q}},\quad \tau\neq 0,$
      \quad \\
      \item[$(iv)$]  $\displaystyle |\partial_Y \widehat{\G}(X,Y,\tau_1) - \partial_Y \widehat{\G}(X,Y,\tau_2)|
					  \lesssim \frac{|\tau_1 - \tau_2|^{\beta}}{|X-Y|^{n-1-2\beta}},$ for  $n\geq 3$ and all $\beta \in (0,1).$
  \end{itemize}
\end{Lem}

\begin{proof}
  $(i)$ and $(ii)$ are straightforward applications of Lemmas \ref{Lem:GXYm} and \ref{Lem:GXYt} $(i)$.

  For $(iii)$, if $q=N \in \N$, we can proceed as in the
  proof of Lemma \ref{Lem:GFourier}, taking into account Lemma \ref{Lem:GXYt} $(ii)$.
  Finally, if $q$ is not an integer, we use the following simple interpolation argument. Namely, write $q=N+\theta$, where $N=\lfloor q \rfloor \in \N$
  and $\theta \in (0,1)$. Then using (\emph{iii}) for the integer values of $q$, we obtain
  \begin{align*}
    |\partial_X \widehat{\G}(X,Y,\tau)|
	& = |\partial_X \widehat{\G}(X,Y,\tau)|^{1-\theta} |\partial_X \widehat{\G}(X,Y,\tau)|^{\theta} \\
	& \lesssim \Big( \frac{|\tau|^{-N}}{ |X-Y|^{n-1+2N}} \Big)^{1-\theta} \Big( \frac{|\tau|^{-(N+1)}}{ |X-Y|^{n-1+2(N+1)}} \Big)^{\theta}
	  = \frac{|\tau|^{-q}}{ |X-Y|^{n-1+2q}}.
  \end{align*}
  The proof of the estimate for $\partial_Y \widehat{\G}$ is exactly the same.

Statement $(iv)$ is a consequence of the elementary estimate  $|e^{-it\tau_1}-e^{-it\tau_2}| \lesssim |t(\tau_1 - \tau_2)|^\beta$, valid for all $0<\beta\leq 1$ and Lemma \ref{Lem:GXYm}. Indeed we have
  \begin{align*}
    & |\partial_Y \widehat{\G}(X,Y,\tau_1) - \partial_Y \widehat{\G}(X,Y,\tau_2)|
	 \leq \int_0^\infty |e^{-it\tau_1}-e^{-it\tau_2}| |\partial_Y \G(X,Y,t)|\, dt\\
&\lesssim |\tau_1 - \tau_2|^{\beta} \int_0^\infty t^{\beta} |\partial_Y \G(X,Y,t)|\, dt
\lesssim |\tau_1 - \tau_2|^{\beta} \int_0^\infty t^{\beta} \frac{e^{-c|X-Y|^2/t}}{t^{(n+1)/2}}\, dt \\
&\lesssim  \frac{|\tau_1 - \tau_2|^{\beta}}{|X-Y|^{n-1-2\beta}}\int_0^\infty e^{-s} s^{(n-3-2\beta)/2}\, ds \lesssim \frac{|\tau_1 - \tau_2|^{\beta}}{|X-Y|^{n-1-2\beta}} ,
  \end{align*}
provided that $n\geq3$ and $\beta\in(0,1).$

\end{proof}

For the Rellich estimates in Section \ref{sec:Parabolic Rellich} we would also need the following general lemma:
\begin{Lem}\label{Lem:tauq}
  Assume that \hyperref[A1]{\small{$(A1)$}} -- \hyperref[A4]{\small{$(A4)$}} hold.
  Let $q=0$ or $q=1/2$ and let $B$ be the operator defined by
  $$B(g)(P)
    := \int_{\partial \Omega} |\tau|^q \widehat{\G}(P,Q,\tau) g(Q) dQ, \quad P \in \partial \Omega.$$
  Then,
  $$\|B g\|_{L^2(\partial \Omega)}
	\lesssim \|g\|_{L^2(\partial \Omega)}, \quad g \in L^2(\partial \Omega),$$
  where the estimate is uniform in $\tau$.
\end{Lem}

\begin{proof}
  By Lemma \ref{Lem:Schurslemma} and the symmetry of the kernel, it is enough to show that
  $$\int_{\partial \Omega} |\tau^q \widehat{\G}(P,Q,\tau)| dQ
      \lesssim 1, \quad P \in \partial \Omega,$$
  uniformly in $\tau$.

  If $q=0$, by Lemma \ref{Lem:GXYFourier} $(i)$, we just need to check that
    $$\int_{\partial \Omega} \frac{dQ}{|Q-P|^{n-2}}
      \lesssim 1, \quad P \in \partial \Omega.$$
  Locally, we can write $P=(P',\varphi(P'))$, $Q=(Q',\varphi(Q'))$, $P',Q' \in \R^{n-1}$, for a certain
  Lipschitz function $\varphi$. Moreover, since $\Vert \nabla\varphi \Vert_{L^\infty}\lesssim 1$ and $\partial \Omega$ is a compact set, there exists $M>0$ such that
  $|Q'-P'|<M$, for every $Q,P \in \partial \Omega$.
  Then, the above integral is equal to
  \begin{align*}
    \int_{\substack{Q' \in \R^{n-1} \\ |Q'-P'|<M}} \frac{\sqrt{1+|\nabla\varphi(Q')|^2}}{(|Q'-P'|^2+|\varphi(Q')-\varphi(P')|^2)^{(n-2)/2}}\, dQ'
      \lesssim \int_{\substack{Q' \in \R^{n-1} \\ |Q'-P'|<M}} \frac{dQ'}{|Q'-P'|^{n-2}}
      \lesssim 1, \quad P \in \partial \Omega.
  \end{align*}

  Suppose now that $q=1/2$. We are going to proceed as before but with a slight modification in order to avoid the dependence on the parameter $\tau$. We use the improved estimate in Lemma \ref{Lem:GFourier} with $N \geq 1$, and write the corresponding integral in $\R^{n-1}$ as
  \begin{align*}
    & \int_{Q' \in \R^{n-1}}
	\frac{|\tau|^{1/2} \min\{1,[|\tau| (|Q'-P'|^2+|\varphi(Q')-\varphi(P')|^2)]^{-N}\}}{(|Q'-P'|^2+|\varphi(Q')-\varphi(P')|^2)^{(n-2)/2}}
	\sqrt{1+|\nabla\varphi(Q')|^2}\, dQ'\\
    & \qquad \qquad \lesssim \int_{Q' \in \R^{n-1}}
	\frac{|\tau|^{1/2} \min\{1,(|\tau| |Q'-P'|^2)^{-N}\}}{|Q'-P'|^{n-2}}\,  dQ'
      = \int_{Z \in \R^{n-1}} \frac{\min\{1,|Z|^{-2N}\}}{Z^{n-2}} \, dZ\\
    & \qquad \qquad = \int_{\R^{n-1}\cap\{|Z|<1\}} \frac{dZ}{Z^{n-2}}
		    + \int_{\R^{n-1}\cap\{|Z| \geq 1\}} \frac{dZ}{Z^{n-2+2N}}
      \lesssim 1, \quad P \in \partial \Omega.
  \end{align*}
\end{proof}

The following lemma will also be useful in dealing with the transference of the invertibility of boundary singular integrals.
\begin{Lem}\label{difference of fundsolutions to fourierparabolic}
  Let $A_1$ and $A_2$ be two diffusion coefficients, with the corresponding fundamental solutions $\widehat{\Gamma}_{A_1}( X, Y, \tau)$ and $\widehat{\Gamma}_{A_2}( X, Y, \tau)$. Then the following equality holds for the difference of fundamental solutions:
  \begin{equation*}
  \widehat{\Gamma}_{A_2}( X, Y, \tau)-\widehat{\Gamma}_{A_1}( X, Y, \tau)= \sum_{i,j=1}^{n} \int_{\mathbb{R}^n} \partial_{z_i} \widehat{\Gamma}_{A_2}( X, Z, \tau)\,\partial_{z_j} \widehat{\Gamma}_{A_1}( Z, Y, \tau)\,(A_2 (Z) -A_1 (Z))\, dZ.
  \end{equation*}
\end{Lem}
\begin{proof}
 The proof is a consequence of Lemma \ref{difference of fundsolutions to parabolic} and taking the Fourier transform in the time variable.
\end{proof}

\section{Parabolic layer potential operators; SLP, DLP and BSI}\label{sec:PLPO}

The main operators, concerning elliptic and parabolic boundary value problems are the Layer-potential operators. One defines the parabolic \textit{single} and \textit{double-layer potential} operators by
$$\Ss(f)(X,t)
    := \int_0^t \int_{\partial \Omega} \G(X,Q,t-s) f(Q,s)\, dQ\, ds, \quad X \in \Omega, \ t>0,$$
and
\begin{align*}
  \DD(f)(X,t)
       :=  & \int_0^t \int_{\partial \Omega} \partial_\nu \G(X,Q,t-s) f(Q,s) \,dQ\, ds \\
        =  & \sum_{i,j=1}^n \int_0^t \int_{\partial \Omega} a_{ij}(Q) n_j(Q) \partial_{y_i} \G(X,Y,t-s)_{|_{Y=Q}}  f(Q,s)\, dQ\, ds \\
        =: & \sum_{i,j=1}^n \DD^{i,j}(f)(X,t), \quad X \in \Omega, \ t>0.
\end{align*}
The single and double-layer potentials (which we shall sometimes refer to as SLP and DLP) satisfy the equation $\LL_A u=0$ with zero Cauchy data. However to solve the DNR problems, one needs to study the boundary traces of these operators. To this end, one considers the \textit{boundary singular integral} (or BSI for short)
$$\KK(f)(P,t)
      :=  \lim_{\eps \to 0} \KK_\eps(f)(P,t), \quad P \in \partial \Omega, \ t>0, $$
where
\begin{align*}
  \KK_\eps(f)(P,t)
      := & \int_0^{t-\eps} \int_{\partial \Omega} \partial_\nu \G(P,Q,t-s)  f(Q,s)\, dQ\, ds \\
      = & \sum_{i,j=1}^n  \int_0^{t-\eps} \int_{\partial \Omega} a_{ij}(Q) n_j(Q) \partial_{y_i} \G(P,Y,t-s)_{|_{Y=Q}} f(Q,s)\, dQ \,ds \\
      =: & \sum_{i,j=1}^n \KK_{\eps}^{i,j}(f)(P,t), \quad \eps>0.
\end{align*}

\begin{Rem}
 Note that one uses the principal value in the integral defining the boundary singular integral because the points $P$ and $Q$ in the integrand are both on the boundary and can get very close to each other resulting in an undesired behaviour in the exponential function hidden in the integral kernel of $\KK$, when $t$ and $s$ are close to each other. The principal value is not needed in the integral formulas for  the single and double-layer potentials because the point $X$ is an interior point while $Q$ is on the boundary, hence they are separated. Also, as we shall see in $\mathrm{Proposition}$ $\ref{Lem:PV}$ below, it makes no difference if we consider the principal value in ``time'' or in ``space''.
\end{Rem}

\begin{Rem}\label{Rem:KKadjoint}
    We also need to consider the adjoint operator
    $$\KK^*(f)(P,t)
	  := \lim_{\eps \to 0} \sum_{i,j=1}^n \int_0^{t-\eps} \int_{\partial \Omega}
						  a_{ij}(P) n_j(P) \partial_{x_i} \G(X,Q,t-s)_{|_{X=P}} f(Q,s)\, dQ \,ds .$$
    Note this presentation is valid thanks to our assumption $A^*=A$.
    All the results that we are going to prove for $\KK$ are also valid for $\KK^*$, due to the same behaviour of their corresponding integral kernels.
\end{Rem}

\subsection{$L^2$ boundedness of BSI}

For the application of Fredholm theory in showing the invertibility of the relevant boundary integral operators, the following boundedness result is crucial.
\begin{Th}\label{Thm:BSI}
  Assume that \hyperref[A1]{\small{$(A1)$}} -- \hyperref[A4]{\small{$(A4)$}} hold.
  Let $\eps>0$. Then,
  $$\|\KK_{\eps}(f)\|_{L^2(S_\infty)}
      \lesssim \|f\|_{L^2(S_\infty)}, \quad f \in L^2(S_\infty).$$
\end{Th}

\begin{proof}
The idea behind the proof is as follows. First, one takes the Fourier transform in time of $\KK_{\eps}(f)$
and rewrites the resulting operator as an elliptic boundary singular integral plus some error terms.
For the elliptic part which contains cancellations, we take advantage of the results in the elliptic theory, previously established in \cite{KeSh2}, while the error terms will be controlled by the Hardy-Littlewood maximal function $\MM_{\partial \Omega}$. Finally an application of Plancherel's identity allows us to return to the original operator.
Now, let $f \in L^2(S_\infty)$ and $i,j=1, \dots, n$.
For every $P \in \partial \Omega$ we have
  \begin{align}\label{eq:Khat}
    \widehat{\KK_{\eps}^{i,j}(f)}(P,\tau)
	& = \int_0^\infty \int_0^{t-\eps} \int_{\partial \Omega} a_{ij}(Q) n_j(Q) \partial_{y_j} \G(P,Y,t-s)_{|_{Y=Q}} f(Q,s) e^{-i \tau t }\, dQ\, ds\, dt \\
	& = \int_{\partial \Omega}  a_{ij}(Q) n_j(Q) \int_0^\infty f(Q,s) \Big[ \int_{s+\eps}^\infty \partial_{y_i} \G(P,Y,t-s)_{|_{Y=Q}}  e^{-i \tau t } dt \Big] \,ds\, dQ  \nonumber \\
	& = \int_{\partial \Omega}  a_{ij}(Q) n_j(Q) \widehat{f}(Q,\tau) \Big[ \int_{\eps}^\infty \partial_{y_i} \G(P,Y,\zeta)_{|_{Y=Q}}  e^{-i \tau \zeta } d\zeta \Big]\, dQ.\nonumber
  \end{align}
  Then we split the above integral as follows:
  $$\widehat{\KK_{\eps}^{i,j}(f)}
      := I_1(f) + I_2(f) + I_3(f) + I_4(f) + I_5(f),$$
  where
\begin{align*}
& I_1(f)(P,\tau)
		= \int_{\partial \Omega\cap\{\sqrt{\eps} < |P-Q| \leq \frac{1}{\sqrt{|\tau|}}\}}  a_{ij}(Q) n_j(Q) \widehat{f}(Q,\tau)
			\partial_{y_j} \widetilde{\G}(P,Y)_{|_{Y=Q}}\, dQ, \\
& I_2(f)(P,\tau)
		= \int_{\partial \Omega\cap\{|P-Q| \leq \sqrt{\eps}\}}  a_{ij}(Q) n_j(Q) \widehat{f}(Q,\tau)
			\Big[ \int_{\eps}^\infty \partial_{y_i} \G(P,Y,\zeta)_{|_{Y=Q}}  e^{-i \tau \zeta } d\zeta \Big]\, dQ,  \\
& I_3(f)(P,\tau)
		= - \int_{\partial \Omega\cap\{|P-Q| > \sqrt{\eps}\}}  a_{ij}(Q) n_j(Q) \widehat{f}(Q,\tau)
			\Big[ \int_0^{\eps} \partial_{y_i} \G(P,Y,\zeta)_{|_{Y=Q}}  e^{-i \tau \zeta } \, d\zeta \Big] \, dQ, \\
& I_4(f)(P,\tau)
		= \int_{\partial \Omega\cap\{|P-Q| > \max\{\sqrt{\eps}, \frac{1}{\sqrt{|\tau|}}\}\}}  a_{ij}(Q) n_j(Q) \widehat{f}(Q,\tau)
			\Big[ \int_0^{\infty} \partial_{y_j} \G(P,Y,\zeta)_{|_{Y=Q}}  e^{-i \tau \zeta } \, d\zeta \Big] \, dQ, \\
& I_5(f)(P,\tau)
		= \int_{\partial \Omega\cap\{\sqrt{\eps} < |P-Q| \leq \frac{1}{\sqrt{|\tau|}}\} }  a_{ij}(Q) n_j(Q) \widehat{f}(Q,\tau)
			\Big[ \int_0^{\infty} \partial_{y_j} \G(P,Y,\zeta)_{|_{Y=Q}} (e^{-i \tau \zeta } -1)\, d\zeta \Big] \, dQ.
  \end{align*}

First of all, observe that
$$|I_1(f)(P,\tau)|
    \leq 2 \widetilde{\KK}^{i,j}(\widehat{f}(\cdot,\tau))(P),$$
where $\widetilde{\KK}^{i,j}$ represents the elliptic boundary singular integral given by
$$\widetilde{\KK}^{i,j}(g)(P)
    := \sup_{\delta>0} \Big| \int_{\partial \Omega\cap\{|P-Q|>\delta\}} a_{ij}(Q) n_j(Q) \partial_{y_i} \widetilde{\G}(P,Y)_{|_{Y=Q}} g(Q)\, dQ \Big|.$$
Thus, the $L^2(S_\infty)$--boundedness of the integral $I_1$ follows from \cite[Theorem 3.1]{KeSh2}.
Next, we deal with the remaining integrals $I_2$ to $I_5$.

From Lemma \ref{Lem:GXYm} it follows that
$$\Big| \int_{\eps}^\infty \partial_{y_j} \G(P,Y,\zeta)_{|_{Y=Q}}  e^{-i \tau \zeta } d\zeta \Big|
      \lesssim \int_{\eps}^\infty \frac{d\zeta}{\zeta^{(n+1)/2}}
      \sim \frac{1}{\eps^{(n-1)/2}}.$$
Hence \hyperref[A3]{\small{$(A3)$}} yields
\begin{align*}
    |I_2(f)(P,\tau)|
	& \lesssim  \frac{1}{\eps^{(n-1)/2}}\int_{\partial \Omega\cap\{|P-Q| \leq \sqrt{\eps}\}} |\widehat{f}(Q,\tau)|\, dQ
	  \leq \MM_{\partial \Omega} (\widehat{f}(\cdot,\tau))(P).
\end{align*}

Lemma \ref{Lem:GXYm} once again yields
\begin{align*}
   \Big| \int_0^{\eps} \partial_{y_j} \G(P,Y,\zeta)_{|_{Y=Q}}  e^{-i \tau \zeta }\, d\zeta \Big|
   &   \lesssim \int_0^{\eps} \frac{e^{-c|P-Q|^2/\zeta}}{\zeta^{(n+1)/2}}\, d\zeta
      \sim \frac{1}{|P-Q|^{n-1}} \int_{c|P-Q|^2/\eps}^{\infty} e^{-s} s^{(n-1)/2}\, ds \\
   &  \lesssim \frac{e^{-c|P-Q|^2/\eps}}{|P-Q|^{n-1}}.
\end{align*}
Next, we apply Lemma \ref{Lem:Duo} to obtain
\begin{align*}
    |I_3(f)(P,\tau)|
	& \lesssim \int_{\partial \Omega\cap\{|P-Q| > \sqrt{\eps}\}}  \frac{e^{-c|P-Q|^2/\eps}}{\eps^{(n-1)/2}} |\widehat{f}(Q,\tau)|\,  dQ
	  \lesssim \MM_{\partial \Omega} (\widehat{f}(\cdot,\tau))(P).
\end{align*}

On the other hand, using Lemma \ref{Lem:GXYFourier} $(iii)$ for any $N \in \N \setminus\{0\}$, and using Lemma \ref{Lem:Duo}, it follows that
\begin{align*}
    |I_4(f)(P,\tau)|
	& \lesssim \frac{1}{|\tau|^N}\int_{\partial \Omega\cap\{|P-Q| > \frac{1}{\sqrt{|\tau|}}\}} \frac{|\widehat{f}(Q,\tau)|}{|P-Q|^{n-1+2N}}\, dQ
	  \lesssim \MM_{\partial \Omega} (\widehat{f}(\cdot,\tau))(P).
\end{align*}

Finally, we once again use the fact that $|e^{-i \tau \zeta } -1| \lesssim |\tau \zeta|^\beta$, for all $0<\beta\leq 1$,
and therefore Lemma \ref{Lem:GXYm} yields
\begin{align}\label{eq:expbeta}
  \Big| \int_0^{\infty} \partial_{y_j} \G(P,Y,\zeta)_{|_{Y=Q}} (e^{-i \tau \zeta } -1)\, d\zeta \Big|
	& \lesssim |\tau|^\beta \int_0^{\infty} \frac{e^{-c|P-Q|^2/\zeta}}{\zeta^{(n+1-2\beta)/2}}\, d\zeta  \\
	& \sim \frac{|\tau|^\beta}{|P-Q|^{n-1-2\beta}} \int_0^{\infty} e^{-s} s^{(n-3-2\beta)/2}\, ds
	  \sim \frac{|\tau|^\beta}{|P-Q|^{n-1-2\beta}}.\nonumber
\end{align}
Using this estimate and Lemma \ref{Lem:Duo}, we obtain
\begin{align*}
    |I_5(f)(P,\tau)|
	& \lesssim \tau^\beta \int_{\partial \Omega\cap\{ |P-Q| \leq \frac{1}{\sqrt{|\tau|}}\}}  \frac{|\widehat{f}(Q,\tau)|}{|P-Q|^{n-1-2\beta}}\, dQ
	\lesssim \MM_{\partial \Omega} (\widehat{f}(\cdot,\tau))(P).
\end{align*}
Summing all the pieces together, the $L^2$ boundedness of the Hardy-Littlewood maximal function and Plancherel's theorem yield the desired result.
\end{proof}

\begin{Rem}\label{rem: boundedness_result} The Banach-Steinhaus theorem, together with the above uniform $L^2$ boundedness and the pointwise convergence of  $\KK_{\eps}f$  for functions in ${C}^\infty_c(S_\infty)$, yield the convergence of $\KK_{\eps}$ in the $L^2$ norm to the $L^2$ bounded operator $\KK$.  Furthermore, based on this fact and on the estimates for $\G ( P, Q, t)$, which are of the same nature as in the constant coefficient case,  standard Calder\'on-Zygmund theory yields that the operator $\KK$ is bounded on $L^p (S_\infty)$ for any $1<p<\infty$.
\end{Rem}

As a consequence, we obtain the following pointwise convergence result:
\begin{Cor}\label{Cor:bddK}
Assume that \hyperref[A1]{\small{$(A1)$}} -- \hyperref[A4]{\small{$(A4)$}} hold.
The operator given by
 $$\tilde{\KK}(f)(P,t)
       :=  \sup_{\eps > 0} |\KK_{\eps}(f)(P,t)|, \quad P \in \partial \Omega, \ t>0,$$
  is bounded in $L^2(S_\infty)$. Hence, for every $f \in L^2(S_\infty)$ the limit
  $$\KK(f)(P,t)
       :=  \lim_{\eps \to 0} \KK_{\eps}(f)(P,t), \quad P \in \partial \Omega, \ t>0,$$
  exists almost everywhere.
\end{Cor}

\begin{proof}

Define $I^1_\eps(f)(P,t)$ as
\[
	\begin{split}
	\int_0^{t} \int_{\partial \Omega} a_{ij}(Q) n_j(Q) \partial_{y_i} \G(P,Q,t-s) f(Q,s) \left(\chi_{\set{t-s>\eps}}(s)-\chi_{\set{\abs{P-Q}^2+t-s>\eps}} (Q,s)\right)\, dQ \,ds,
	\end{split}
\]
and let $I^2_\eps(f)(P,t):=\KK_{\eps}^{i,j}(f)(P,t)-I^1_\eps(f)(P,t)$. For the sake of simplicity, here and in the rest of the proof,  we suppress the dependency on the $i,j$ parameters.   Then, using the estimates for $ \partial_{Y} \G$ from Lemma \ref{Lem:GXYm}, and following a similar argument as in the proof of Theorem \ref{Thm:BSI} yield

\begin{equation}\label{eq:I want a beer 1}
	\sup_{\eps>0}\abs{I_\eps^1(f)(P,t)}\lesssim \mathcal{M}_{\partial \Omega}\brkt{\mathcal{M}_1 (f)(t)}(P).
\end{equation}

Now observe that for each $(P_1,t_1)\in S(P,t,\eps)$ where
\[
 	S(P,t,\eps):=\brkt{B (P, \sqrt{\eps/3})\cap \partial\Omega}\times (t-\eps/3,t+\eps/3),
\]
we have

\begin{equation}\label{pain_in_ass}
	\begin{split}
	I^2_\eps(f)(P,t) &=\KK(f)(P_1,t_1)-\KK \brkt{f(Q,s)\chi_{\set{\abs{P-Q}^2+t-s\leq\eps}} (Q,s)}(P_1,t_1)+\\
&+\int_0^\infty \int_{\partial \Omega} \brkt{K(P,Q,t-s)-K(P_1,Q,t_1-s)} \chi_{\set{\abs{P-Q}^2+t-s>\eps}} (Q,s) f(Q,s)\, dQ\, ds,\\
	\end{split}
\end{equation}
where $K(P,Q,t)=a_{ij}(Q) n_j(Q) \partial_{y_i} \G(P,Q,t)$.

We claim that, for any $1<q<2$
\begin{equation}\label{eq:I want a beer 2}
	\sup_{\eps>0} \abs{I^2_\eps(f)(P,t)}\lesssim {\mathcal{M}_{\partial\Omega}(\mathcal{M}_1(\KK f)(t)})(P)+\brkt{\mathcal{M}_{\partial\Omega}(\mathcal{M}_1( \abs{f}^q)(t))(P)}^{1/q}+\mathcal{M}_1({\mathcal{M}_{\partial\Omega}{(f)(t)})(P)}.
\end{equation}

Averaging the second term on the right hand side of \eqref{pain_in_ass}  over $S(P,t,\eps)$, H\"older's inequality with $1<q<2$, and the $L^q$ boundedness of the operator $\KK$ (see Remark \ref{rem: boundedness_result} above) yield
\[
\begin{split}
	&{\abs{S(P,t,\eps)}}^{-1}\int_{S(P,t,\eps)} \abs{\KK \brkt{f(Q,s)\chi_{\set{\abs{P-Q}^2+t-s\leq\eps}} (Q,s)}(P_1,t_1)} dP_1\, dt_1\\
&\lesssim \brkt{{\eps^{-\frac{n+1}{2}}}\int_{S(P,t,\eps)} \abs{\KK \brkt{f(Q,s)\chi_{\set{\abs{P-Q}^2+t-s\leq\eps}} (Q,s)}(P_1,t_1)}^q dP_1\, dt_1}^{1/q}\\
&\lesssim  \brkt{{\eps^{-\frac{n+1}{2}}}\int_{\set{\abs{P-Q}^2+t-s\leq\eps}} \abs{f(Q,s)}^q\, dQ ds}^{1/q}\lesssim \brkt{\mathcal{M}_{\partial\Omega}(\mathcal{M}_1( \abs{f(t)}^q))(P)}^{1/q},
\end{split}
\]
with constants independent on $\eps$.  Also, clearly we have
\[
	{\abs{S(P,t,\eps)}}^{-1}\int_{S(P,t,\eps)} \abs{\KK(f)(P_1,t_1)}\, d P_1\, d t_1\lesssim  {\mathcal{M}_{\partial\Omega}(\mathcal{M}_1(\KK f)(t)})(P).
\]

Thus it remains to show that the last term on the right hand side of  \eqref{pain_in_ass}  is bounded uniformly in $\eps$ by ${\mathcal{M}_1(\mathcal{M}_{\partial\Omega}{(f)(P)})(t)}$. To this end we observe that the aforementioned term is bounded by  a constant times
\[
	\begin{split}
	&\int_{\partial\Omega} \int_0^{t-\eps+\abs{P-Q}^2}{\abs{\partial_{y_j} \G (P,Q,t-s)-\partial_{y_j} \G (P_1,Q,t-s)}} \abs{f(Q,s)} \, d Q \, ds \\
&+
\int_{\partial\Omega} \int_0^{t-\eps+\abs{P-Q}^2}\abs{\partial_{y_j} \G (P_1,Q,t-s)-\partial_{y_j} \G (P_1,Q,t_1-s)} \abs{f(Q,s)} \, d Q \, ds.
	\end{split}
\]
Therefore, using the mean value theorem and Lemma \ref{Lem:GXYt} we have that the expression above is bounded by
\[
	\begin{split}
	&\sqrt{\eps} \int_{\partial\Omega}  \int_0^{t-\eps+\abs{P-Q}^2}\sup_{Z\in [P,P_1]} \frac{e^{-c \abs{Z-Q}^2/(t-s)}}{(t-s)^{(n+2)/2}} \abs{f(Q,s)} \, d Q \, ds \\
&+\eps
\int_{\partial\Omega} \int_0^{t-\eps+\abs{P-Q}^2}\sup_{\tau\in [t,t_1]} \frac{e^{-c \abs{P_1-Q}^2/(\tau-s)}}{(\tau-s)^{(n+3)/2}}\abs{f(Q,s)} \, d Q \, ds,
	\end{split}
\]
where $[a,b]$ denotes the segment connecting the points $a$ and $b$.

The analysis of the two terms above are quite similar, so we confine ourselves to deal with the first one.  To this end, we decompose the boundary integral in the first term above into two integrals - one over $\partial\Omega\cap\set{\abs{P-Q}\leq \sqrt{\eps/2}}$  and another over its complement. The first resulting integral can be bounded by
\[
	\begin{split}
	\sqrt{\eps} \int_{\partial\Omega\cap\set{\abs{P-Q}\leq \sqrt{\eps/2}}}  \int_0^{t-\eps/2} \frac{\abs{f(Q,s)}}{(t-s)^{(n+2)/2}}  \, d Q \, ds &\lesssim \eps^{n/2}   \int_0^{t-\eps/2} \frac{\mathcal{M}_{\partial\Omega} (f(\cdot,s))(P)}{(t-s)^{(n+2)/2}}\, ds\\
&\lesssim \mathcal{M}_1(\mathcal{M}_{\partial\Omega}{(f)(P)})(t),
	\end{split}
\]
where the last step follows from Lemma \ref{Lem:Duo}.  Using the triangle inequality we have that $\abs{Z-Q}> C \abs{P-Q}$ for $Z\in [P, P_1]$. Hence the term which is defined by integrating over  $\partial\Omega\cap\set{\abs{P-Q}> \sqrt{\eps/2}}$ is bounded by
\[
	\begin{split}
	&\sqrt{\eps}  \int_0^{\infty}  \frac{e^{-c \eps/(t-s)}}{(t-s)^{5/4}} \int_{\partial\Omega\cap\set{\abs{P-Q}> \sqrt{\eps/2}}} \frac{ e^{-c \abs{P-Q}^2/(t-s)}}{(t-s)^{n/2-1/4}} \abs{f(Q,s)} \, d Q \, ds\\
&\lesssim
\sqrt{\eps}  \int_0^{\infty}  \frac{e^{-c \eps/(t-s)}}{(t-s)^{5/4}} \int_{\partial\Omega\cap\set{\abs{P-Q}> \sqrt{\eps/2}}} \frac{\abs{f(Q,s)} }{\abs{P-Q}^{n-1/2}}  \, d Q \, ds\\
&\lesssim \mathcal{M}_1(\mathcal{M}_{\partial\Omega}{(f)(P)})(t),
	\end{split}
\]
where the last step follows once again from Lemma \ref{Lem:Duo}. This ends the proof of the claim for ${I^2_\eps(f)(P,t)}$.

Now, using \eqref{eq:I want a beer 1} and \eqref{eq:I want a beer 2},  the boundedness of the Hardy-Littlewood maximal functions and Remark \ref{rem: boundedness_result}, we obtain the desired boundedness result for $\widetilde{\KK}$, from which the pointwise converge follows at once.

\end{proof}

In connection to the jump relation for the double-layer potential, the following proposition will prove useful.
\begin{Prop}\label{Lem:PV}
  Let $f \in L^2(S_T)$. Then for a.e. $(P,t) \in S_T$,
  \begin{align*}
     & \lim_{\eps \to 0} \int_0^{t-\eps} \int_{\partial \Omega} \partial_\nu \G(P,Q,t-s) f(Q,s) dQ ds
      = \lim_{\eps \to 0} \int_0^{t} \int_{\partial \Omega\cap \{|P-Q| > \sqrt{\eps}\}} \partial_\nu \G(P,Q,t-s) f(Q,s)\, dQ \, ds.
  \end{align*}
\end{Prop}

\begin{proof}
Since
\begin{align*}
& \int_0^{t-\eps} \int_{\partial \Omega} \partial_\nu \G(P,Q,t-s) f(Q,s) dQ ds
      -\int_0^{t} \int_{\partial \Omega\cap \{|P-Q| > \sqrt{\eps}\}} \partial_\nu \G(P,Q,t-s) f(Q,s)\, dQ\, ds \\
& \qquad= \int_0^{t-\eps} \int_{\partial \Omega\cap \{|P-Q| \leq \sqrt{\eps}\}} \partial_\nu \G(P,Q,t-s) f(Q,s) \,dQ \,ds \\
& \qquad \quad  -\int_{t-\varepsilon}^{t} \int_{\partial \Omega\cap \{|P-Q| >\sqrt{\eps}\}} \partial_\nu \G(P,Q,t-s) f(Q,s)\, dQ \,ds=:I_{\varepsilon}f(P,t)+J_{\varepsilon}f(P,t).
\end{align*}

    It will be enough to show that,  for all $f\in C^{\infty}_{c}(S_T)$
    \begin{equation}\label{max estim for I epsilon}
    \Vert\sup_{\varepsilon >0} |I_{\varepsilon}f| \Vert_{L^2 (S_\infty)}\lesssim \Vert f \Vert_{L^2 (S_\infty)},
    \end{equation}
     \begin{equation}\label{max estim for J epsilon}
     \Vert\sup_{\varepsilon >0} |J_{\varepsilon}f| \Vert_{L^2 (S_\infty)}\lesssim \Vert f\Vert_{L^2 (S_\infty)},
     \end{equation}
and
 \begin{equation}\label{bloody limits}
 \lim_{\varepsilon\to 0}I_{\varepsilon}f=\lim_{\varepsilon\to 0}J_{\varepsilon}f=0.
 \end{equation}

 Now \eqref{max estim for I epsilon} would follow, if we could show that

 \begin{equation}\label{max estim for I epsilon 2}
\sup_{\varepsilon>0}|\int_0^{t-\eps} \int_{\partial \Omega\cap \{|P-Q| \leq \sqrt{\eps}\}} a_{ij}(Q) n_j(Q) \,\partial_{y_i} \G(P,Y,t-s)_{|_{Y=Q}}\, f(Q,s)\, dQ\, ds|\lesssim \MM_1(\MM_{\partial \Omega}(f)(P))(t).
 \end{equation}

  But then Lemma \ref{Lem:GXYm} yields that the left hand side of \eqref{max estim for I epsilon 2} is bounded by
\begin{align*}
    \int_0^{t-\eps} \int_{\partial \Omega\cap \{|P-Q| \leq 2 \sqrt{\eps}\}} \frac{|f(Q,s)|}{(t-s)^{(n+1)/2}}\, dQ \, ds
	&  \lesssim \eps^{(n-1)/2}\int_0^{t-\eps}  \frac{\MM_{\partial \Omega}(f(\cdot, s))(P)}{(t-s)^{(n+1)/2}}\,  ds \\
	& \lesssim \MM_1(\MM_{\partial \Omega}(f)(P))(t),
  \end{align*}
where in the last step we applied Lemma \ref{Lem:Duo}. This shows \eqref{max estim for I epsilon}. To prove \eqref{max estim for J epsilon}, it is enough to show that

 \begin{equation}\label{max estim for J epsilon 2}
\sup_{\varepsilon>0}|\int_{t-\eps}^t \int_{\partial \Omega\cap \{|P-Q| >\sqrt{\eps}\}} a_{ij}(Q) n_j(Q)\,\partial_{y_i} \G(P,Y,t-s)_{|_{Y=Q}}\, f(Q,s)\, dQ\, ds|\lesssim \MM_1(\MM_{\partial \Omega}(f)(P))(t).
 \end{equation}
Hence lemmas \ref{Lem:GXYm} and \ref{Lem:Duo}, yield that the left hand side of \eqref{max estim for J epsilon 2} is dominated by
\begin{align*}
   &\int_{t-\eps}^t \int_{\partial \Omega\cap \{|P-Q| > \sqrt{\eps}\}} \frac{e^{-c|P-Q|^2/(t-s)}}{(t-s)^{(n+1)/2}} |f(Q,s)|\, dQ\, ds
	  \lesssim \int_{t-\eps}^t \int_{\partial \Omega\cap \{|P-Q| > \sqrt{\eps}\}} \frac{|f(Q,s)|}{|P-Q|^{n+1}} \,dQ\, ds \\
	& \qquad \lesssim \frac{1}{\eps} \int_{t-\eps}^t \MM_{\partial \Omega}(f(\cdot, s))(P)\, ds
	  \leq \MM_1(\MM_{\partial \Omega}(f)(P))(t),
  \end{align*}
and \eqref{max estim for J epsilon}  follows easily from this.\\

Since the proofs of \eqref{bloody limits} are similar, we confine ourselves to the proof of $\lim_{\varepsilon\to 0}I_{\varepsilon}f(P,t)=0.$ To this end, without loss of generality, we translate the limit from the point $P$ to the origin and hence aim to prove that $\lim_{\varepsilon\to 0}I_{\varepsilon}f(0,t)=0.$ Note that in dealing with this limit, we can locally write $P=(P',\varphi(P'))$, $Q=(Q',\varphi(Q'))$, $P',Q' \in \R^{n-1}$, for a certain
Lipschitz function $\varphi$ with $|\varphi(P')|\leq |P'|\,\omega (|P'|)$ where $\omega \geq 0$, $\Vert \omega \Vert_{L^\infty}\leq\Vert \nabla\varphi \Vert_{L^\infty}\lesssim 1$ and $\lim_{t\to 0^+} \omega(t)=0$. Then, we have that
\begin{equation*}
  I_\varepsilon f(0, t)
    =  \sum_{i,j=1}^n\int_0^{t-\eps} \int_{\R^{n-1}\cap \{|Q'|^2 +|\varphi(Q')|^2 \leq \eps\}}  \partial_{y_i} \G((0,0),Y,t-s)_{|_{Y=(Q',\varphi(Q'))}} F_{ij}(Q',s)   \, dQ'\, ds,
\end{equation*}
where $F_{ij}(Q',s):=a_{ij}(Q', \varphi(Q')) n_j(Q', \varphi(Q'))f(Q', \varphi(Q'),s)\,\sqrt{1+|\nabla\varphi(Q')|^2}$. Observe that, since $F_{ij}\in L^2(S_T)$, in order to show that $\lim_{\varepsilon\to 0}I_{\varepsilon}f(0,t)=0,$ it is enough, by a standard density argument, to show the result for $F_{ij}(Q', s)= g(Q')\, h(s)$ where $g$ and $h$ are smooth compactly supported functions.
Moreover, using \cite[Lemma 4.4]{DV} we have that $\lim_{\varepsilon \to 0}I_{\varepsilon} f(0, t)= \lim_{\varepsilon \to 0}\tilde{I}_{\varepsilon} f(0, t)$ where

 \begin{equation}\label{thanks DV}
  \tilde{I}_\varepsilon f(0, t)
    :=  \sum_{i,j=1}^n\int_0^{t-\eps} \int_{\R^{n-1}\cap \{|Q'| \leq \sqrt{\eps}\}}  \partial_{y_i} \G((0,0),Y,t-s)_{|_{Y=(Q',\varphi(Q'))}} g(Q')\, h(s)   \, dQ'\, ds.
\end{equation}
Now we split \eqref{thanks DV} into the following three pieces
\begin{align*}
  & \int_0^{t-\eps} \int_{\R^{n-1}\cap \{|Q'| \leq \sqrt{\eps}\}}
    \Big(\partial_{y_i} \G((0,0),Y,t-s)_{|_{Y=(Q',\varphi(Q'))}} - \partial_{y_i} \G((0,0),Y,t-s)_{|_{Y=(Q',0)}}\Big)\\
  & \hspace{4.5cm} \times g(Q')\, h(s)  \, dQ'\, ds \\
  & \qquad + \int_0^{t-\eps} \int_{\R^{n-1}\cap \{|Q'| \leq \sqrt{\eps}\}}
     \partial_{y_i} \G((0,0),Y,t-s)_{|_{Y=(Q',0)}} \Big(g(Q')\, h(s) - g(0)\, h(s) \Big)  \, dQ'\, ds \\
  & \qquad + \int_0^{t-\eps}  g(0)\, h(s) \int_{\R^{n-1}\cap \{|Q'| \leq \sqrt{\eps}\}}
     \partial_{y_i} \G((0,0),Y,t-s)_{|_{Y=(Q',0)}}   \, dQ'\, ds \\
  & =: I_{\varepsilon,1}^{ij} f(0, t) + I_{\varepsilon,2}^{ij} f(0, t) + I_{\varepsilon,3}^{ij} f(0, t).
\end{align*}

Observe that the mean value theorem and Lemma \ref{Lem:GXYm} yield that
\begin{align*}
    |I_{\varepsilon,1}^{ij} f(0, t)|
        & \lesssim \int_0^{t-\eps} \int_{\R^{n-1}\cap \{|Q'| \leq \sqrt{\eps}\}} |\varphi(Q')| \,
        \frac{e^{-c|Q'|^2/(t-s)}}{(t-s)^{(n+2)/2}} \, dQ'\, ds \\
        & \lesssim \int_{\eps}^t \int_{\R^{n-1}\cap \{|Q'| \leq \sqrt{\eps}\}} |Q'| \, \omega(|Q'|) \,
        \frac{e^{-c|Q'|^2/s}}{s^{(n+2)/2}} \, dQ'\, ds \\
        & \lesssim \sqrt{\eps} \Big( \int_{\R^{n-1}\cap \{|Q'| \leq \sqrt{\eps}\}} \omega(|Q'|)  \, dQ' \Big)
        \Big(\int_{\eps}^\infty\frac{1}{s^{(n+2)/2}}\, ds \Big) \\
        & \sim \Big( \int_{\R^{n-1}\cap \{|Q'| \leq 1\}} \omega(\sqrt{\eps}|Q'|)  \, dQ' \Big)
        \Big(\int_{1}^\infty\frac{1}{s^{(n+2)/2}}\, ds \Big) \longrightarrow 0, \quad \varepsilon \to 0.
\end{align*}

Using again Lemma \ref{Lem:GXYm} and the mean value theorem to $g$, we get
\begin{align*}
    |I_{\varepsilon,2}^{ij} f(0, t)|
        & \lesssim \int_0^{t-\eps} \int_{\R^{n-1}\cap \{|Q'| \leq \sqrt{\eps}\}}
      |Q'| \, \frac{e^{-c|Q'|^2/(t-s)}}{(t-s)^{(n+1)/2}} \, dQ'\, ds
      \lesssim \int_{\R^{n-1}\cap \{|Q'| \leq \sqrt{\eps}\}} \int_{\eps}^t
      \frac{e^{-c|Q'|^2/s}}{s^{n/2}} \, ds \, dQ' \\
      & \lesssim \Big(\int_{\R^{n-1}\cap \{|Q'| \leq \sqrt{\eps}\}} \frac{dQ'}{|Q'|^{n-2}}\Big)
      \Big(\int_{0}^\infty e^{-r} r^{n/2-2} \, dr \Big)
      \sim \sqrt{\eps} \longrightarrow 0, \quad \varepsilon \to 0.
\end{align*}

Finally, note that $\G(X,Y,t-s)=\G_0(X,Y,t-s)+\G_1(X,Y,t-s)$, where
$$\G_0(X,Y,t-s)
    = C_n \frac{e^{-\langle A^{-1}(Y)(X-Y), X-Y\rangle /4(t-s)}}{(t-s)^{n/2}(\det A(Y))^{1/2}}.$$
See e.g. \cite{LSU} for the details. Now we claim that

\begin{equation*}%
\int_{\R^{n-1}\cap \{|Q'| \leq \sqrt{\eps}\}}
     \partial_{y_i} \G_0 ((0,0),Y,t-s)_{|_{Y=(Q',0)}}   \, dQ' =0.
\end{equation*}
This follows by using the same reasoning as in the proof of Lemma \ref{Lem:GXYm} and the oddness of the resulting kernel.

Now for $\G_1$ one has the estimate
\begin{equation}\label{eq:Galpha}
|\partial_Y \G_1(X,Y,t-s)|
	\lesssim \frac{e^{-c|X-Y|^2/(t-s)}}{(t-s)^{(n+1-\alpha)/2}}\chi_{(0,\infty)}(t-s),
\end{equation}
where $0<\alpha<1$ is the H\"older exponent appearing in assumption (A4). The estimate \eqref{eq:Galpha} follows from  those in \cite[p. 377]{LSU}, and once again from the same reasoning as in the proof of Lemma \ref{Lem:GXYm}. Therefore

$$\lim_{\varepsilon\to 0} \int_{0}^{t-\varepsilon}\int_{\R^{n-1}\cap \{|Q'| \leq \sqrt{\eps}\}}
     |\partial_{y_i} \G_1 ((0,0),Y,t-s)_{|_{Y=(Q',0)}}|   \, dQ' \, dt=0,$$
since \eqref{eq:Galpha} allows one to apply the Lebesgue dominated convergence theorem. This in turn yields that $\lim_{\varepsilon\to 0}I_{\varepsilon,3}^{ij} f(0, t)=0,$ and summing up, we obtain $ \lim_{\varepsilon \to 0}\tilde{I}_{\varepsilon} f(0, t)=0$ which concludes the proof.

\end{proof}

\subsection{$L^2$ boundedness maximal DLP}
The estimates for non-tangential maximal functions of the layer potentials are crucial for establishing almost everywhere convergence of the solutions to the initial data as well as the jump relations, which will be used in the analysis of the invertibility problems. In analogy with the usual heat equation, the following $L^2$ estimate holds.
\begin{Th}\label{Thm:NMF}
  Assume that \hyperref[A1]{\small{$(A1)$}} -- \hyperref[A4]{\small{$(A4)$}} hold.
Then for $f \in L^2(S_\infty)$ one has
  $$  \Big \|(\DD(f))_{*}^{\pm} \Big \|_{L^2(S_\infty)}
      \lesssim \|f\|_{L^2(S_\infty)},$$
  where $(\cdot)_{*}^{\pm}$ denotes  the non-tangential maximal function defined in \eqref{defn:nontangentialmax}.
\end{Th}

\begin{proof}
  Fix $i,j=1,\dots,n$. Let $P \in \partial \Omega$, $X \in \g_\pm(P)$ and set $\eps:=|X-P|^2$. Then we can write
  \begin{align*}
      & \DD^{i,j}(f)(X,t)\\
      &  \qquad =\KK^{i,j}_\eps(f)(P,t) + J_1(f)(X,P,t) + J_2(f)(X,P,t) + J_3(f)(X,P,t) + J_4(f)(X,P,t),
  \end{align*}
  where
\begin{align*}
& J_1(f)(X,P,t)
				  := \int_0^{t-\eps} \int_{\partial \Omega\cap \{|P-Q| \leq 2 \sqrt{\eps}\}} a_{ij}(Q) n_j(Q) f(Q,s) \\
& \qquad \qquad \qquad \qquad \qquad \times \Big[ \partial_{y_i} \G(X,Y,t-s)_{|_{Y=Q}} - \partial_{y_i} \G(P,Y,t-s)_{|_{Y=Q}} \Big] \, dQ\, ds, \\
& J_2(f)(X,P,t)
				  := \int_0^{t-\eps} \int_{\partial \Omega\cap \{|P-Q| > 2 \sqrt{\eps}\}} a_{ij}(Q) n_j(Q) f(Q,s)\\
& \qquad \qquad \qquad \qquad \qquad\times \Big[ \partial_{y_i} \G(X,Y,t-s)_{|_{Y=Q}} - \partial_{y_i} \G(P,Y,t-s)_{|_{Y=Q}} \Big] \, dQ \,ds,  \\
& J_3(f)(X,P,t)
				  := \int_{t-\eps}^t \int_{\partial \Omega\cap \{|P-Q| \leq 2 \sqrt{\eps}\}} a_{ij}(Q) n_j(Q) f(Q,s)
					\partial_{y_i} \G(X,Y,t-s)_{|_{Y=Q}}\, dQ \,ds,  \\
& J_4(f)(X,P,t)
				  := \int_{t-\eps}^t \int_{\partial \Omega\cap \{|P-Q| > 2 \sqrt{\eps}\}} a_{ij}(Q) n_j(Q) f(Q,s)
					\partial_{y_i} \G(X,Y,t-s)_{|_{Y=Q}}\, dQ \,ds.
  \end{align*}

  By Lemma \ref{Lem:GXYm} we easily obtain
  \begin{align*}
    |J_1(f)(X,P,t)|
	& \lesssim \int_0^{t-\eps} \int_{\partial \Omega\cap \{|P-Q| \leq 2 \sqrt{\eps}\}} \frac{|f(Q,s)|}{(t-s)^{(n+1)/2}}\, dQ \, ds
	  \lesssim \eps^{(n-1)/2}\int_0^{t-\eps}  \frac{\MM_{\partial \Omega}(f(\cdot, s))(P)}{(t-s)^{(n+1)/2}}\,  ds \\
	& \lesssim \MM_1(\MM_{\partial \Omega}(f)(P))(t),
  \end{align*}
  where in the last step we applied Lemma \ref{Lem:Duo}.

  For $J_2$ we first use the mean value theorem and then Lemma \ref{Lem:GXYt} $(i)$, to get
  \begin{align}\label{eq:J2}
    & |J_2(f)(X,P,t)|\\
	& \qquad \lesssim \int_0^{t-\eps} \int_{\partial \Omega\cap \{|P-Q| > 2 \sqrt{\eps}\}} |f(Q,s)| |X-P|
			    \sup_{Z \in [X,P]} \big|\nabla_X \partial_{y_i} \G(X,Y,t-s)_{|_{X=Z, Y=Q}}\big| \,dQ\, ds \nonumber \\
	& \qquad \lesssim \sqrt{\eps} \int_0^{t-\eps} \int_{\partial \Omega\cap \{|P-Q| > 2 \sqrt{\eps}\}}
			    \sup_{Z \in [X,P]} \frac{|f(Q,s)| }{|Z-Q|^{n-1/2}  (t-s)^{5/4}} \,dQ \,ds,\nonumber
  \end{align}
  where $[X,P]$ denotes the line segment which connects the points $X$ and $P$. Moreover, since $|X-P|= \sqrt{\varepsilon}$, if $|P-Q| > 2 \sqrt{\eps}$ then $|X-P| <|P-Q|/2$ and therefore
   \begin{equation}\label{eq:ZQPQ}
    |Z-Q| > |P-Q| - |X-P| > \frac{|P-Q|}{2} \quad \mathrm{for}\quad |P-Q| > 2 \sqrt{\eps}, \ Z \in [X,P].
  \end{equation}
  Hence, the last integral in \eqref{eq:J2} can be controlled by
  \begin{align*}
     \sqrt{\eps} \int_0^{t-\eps} \frac{1}{(t-s)^{5/4}} \int_{\partial \Omega\cap \{|P-Q| > 2 \sqrt{\eps}\}} \frac{|f(Q,s)| }{|P-Q|^{n-1/2}}\, dQ \,ds
	& \lesssim \eps^{1/4} \int_0^{t-\eps} \frac{\MM_{\partial \Omega}(f(\cdot, s))(P)}{(t-s)^{5/4}}\,  ds \\
	& \lesssim \MM_1(\MM_{\partial \Omega}(f)(P))(t),
  \end{align*}
  where we have used Lemma \ref{Lem:Duo} twice.

  For $J_3$, estimate \eqref{eq:gXQXP} and Lemmas \ref{Lem:GXYm} and \ref{Lem:Duo} yield
  \begin{align*}
    |J_3(f)(X,P,t)|
	& \lesssim \eps^{(n-1)/2} \int_{t-\eps}^t \frac{e^{-c\eps/(t-s)}}{(t-s)^{(n+1)/2}} \MM_{\partial \Omega}(f(\cdot, s))(P)\, ds
	  \lesssim \MM_1(\MM_{\partial \Omega}(f)(P))(t).
  \end{align*}

  Finally,  \eqref{eq:ZQPQ}, Lemmas \ref{Lem:GXYm} and \ref{Lem:Duo} yield
  \begin{align*}
    |J_4(f)(X,P,t)|
	& \lesssim \int_{t-\eps}^t \int_{\partial \Omega\cap \{|P-Q| > 2\sqrt{\eps}\}} \frac{e^{-c|P-Q|^2/(t-s)}}{(t-s)^{(n+1)/2}} |f(Q,s)|\, dQ\, ds \\
	& \lesssim \int_{t-\eps}^t \int_{\partial \Omega\cap \{|P-Q| > 2\sqrt{\eps}\}} \frac{|f(Q,s)|}{|P-Q|^{n+1}} \,dQ\, ds \\
	& \lesssim \frac{1}{\eps} \int_{t-\eps}^t \MM_{\partial \Omega}(f(\cdot, s))(P)\, ds
	  \leq \MM_1(\MM_{\partial \Omega}(f)(P))(t).
  \end{align*}

  In conclusion, we obtain the pointwise estimate
  $$\sup_{X \in \g_\pm(P)} |\DD^{i,j}(f)(X,t)|
      \lesssim \tilde{\KK}^{i,j}(f)(P,t) + \MM_1(\MM_{\partial \Omega}(f)(P))(t), \quad (P,t) \in S_\infty.$$
  Hence to end the proof of this proposition, we use Corollary \ref{Cor:bddK} and the
  $L^2$-boundedness of the Hardy-Littlewood maximal operators.
\end{proof}

\begin{Rem}
  Since $\mathrm{Lemma}$ $\ref{Lem:GXYm}$ part $(i)$ and $\mathrm{Lemma}$ $\ref{Lem:GXYt}$ part $(i)$ yield the same estimate $($as far as the decay in $|X-Y|$ is concerned$)$ for the second derivative of $\partial_X ^2 \Gamma(X, Y, t)$ and $\partial_{XY} ^2 \Gamma(X, Y, t)$, using a similar argument as in the proof of $\mathrm{Theorem}$ $\ref{Thm:NMF}$, one has

 \begin{equation}\label{estim:grad of SLP parabolic}
    \Big \| (\nabla \Ss(f))_*^{\pm} \Big\|_{L^2(S_{\infty})}
      \lesssim \|f\|_{L^2(S_{\infty})}, \quad f \in L^2(S_{\infty}).
  \end{equation}
\end{Rem}
This will be important for the invertibility of the BSI associated to parabolic equations.
\subsection{The jump relations}
The discontinuity of the double-layer potential in the non-tangential direction across the boundary is reflected in a precise way in the following jump relation.
\begin{Prop}\label{Prop:Jump}
  Assume that \hyperref[A1]{\small{$(A1)$}} -- \hyperref[A4]{\small{$(A4)$}} hold.
  Let $f \in L^2(S_\infty)$.
  Then,
  $$\lim_{\substack{X \to P \\ X \in \g_{\pm}(P)}} \DD(f)(X,t)
    = \mp \frac{1}{2} f(P,t) + \KK(f)(P,t), \quad \text{a.e. } (P,t) \in S_\infty.$$
\end{Prop}

\begin{proof}
  Taking into account Proposition \ref{Thm:NMF} and using standard techniques,
  it is enough to see that, for every $f \in L^2(S_\infty)$ such that $\widehat{f} \in C_c^\infty(S_\infty)$,
  \begin{equation*}
    \lim_{\substack{X \to P \\ X \in \g_{\pm}(P)}} \widehat{\DD(f)}(X,\tau)
	= \mp \frac{1}{2} \widehat{f}(P,\tau) + \widehat{\KK(f)}(P,\tau), \quad \text{a.e. } P \in \partial \Omega .
  \end{equation*}

  Let $i,j=1,\dots,n$, $P \in \partial \Omega$ and $X \in \g_\pm(P)$. As in \eqref{eq:Khat}, we can write
  \begin{align*}
    \widehat{\DD^{i,j}(f)}(X,\tau)
       = \ & \int_{\partial \Omega}  a_{ij}(Q) n_j(Q) \widehat{f}(Q,\tau) \Big[ \int_0^\infty \partial_{y_i} \G(X,Y,\zeta)_{|_{Y=Q}}  e^{-i \tau \zeta }\, d\zeta \Big] \, dQ \\
       = \ & \int_{\partial \Omega}  a_{ij}(Q) n_j(Q) \widehat{f}(Q,\tau) \partial_{y_i} \widetilde{\G}(X,Y)_{|_{Y=Q}}\, dQ \\
         & + \int_{\partial \Omega}  a_{ij}(Q) n_j(Q) \widehat{f}(Q,\tau) \Big[ \int_0^\infty \partial_{y_i} \G(X,Y,\zeta)_{|_{Y=Q}}  (e^{-i \tau \zeta }-1)\, d\zeta \Big]\, dQ \\
       =: & \mathcal{A}^{i,j}_1(f)(X,\tau) + \mathcal{A}^{i,j}_2(f)(X,\tau).	
  \end{align*}

  Recall that, from the elliptic jump relation (see for example \cite[Theorem 4.6]{KeSh2}),
  \begin{align}\label{eq:jumpellip}
      \lim_{\substack{X \to P \\ X \in \g_{\pm}(P)}} \sum_{i,j=1}^n \mathcal{A}^{i,j}_1(f)(X,\tau)
	  = & \mp \frac{1}{2} \widehat{f}(P,\tau)
	     + \lim_{\eps \to 0} \int_{\partial \Omega\cap \{|P-Q|>\sqrt{\eps}\}}  \partial_\nu \widetilde{\G}(P,Q) \widehat{f}(Q,\tau)dQ.
  \end{align}

  Moreover,
  \begin{equation}\label{eq:limA}
   \lim_{\substack{X \to P \\ X \in \g_{\pm}(P)}} \mathcal{A}^{i,j}_2(f)(X,\tau)
      = \mathcal{A}^{i,j}_2(f)(P,\tau),
  \end{equation}
  by the Lebesgue's dominated convergence theorem. Indeed, estimates \eqref{eq:expbeta} and \eqref{eq:gXQPQ} yield
  \begin{align*}
      |\mathcal{A}^{i,j}_2(f)(X,\tau)|	
	  & \lesssim \int_{\partial \Omega} \frac{|\widehat{f}(Q,\tau)|}{|X-Q|^{n-1-2\beta}}\, dQ
	    \lesssim \int_{\partial \Omega} \frac{dQ}{|P-Q|^{n-1-2\beta}}
	    < \infty,
  \end{align*}
  where it was important that $\widehat{f}$ is bounded. Observe that the constants involved in the above
  inequalities might depend on $\tau$
  but not on $X$.

  It remains to see that the second term in the right hand side of \eqref{eq:jumpellip} plus the sum of the right hand side factors of
  \eqref{eq:limA} equals $\widehat{\KK(f)}$. Since the integrals $\mathcal{A}^{i,j}_2(f)$ are absolutely convergent,
  we easily get
  \begin{align*}
    & \lim_{\eps \to 0} \int_{\partial \Omega\cap \{|P-Q|>\sqrt{\eps}\}} \partial_\nu \widetilde{\G}(P,Q) \widehat{f}(Q,\tau)\, dQ
	+ \sum_{i,j=1}^n \mathcal{A}^{i,j}_2(f)(P,\tau)\\
    & \qquad \qquad= \lim_{\eps \to 0} \int_{\partial \Omega\cap \{|P-Q|>\sqrt{\eps}\}} \partial_\nu \widehat{\G}(P,Q,\tau) \widehat{f}(Q,\tau)\, dQ.
  \end{align*}
  On the other hand, by Proposition \ref{Lem:PV} we have that
  $$\mathcal{K}(f)(P,t)
      = \lim_{\eps \to 0} \int_0^{t} \int_{\partial \Omega\cap \{|P-Q| > \sqrt{\eps}\}} \partial_\nu \G(P,Q,t-s) f(Q,s)\, dQ \,ds.$$
  Hence, proceeding as in \eqref{eq:Khat}, we conclude that
  $$\widehat{\mathcal{K}(f)}(P,\tau)
      = \lim_{\eps \to 0} \int_{\partial \Omega\cap \{|P-Q| > \sqrt{\eps}\}} \partial_\nu \widehat{\G}(P,Q,\tau) \widehat{f}(Q,\tau)\, dQ.$$
  Notice that we have used (see Corollary \ref{Cor:bddK})
  $$\mathcal{K}_\eps(f)(P,\cdot) \longrightarrow \mathcal{K}(f)(P,\cdot), \quad \eps \to 0, \quad \text{in } L^2(0,\infty),$$
  and the continuity of the Fourier transform in $L^2(0,\infty)$.
\end{proof}
For the normal derivative of the single layer potential one has the following jump relation.
\begin{Prop}\label{Prop:Jump2}
  Assume that \hyperref[A1]{\small{$(A1)$}} -- \hyperref[A4]{\small{$(A4)$}} hold and
  let $f \in L^2(S_\infty).$
  Then we have
  $$\partial_\nu (\Ss (f))^{\pm}(P,t)
    = \pm \frac{1}{2} f(P,t) + \KK^*(f)(P,t), \quad \text{a.e. } (P,t) \in S_\infty,$$
where $\partial_\nu$ is defined as in \eqref{defn:normal derivative} and $(\cdot)^{\pm}$ is defined as in \eqref{defn:nontangentiallimitplus} and \eqref{defn:nontangentiallimitminus}.
\end{Prop}

\begin{proof}
 We observe that the integral kernel of $\partial_\nu \Ss $ is  similar to that of DLP's and hence verifies the same estimates. Therefore proceeding as in the proof of Proposition \ref{Prop:Jump}, we will obtain the desired jump relation.
 \end{proof}

As a consequence of the jump relation above, and the definition of $\nabla_T$ one sees just as in the case of constant coefficient operators that
\begin{equation*}
 \nabla_T (\Ss f)^+ = \nabla_T (\Ss f)^- ,
\end{equation*}
where $\nabla_T$ is defined as in \eqref{defn:tangential derivative}, see also \cite{BrownPhD}.

Furthermore, taking the limit in the integral defining the SLP, using the estimate \eqref{Lem:Aronson} and the Lebesgue dominated convergence theorem yield that
\begin{equation*}
 (\Ss f)^+ = (\Ss f)^- .
\end{equation*}

\subsection{$L^2$ boundedness of the maximal fractional time-derivative of SLP} \label{sec:FractDer}
In connection to the problem of invertibility of the single-layer potential and the regularity problem, we would also need an $L^2$ estimate for the fractional derivative of the single-layer potential. The fractional time-derivative of the single -layer potential is defined as
$$D_t^{1/2}\Ss(f)(X,t)
    := \int_0^t \int_{\partial \Omega} D_t^{1/2}\G(X,Q,t-s) f(Q,s)\, dQ \,ds, \quad X \in \Omega, \ t>0.$$

Our goal is to prove the analogues of the estimates derived for the DLP, for the fractional derivative of SLP defined above. To this end we need a series of results showing the $L^2$ boundedness of various operators related to $D^{1/2}\Ss$.
Let $\varphi:\mathbb{R}^{n-1}\mapsto \R$ be Lipschitz and define a kernel $\mathscr{K}$ by

\begin{equation*}%
 \mathscr{K}(P', Q', t):=D_t^{1/2}(\Gamma((P', \varphi(P')),(Q', \varphi(Q')),t ).
\end{equation*}
Since $\varphi$ is Lipschitz it follows that $\mathscr{K}(P', Q', t)$ satisfies all the estimates in Lemma \ref{Lem:GDt}, if we replace $|X-Y|$ in the right hand sides of those estimates by $|P' -Q'|$.
Now let $r(P', Q'):=\sqrt{|P'-Q'|^2 +|\varphi(P')-\varphi(Q')|^2 }$ and for any $\varepsilon >0$ define for $f\in L^2(\R^{n-1}\times (0, \infty))$ the operator $T_\varepsilon$, acting on $f$ by

\begin{equation*}%
  T_\varepsilon f(P',t)= \int_{0}^{t} \int_{r(P', Q')>\varepsilon}\mathscr{K}(P', Q', t-s)\, f(Q', s) \, dQ'\, ds.
\end{equation*}

One now proceeds by investigating the $T_\varepsilon (f)$ and $Tf:=\lim_{\varepsilon\to 0}T_\varepsilon (f)$ for all $f\in L^2 (\R^{n-1}\times (0, \infty))$, however in the proposition below, we view those $f$ whose domain have been extended to $\R^{n-1}\times \R$ in such a way that $f(Q', t)=0,$ for $t\leq 0.$ This enables us to take the Fourier transform of $f(Q', t)$ in $t$, which will prove useful.

\begin{Prop}
 Let $\R_{+}^{n-1}:=\R^{n-1}\times (0, \infty)$ and $f\in L^2 (\R_{+}^{n-1}).$ Then one has
 \begin{enumerate}
   \item [$(i)$] $\Vert T_\varepsilon f\Vert_{L^2 (\R_{+}^{n-1})}\lesssim \Vert f\Vert_{L^2 (\R_{+}^{n-1})}$ where the constant hidden in the right hand side of this estimate doesn't depend on $\varepsilon$.
   \item [$(ii)$] The limit $\lim_{\varepsilon\to 0}T_\varepsilon (f)$ exits in the $L^2 (\R_{+}^{n-1})$ norm and therefore it is legitimate to define $Tf$ as this limit.
 \end{enumerate}
\end{Prop}
\begin{proof}
    To prove $(i)$, using \eqref{eq:fractionalfouriertransform}, we observe that the time-Fourier transform of $T_\varepsilon f$ is bounded by
    $$|\widehat{T_\varepsilon f}(P',\tau)|
        \lesssim \int_{r(P', Q')>\varepsilon} |\tau|^{1/2} |\widehat{\Gamma}((P',\varphi(P')),(Q',\varphi(Q')),\tau)| \, |\widehat{f}(Q', \tau)| \, dQ'.$$
    Now Lemma \ref{Lem:GFourier} yields
    $$|\widehat{\G}((P',\varphi(P')),(Q',\varphi(Q')),\tau)|
	\lesssim \frac{\min\{1,(|\tau| r(P', Q')^2)^{-N}\}}{r(P', Q')^{n-2}},$$
    and therefore we can proceed as in the proof of Lemma \ref{Lem:tauq} with $q=1/2$ to show that
    $$\int_{\R^{n-1}} |\tau|^{1/2} |\widehat{\Gamma}((P',\varphi(P')),(Q',\varphi(Q')),\tau)|  \, dQ'
      \lesssim 1.$$
    Thus, Lemma \ref{Lem:Schurslemma} and Plancherel's theorem imply the boundedness of $T_\varepsilon f$.

    The proof of $(ii)$ is standard since it amounts to show that the sequence $\{T_\varepsilon f\}$ is Cauchy in
    $L^2(\R_{+}^{n-1})$. This is done by using Fubini's, Plancherel's and Lebesgue's dominated convergence theorems. The details are left to the reader.
\end{proof}

Now if one uses the fact that for $(P',t)\ni \mathrm{supp}\, (f)$ one has
$$T f(P',t)= \int_{0}^{t} \int_{\R_{+}^{n-1}}\mathscr{K}(P', Q', t-s)\, f(Q', s) \, dQ'\, ds,$$ then thanks to Lemma \ref{Lem:GDt}, one can follow the exact same steps in \cite[Appendix A, Proposition A.3]{BrownPhD}  to show that $\Vert T f\Vert_{L^2 (\R_{+}^{n-1})}\lesssim \Vert f\Vert_{L^2 (\R_{+}^{n-1})}.$ Moreover, setting $T_{*}f:= \sup_{\varepsilon >0} T_\varepsilon (F)(P',t)$ and following the same argument as in \cite[Appendix A, Theorem A.5]{BrownPhD}, one can show that

\begin{equation}\label{L2 max function estim for T epsilon}
 \Vert T_{*}f\Vert_{L^2 (\R_{+}^{n-1})}\lesssim \Vert f\Vert_{L^2 (\R_{+}^{n-1})}.
\end{equation}
 We would like to emphasise once again that these $L^2$ estimates follow using the same method as in the constant coefficient case in \cite{BrownPhD}, because the proofs there are all only dependent on estimates for the kernel $\mathscr{K}(P', Q', t-s)$, which due to Lemma \ref{Lem:GDt} are the same as those in the constant coefficient case. As usual the maximal function estimate \eqref{L2 max function estim for T epsilon} yields that, for $f\in L^2 (\R_{+}^{n-1})$ one has

$$T f(P',t)= \lim_{\varepsilon\to 0}\int_{0}^{t} \int_{r(P', Q')>\varepsilon}  \mathscr{K}(P', Q', t-s)\, f(Q', s) \, dQ'\, ds,$$
for almost all $(P', t) \in \R_{+}^{n-1}.$

Now using the definition \eqref{defn: The lipshitz domain Omega} of the boundary of the Lipschitz domain $\Omega ,$ one can transfer the operator $T$ to $S_\infty$, by setting

\begin{equation*}
  \widetilde{T} f(P,t)= \lim_{\varepsilon\to 0}\int_{0}^{t} \int_{\partial \Omega \cap\{|P-Q|>\varepsilon\}}  \mathscr{K}(P, Q, t-s)\, f(Q', s) \, dQ'\, ds.
\end{equation*}

Therefore using the boundary coordinates $(P', \varphi(P'))$, one sees that $$\widetilde{T} f((P', \varphi(P')),t)=T(f\,\sqrt{1+|\nabla\varphi|^2}) (P',t),$$ and hence all the results above for $T$ and $T_*$ are also valid for $\widetilde{T}$ and $\widetilde{T}_*$.

Finally, following the same arguments in \cite[Propositions A.7 and A.8]{BrownPhD}, line by line, one has the following two results:

\begin{Prop}\label{Prop:NMFDt1/2}
  Assume that \hyperref[A1]{\small{$(A1)$}} -- \hyperref[A4]{\small{$(A4)$}} hold.
Then
  $$ \Big \| (D_t^{1/2}\Ss f)_*^{\pm} \Big \|_{L^2(S_\infty)}
      \lesssim \|f\|_{L^2(S_\infty)}, \quad f \in L^2(S_\infty).$$
\end{Prop}

To establish the jump relation for the fractional derivative of SLP, one also needs the following result:

\begin{Prop}\label{Prop:JumpDt1/2}
  Assume that \hyperref[A1]{\small{$(A1)$}} -- \hyperref[A4]{\small{$(A4)$}} hold and let $f \in L^2(S_\infty)$.
  Then for a.e. $(P,t) \in S_\infty$,
 \begin{equation*}
      \lim_{\substack{X \to P \\ X \in \g_{\pm}(P)}} D_t^{1/2}\Ss(f)(X,t)= \widetilde{T} f(P,t)=\lim_{\eps \to 0} \int_0^{t} \int_{\partial \Omega\cap \{|P-Q| > \sqrt{\eps}\}} D_t^{1/2} \G(P,Q,t-s) f(Q,s)\, dQ\, ds.
  \end{equation*}
\end{Prop}

The proof of this proposition is once again the same as the corresponding result in the constant coefficient case in Appendix A of \cite{BrownPhD}. From Proposition \ref{Prop:JumpDt1/2}, it follows at once that
\begin{equation*}
 (D_t^{1/2}\Ss f)^+ = (D_t^{1/2}\Ss f)^-.
 \end{equation*}

For the invertibility of the SLP discussed in Subsection \ref{subsec:Invertibility of the SLP}, the following boundedness estimate plays a crucial role.

\begin{Th}\label{boundedness SLP of the difference}
Assume that $A_1$ and $A_2$ satisfy the conditions \hyperref[A1]{\small{$(A1)$}} -- \hyperref[A4]{\small{$(A4)$}} and that $A_1 =A_2$ on $\partial \Omega$. Then one has for $0\leq T<\infty$
\begin{align}\label{estim:boundedness SLP of the difference}
& \Vert \nabla ( \Ss_{A_{1}}- \Ss_{A_{2}} ) f \Vert_{L^{2}(S_{T})}
+ \Vert D_{t}^{1/2}( \Ss_{A_{1}} - \Ss_{A_{2}} ) f\Vert_{L^{2}(S_{T})}
+ \Vert (\Ss_{A_{1}} - \Ss_{A_{2}}) f\Vert_{L^{2}(S_{T})} \\
& \qquad \lesssim \Vert A_{1} - A_{2}\Vert^{1/2}_{\infty} \Vert f\Vert_{L^{2}(S_{T})}. \nonumber
\end{align}
\end{Th}
\begin{proof}
 Since
  $$ D_{t}^{1/2}( \Ss_{A_{1}} - \Ss_{A_{2}} ) f(X, t)= \int_0^\infty \int_{\partial \Omega} D_{t}^{1/2}( \G_{A_1}(X,Q,t-s) - \G_{A_2}(X,Q,t-s) )\, f(Q,s)\, dQ\, ds,$$
using Lemma \ref{difference of fundsolutions to parabolic} we observe that the Schwartz kernel of the integral operator $D_{t}^{1/2}( \Ss_{A_{1}} - \Ss_{A_{2}} )$ is given by
\begin{equation*}
  K_1(X, Q, t-s)=\sum_{i,j=1}^{n} \int_{0}^{\infty} \int_{\mathbb{R}^n} \partial_{z_i} \Gamma_{A_1}( X, Z, u)\,\partial_{z_j}\, D_{t}^{1/2}\Gamma_{A_2}( Z, Q, t-s-u)\,(A_1 (Z) -A_2 (Z))\, dZ\, du.
\end{equation*}

Now we observe that since $A_1= A_2$ on $\partial\Omega$ for any point $Q\in \Omega$ and any $Z\in \overline{\Omega}$ we have

\begin{align*}
A_1 (Z)- A_2 (Z)
	& = A_1 (Z)-A_1 (Q)+A_1 (P)-A_2 (Q)+ A_2 (Q)- A_2 (Z)\\
	&=A_1 (Z)-A_1 (Q)+ A_2 (Q)- A_2 (Z).
\end{align*}

Hence, using the fact that $A_1$ and $A_2$ are H\"older-continuous matrices, we have that for any point $Q\in \Omega$ and any $Z\in \overline{\Omega}$, one has that
\begin{equation}\label{Holder cont of difference}
|A_1 (Z)- A_2 (Z)|\lesssim |Z-Q|^{\alpha},\qquad 0<\alpha<1.
\end{equation}

Taking this fact into account, Lemma \ref{Lem:GXYm} and Lemma \ref{Lem:GDt} part $(ii)$ yield that
\begin{align*}
  &|K_1(X, Q, t-s)| \\
  & \quad \lesssim  \Vert A_1 -A_2 \Vert^{1/2}_{\infty} \int_{0}^{t-s} \int_{\mathbb{R}^n}\frac{e^{-c|X-Z|^2/u}}{u^{(n+1)/2}}\,\frac{e^{-c|Z-Q|^2/(t-s-u)}}{(t-s-u)^{3/2}|Z-Q|^{n-1-\alpha/2}} \, dZ\, du  \nonumber\\
  & \quad \lesssim \Vert A_1 -A_2 \Vert^{1/2}_{\infty}\int_{0}^{t-s} \frac{1}{u^{(n+1)/2}}\frac{1}{(t-s-u)^{3/2}}  \Big(\int_{\mathbb{R}^n}\frac{e^{-c|X-Q-Z|^2/u}\,e^{-c|Z|^2/(t-s-u)}}{|Z|^{n-1-\alpha/2}} \, dZ \Big)\, du \nonumber\\
  & \quad \lesssim \Vert A_1 -A_2 \Vert^{1/2}_{\infty}\int_{0}^{t-s} \frac{1}{u^{(n+1)/2}}\frac{1}{(t-s-u)^{3/2}}  \Big(\int_{|Z|\geq (t-s-u)^{1/2} }\frac{e^{-c|X-Q-Z|^2/u}\,e^{-c|Z|^2/(t-s-u)}}{|Z|^{n-1-\alpha/2}} \, dZ \Big)\, du \nonumber\\
  & \quad \quad +\Vert A_1 -A_2 \Vert^{1/2}_{\infty}\int_{0}^{t-s} \frac{1}{u^{(n+1)/2}}\frac{1}{(t-s-u)^{3/2}}  \Big(\int_{|Z|\leq (t-s-u)^{1/2} }\frac{e^{-c|X-Q-Z|^2/u}\,e^{-c|Z|^2/(t-s-u)}}{|Z|^{n-1-\alpha/2}} \, dZ \Big)\, du \\
  & \quad := \textbf{I}+\textbf{J}\nonumber.
  \end{align*}
Now using Chapman-Kolmogorov formula \eqref{chapman-kolmogorov} we have
\begin{align*}
 \textbf{I}
 	& \lesssim \Vert A_1 -A_2 \Vert^{1/2}_{\infty} \Big(\int_{0}^{t-s} \frac{du}{u^{(n+1)/2 - n/2}(t-s-u)^{3/2 +(n-1-\alpha/2)/2 - n/2}}\Big)\frac{e^{-c|X-Q|^{2}/t-s}}{(t-s)^{n/2}}\\
	&= \Vert A_1 -A_2 \Vert^{1/2}_{\infty}\Big(\int_{0}^{t-s} \frac{du}{u^{1/2}(t-s-u)^{1-\alpha/4}}\Big)\frac{e^{-c|X-Q-Z|^{2}/t-s}}{(t-s)^{n/2}}\\
	&= \Vert A_1 -A_2 \Vert^{1/2}_{\infty} \Big(\int_{0}^{1} \frac{du}{u^{1/2}(1-u)^{1-\alpha/4}}\Big)\frac{e^{-c|X-Q-Z|^{2}/t-s}}{(t-s)^{(n+1)/2 -\alpha/4}} \\
	&= B(1/2,\alpha/4)\Vert A_1 -A_2\Vert^{1/2}_{\infty}  \frac{e^{-c|X-Q|^{2}/t-s}}{(t-s)^{(n+1)/2 -\alpha/4}},
\end{align*}
where $B(\cdot, \cdot)$ is Euler's Beta function.

To estimate $\textbf{J}$ we observe that
\begin{align*}
&\int_{|Z|\leq (t-s-u)^{1/2} }\frac{e^{-c|X-Q-Z|^2/u}\,e^{-c|Z|^2/(t-s-u)}}{|Z|^{{n-1-\alpha/2}}} \, dZ\\
& \qquad \simeq \sum_{j=0}^{\infty}\int_{ 2^{-j-1}(t-s-u)^{1/2} <|Z|\leq 2^{-j}(t-s-u)^{1/2}} \frac{e^{-c|X-Q-Z|^2/u}\,e^{-c|Z|^2/(t-s-u)}}{|Z|^{{n-1-\alpha/2}}} \, dZ \\
& \qquad \simeq \sum_{j=0}^{\infty} 2^{j(n-2)}\,e^{-c 2^{-2j}/2} (t-s-u)^{-({n-1-\alpha/2})/2}\\
& \qquad \qquad \times \int_{2^{-j-1}(t-s-u)^{1/2} <|Z|\leq 2^{-j}(t-s-u)^{1/2}} e^{-c|X-Q-Z|^2/2u} \,e^{-c|Z|^2/ 2(t-s-u)} \, dZ\\
& \qquad \lesssim (t-s-u)^{-({n-1-\alpha/2})/2} \int_{\mathbb{R}^n} e^{-c|X-Q-Z|^2/2u} \,e^{-c|Z|^2/ 2(t-s-u)} \, dZ.
\end{align*}
Therefore, using Chapman-Kolmogorov formula \eqref{chapman-kolmogorov} again we obtain
\begin{align*}
 \textbf{J}
 	& \lesssim \Vert A_1 -A_2 \Vert^{1/2}_{\infty} \Big(\int_{0}^{t-s} \frac{du}{u^{(n+1)/2 - n/2}(t-s-u)^{3/2 +({n-1-\alpha/2})/2 - n/2}}\Big)\frac{e^{-c|X-Q|^{2}/2(t-s)}}{(2t-2s)^{n/2}}\\
	& = \Vert A_1 -A_2 \Vert^{1/2}_{\infty}\Big(\int_{0}^{t-s} \frac{du}{u^{1/2}(t-s-u)^{1-\alpha/4}}\Big)\frac{e^{-c|X-Q-Z|^{2}/2(t-s)}}{(2t-2s)^{n/2}}\\
	&\lesssim  B(1/2,\alpha/4)\Vert A_1 -A_2 \Vert^{1/2}_{\infty}  \frac{e^{-c|X-Q|^{2}/2(t-s)}}{(t-s)^{(n+1)/2 -\alpha/4}}.
\end{align*}
Thus
$$|K_1(X, Q, t-s)|\lesssim   \frac{\Vert A_1 -A_2 \Vert^{1/2}_{\infty} }{(|X-Q|^2+t-s)^{(n+1)/2 -\alpha/4}}.$$

For the operator $\nabla( \Ss_{A_{1}} - \Ss_{A_{2}} )$, we observe that its Schwartz kernel is given by
\begin{equation*}
  K_2 (X, Q, t-s)=\sum_{i,j,k=1}^{n} \int_{0}^{\infty} \int_{\mathbb{R}^n} \partial_{x_k}\partial_{z_i} \Gamma_{A_1}( X, Z, u)\,\partial_{z_j}\, \Gamma_{A_2}( Z, Q, t-s-u)\,(A_1 (Z) -A_2 (Z))\, dZ\, du.
\end{equation*}
This, Lemma \ref{Lem:GXYt} part $(i)$, Lemma \ref{Lem:GXYm} and \eqref{Holder cont of difference} yield that
\begin{align*}
& | K_2 (X, Q, t-s)|
 	 \lesssim \int_{0}^{\infty} \int_{\mathbb{R}^n} |\partial_{x_k}\partial_{z_i} \Gamma_{A_1}( X, Z, u)\,\partial_{z_j}\, \Gamma_{A_2}( Z, Q, t-s-u)|\,|A_1 (Z) -A_2 (Z)|\, dZ\, du\\
	& \qquad \lesssim\Vert A_{1} - A_{2}\Vert^{1/2}_{\infty} \int_{0}^{\infty}\int_{\mathbb{R}^n} \frac{1}{(|Z-X|^2+u)^{(n+2)/2}}\frac{1}{(|Z-Q|^2+t-s-u)^{(n+1)/2 -\alpha/2}}\, dZ \, du\\
	& \qquad \lesssim  \frac{\Vert A_{1} - A_{2}\Vert^{1/2}_{\infty}}{(|X-Q|^2 + t-s)^{(n+1-\alpha)/2}}.
\end{align*}

Finally for $ \Ss_{A_{1}} - \Ss_{A_{2}} $ we observe that its Schwartz kernel is
\begin{equation*}
  K_3(X, Q, t-s)=\sum_{i,j,k=1}^{n} \int_{0}^{\infty} \int_{\mathbb{R}^n} \partial_{z_i} \Gamma_{A_1}( X, Z, u)\,\partial_{z_j}\, \Gamma_{A_2}( Z, Q, t-s-u)\,(A_1 (Z) -A_2 (Z))\, dZ\, du.
\end{equation*}
This, Lemma \ref{Lem:GXYm} and \eqref{Holder cont of difference} yield
\begin{align*}
& | K_3(X, Q, t-s)|
 	 \lesssim \int_{0}^{\infty} \int_{\mathbb{R}^n} |\partial_{z_i} \Gamma_{A_1}( X, Z, u)\,\partial_{z_j}\, \Gamma_{A_2}( Z, Q, t-s-u)|\,|A_1 (Z) -A_2 (Z)|\, dZ\, du\\
	& \qquad \lesssim\Vert A_{1} - A_{2}\Vert^{1/2}_{\infty} \int_{0}^{\infty}\int_{\mathbb{R}^n} \frac{1}{(|Z-X|^2+u)^{(n+1)/2}}\frac{1}{(|Z-Q|^2+t-s-u)^{(n+1)/2 -\alpha/2}}\, dZ \, du\\
	& \qquad \lesssim  \frac{\Vert A_{1} - A_{2}\Vert^{1/2}_{\infty}}{(|X-Q|^2 + t-s)^{(n-\alpha)/2}}
	=\frac{\Vert A_{1} - A_{2}\Vert^{1/2}_{\infty}}{(|X-Q|^2 + t-s)^{(n+1)/2-(1+\alpha)/2}}.
\end{align*}

In conclusion, combining the above estimates we get
\begin{align*}
	|K_1(X, Q, t-s)| + |K_2(X, Q, t-s)| + |K_3(X, Q, t-s)|
		\lesssim \frac{\Vert A_{1} - A_{2}\Vert^{1/2}_{\infty}\chi_{(0,t)}(s)}{(|X-Q|^2 + t-s)^{(n+1)/2-\alpha/4}}.
\end{align*}

Now observe that

$$\frac{\chi_{(0,t)}(s)}{(|X-Q|^2 + t-s)^{(n+1)/2-\alpha/4}}
\lesssim \frac{\chi_{(0,t)}(s)}{(t-s)^{1-\alpha/8}} \frac{1}{|P-Q|^{(n-1)/2-\alpha/8}}.$$

Therefore, since the estimate in the proof of Lemma \ref{Lem:tauq} shows that
\begin{align*}
	&\sup_{P\in \partial \Omega} \Big(\int_{\partial \Omega} \frac{dQ}{|P-Q|^{(n-1)/2-\alpha/8}} \Big) \lesssim 1\\
	&\sup_{Q\in \partial \Omega} \Big(\int_{\partial \Omega} \frac{dP}{|P-Q|^{(n-1)/2-\alpha/8}} \Big) \lesssim 1,
\end{align*}
and since
\begin{align*}
	&\sup_{0\leq t\leq T}\Big(\int_0^T \frac{\chi_{(0,t)}(s)}{(t-s)^{1-\alpha/8}}\, ds \Big)=\sup_{0\leq t\leq T}\Big(\int_0^t \frac{ds}{(t-s)^{1-\alpha/8}}\Big)\lesssim \sup_{0\leq t\leq T} t^{\alpha/8}\lesssim 1 \\
	&\sup_{0\leq s\leq T}\Big(\int_0^T \frac{\chi_{(0,t)}(s)}{(t-s)^{1-\alpha/8}}\, dt\Big)=\sup_{0\leq s\leq T}\Big(\int_s^T \frac{dt}{(t-s)^{1-\alpha/8}}\Big)\lesssim \sup_{0\leq s\leq T} (T-s)^{\alpha/8}\lesssim 1,
\end{align*}

Schur's Lemma implies \eqref{estim:boundedness SLP of the difference}.
\end{proof}

\section{Layer potential operators associated to the Fourier-transformed equation}\label{sec:Fourier transformed Layer potential}

As mentioned earlier, the independence of the diffusion coefficient matrix of the time variable is an advantage here which enable us to perform a Fourier transform in the time variable in the parabolic equation and bring it to a parameter-dependent elliptic equation. However, one then needs to establish all the estimates we derived in the previous sections for the corresponding layer potentials of the Fourier-transformed equation.

To this end, we define the Fourier-transformed \textit{single} and \textit{double-layer potential} operators by
$$\mathbb{S}^\tau(g)(X)
    := \int_{\partial \Omega} \widehat{\G}(X,Q,\tau) g(Q)\, dQ, \quad X \in \Omega,$$
and
\begin{align*}
  \mathbb{D}^\tau(g)(X)
       :=  & \int_{\partial \Omega}  \partial_\nu \widehat{\G}(X,Q,\tau) g(Q)\, dQ
        =   \sum_{i,j=1}^n \int_{\partial \Omega} a_{ij}(Q) n_j(Q) \partial_{y_i} \widehat{\G}(X,Y,\tau)_{|_{Y=Q}}  g(Q)\, dQ  \\
        =: & \sum_{i,j=1}^n (\mathbb{D}^\tau)^{i,j}(g)(X), \quad X \in \Omega.
\end{align*}
We also consider the \textit{boundary singular integral}
$$\mathbb{K}^\tau(g)(P)
      :=  \lim_{\eps \to 0} \mathbb{K}^\tau_\eps(g)(P), \quad P \in \partial \Omega,$$
where
\begin{align*}
  \mathbb{K}^\tau_\eps(g)(P)
      := & \int_{\partial \Omega\cap \{|P-Q|>\eps\}} \partial_\nu \widehat{\G}(P,Q,\tau)  g(Q) \,dQ  \\
      = & \sum_{i,j=1}^n  \int_{\partial \Omega\cap \{|P-Q|>\eps\}} a_{ij}(Q) n_j(Q) \partial_{y_i} \widehat{\G}(P,Y,\tau)_{|_{Y=Q}} g(Q)\, dQ  \\
      =: & \sum_{i,j=1}^n (\mathbb{K}^\tau_{\eps})^{i,j}(g)(P), \quad \eps>0.
\end{align*}

\begin{Rem}\label{Rem:Kadjoint}
    Note that we would also need to consider the adjoint operator, which is given by
\begin{equation*}%
(\mathbb{K}^\tau)^*(g)(P)
	  := \lim_{\eps \to 0} \sum_{i,j=1}^n  \int_{\partial \Omega\cap \{|P-Q|>\eps\}}
		a_{ij}(P) n_j(P) \partial_{x_i} \widehat{\G}(X,Q,\tau)_{|_{X=P}} g(Q)\, dQ.
\end{equation*}
    Once again, $(\mathbb{K}^\tau)^*$ and $(\mathbb{K}^\tau)$ will satisfy the same $L^2$ estimates.
\end{Rem}

\begin{Rem}
In subsections $\ref{subsec:The boundary singular integral}$ and $\ref{subsec:The nontangential maximal function}$ the constants hidden in the right hand sides of the estimates are all independent of parameter $\tau.$
\end{Rem}
\subsection{$L^2$ boundedness of the truncated Fourier-transformed BSI}\label{subsec:The boundary singular integral}
The following theorem is one of the main tools in proving the invertibility of Fourier-transformed boundary singular integral.
\begin{Prop}\label{Prop:BSItau}
  Assume that \hyperref[A1]{\small{$(A1)$}} -- \hyperref[A4]{\small{$(A4)$}} hold and
  let $\eps>0$. Then one has
  $$\|\mathbb{K}^\tau_{\eps}(g)\|_{L^2(\partial \Omega)}
      \lesssim \|g\|_{L^2(\partial \Omega)}, \quad g \in L^2(\partial \Omega).$$
\end{Prop}

\begin{proof}
    Let $g \in L^2(\partial \Omega)$ and $i,j=1, \dots, n$. We write, for each $P \in \partial \Omega$,
    \begin{align*}
	(\mathbb{K}^\tau_{\eps})^{i,j}(g)(P)
	  & =  \int_{\partial \Omega\cap \{\eps < |P-Q| \leq \frac{1}{\sqrt{|\tau|}}\}} a_{ij}(Q) n_j(Q) g(Q) \partial_{y_i}\widetilde{\G}(P,Y)_{|_{Y=Q}}\,  dQ \\
	  & \qquad + \int_{\partial \Omega\cap \{|P-Q|>\max\{\eps,\frac{1}{\sqrt{|\tau|}}\}\}} a_{ij}(Q) n_j(Q) g(Q) \Big[\int_0^\infty \partial_{y_i} \G(P,Y,t)_{|_{Y=Q}} e^{-i  \tau t}\, dt \Big] \, dQ \\
	  & \qquad + \int_{\partial \Omega\cap \{\eps < |P-Q| \leq \frac{1}{\sqrt{|\tau|}}\}} a_{ij}(Q) n_j(Q) g(Q) \Big[\int_0^\infty \partial_{y_i} \G(P,Y,t)_{|_{Y=Q}} (e^{-i  \tau t} -1)\, dt \Big] \, dQ \\
	  & =: I_1(g)(P) + I_2(g)(P) + I_3(g)(P).
    \end{align*}
    It is clear that
    $$|I_1(g)(P)|
	\leq 2 \widetilde{\KK}^{i,j}(g)(P).$$
    Moreover, following the same method as in the proof of Theorem \ref{Thm:BSI} for $I_4$ and $I_5 ,$ we obtain the estimate
    $$|I_2(g)(P)| + |I_3(g)(P)|
	  \lesssim \MM_{\partial \Omega}(g)(P).$$

\end{proof}

\begin{Cor}\label{Cor:bddKtau}
Assume that \hyperref[A1]{\small{$(A1)$}} -- \hyperref[A4]{\small{$(A4)$}} hold.
The operator given by
 $$\tilde{\mathbb{K}}^\tau(g)(P)
       :=  \sup_{\eps > 0} |\mathbb{K}^\tau_{\eps}(g)(P)|, \quad P \in \partial \Omega,$$
  is bounded in $L^2(\partial \Omega)$. Hence, for every $g \in L^2(\partial \Omega)$ the limit
  $$\mathbb{K}^\tau(g)(P)
       :=  \lim_{\eps \to 0} \mathbb{K}^\tau_{\eps}(g)(P), \quad P \in \partial \Omega,$$
  exists almost everywhere, and it also defines a bounded operator in $L^2(\partial \Omega)$.
\end{Cor}

\begin{proof}
  The proof is standard and goes along the same lines as that of Corollary \ref{Cor:bddK}.
\end{proof}

\subsection{$L^2$ boundedness of the maximal Fourier-transformed DLP}\label{subsec:The nontangential maximal function}
For the Fourier-transformed double-layer potential one also needs to demonstrate the $L^2$ boundedness of non-tangential maximal function, in order to show the corresponding almost everywhere convergence and jump relations. To this end we have
\begin{Prop}\label{Prop:NMFtau}
  Assume that \hyperref[A1]{\small{$(A1)$}} -- \hyperref[A4]{\small{$(A4)$}} hold.
  Then,
 $$\Big \| (\mathbb{D}^\tau(g))_*^{\pm}\Big\|_{L^2(\partial \Omega)}
      \lesssim \|g\|_{L^2(\partial \Omega)}, \quad g \in L^2(\partial \Omega).$$
\end{Prop}

\begin{proof}
  Fix $i,j=1,\dots,n$. Let $P \in \Omega$, $X \in \g_\pm(P)$ and take $\eps:=|X-P|$. We consider the following decomposition
  $$(\mathbb{D}^\tau)^{i,j}(g)(X)
	= (\mathbb{K}^\tau_{2\eps})^{i,j}(g)(P) + J_1(g)(X,P) + J_2(g)(X,P),$$
  where
  \begin{itemize}
      \item[] $\displaystyle J_1(g)(X,P)
				  := \int_{\partial \Omega \cap\{|P-Q| > 2 \eps\}} a_{ij}(Q) n_j(Q) g(Q)
					  \Big[ \partial_{y_i} \widehat{\G}(X,Y,\tau)_{|_{Y=Q}}
					      - \partial_{y_i} \widehat{\G}(P,Y,\tau)_{|_{Y=Q}} \Big] \, dQ,$ \\ \quad \\

      \item[] $\displaystyle J_2(g)(X,P)
				  := \int_{\partial \Omega \cap\{|P-Q| \leq 2 \eps\}} a_{ij}(Q) n_j(Q) g(Q)
					  \partial_{y_i} \widehat{\G}(X,Y,\tau)_{|_{Y=Q}}\, dQ.$
  \end{itemize}

  For $J_1$ we first apply the mean value theorem, secondly Lemma \ref{Lem:GXYFourier} $(ii)$;
  next the relation \eqref{eq:ZQPQ} and finally Lemma \ref{Lem:Duo} to get
  \begin{align*}
    |J_1(g)(X,P)|
	& \lesssim \int_{\partial \Omega \cap\{|P-Q| > 2 \eps\}} |g(Q)| |X-P| \sup_{Z \in [X,P]} \big|\nabla_X \partial_{y_i} \widehat{\G}(X,Y,\tau)_{|_{X=Z, Y=Q}}\big| \, dQ \\
	& \lesssim \eps \int_{\partial \Omega \cap\{|P-Q| > 2 \eps\}} |g(Q)|  \sup_{Z \in [X,P]} \frac{1}{|Z-Q|^n}\, dQ
	 \lesssim \eps \int_{\partial \Omega \cap\{|P-Q| > 2 \eps\}}  \frac{|g(Q)|}{|P-Q|^n}\, dQ \\
	& \lesssim \MM_{\partial \Omega}(g)(P).
  \end{align*}

  To estimate $J_2$ we just need Lemma \ref{Lem:GXYFourier} $(i)$, and estimate \eqref{eq:gXQXP},
  \begin{align*}
    |J_2(g)(X,P)|
	& \lesssim \int_{\partial \Omega \cap\{|P-Q| \leq 2 \eps\}} |g(Q)| \big|\partial_{y_i} \widehat{\G}(X,Y,\tau)_{|_{Y=Q}}\big| \, dQ \\
	& \lesssim \eps \int_{\partial \Omega \cap\{|P-Q| \leq 2 \eps\}}  \frac{|g(Q)|}{|X-Q|^n}\, dQ
	  \lesssim \MM_{\partial \Omega}(g)(P).
  \end{align*}

  Thus,
  $$\sup_{X \in \g_\pm(P)} |(\mathbb{D}^\tau)^{i,j}(g)(X)|
      \lesssim (\mathbb{K}^{\pm, \tau}_*)^{i,j}(g)(P) + \MM_{\partial \Omega}(g)(P), \quad P \in \partial \Omega,$$
  and the proof of this proposition follows from Corollary \ref{Cor:bddKtau}.
\end{proof}

\begin{Rem}
  Since Lemma $\ref{Lem:GXYFourier}$ part $(i)$ and $(ii)$ yield the same estimate for the second derivative of $\partial_X ^2\widehat{\Gamma}(X, Y, \tau)$ and $\partial_{XY} ^2\widehat{\Gamma}(X, Y, \tau)$, using a similar argument as in the proof of $\mathrm{Proposition}$ $\ref{Prop:NMFtau}$, one readily sees that
  \begin{equation}\label{estim:grad of SLP}
    \Big \| (\nabla \mathbb{S}^{\tau}(g))^\pm_* \Big\|_{L^2(\partial \Omega)}
      \lesssim \|g\|_{L^2(\partial \Omega)}, \quad g \in L^2(\partial \Omega).
  \end{equation}
This will be important in the proof of the Rellich estimates in $\mathrm{Section}$ $\ref{sec:Parabolic Rellich}.$
\end{Rem}

\begin{Th}\label{boundedness Ftrans BSI of the difference}
  Assume that $A_1$ and $A_2$ satisfy the conditions \hyperref[A1]{\small{$(A1)$}} -- \hyperref[A4]{\small{$(A4)$}} and that $A_1 =A_2$ on $\partial \Omega$. Then one has $($uniformly in $\tau$$)$
 \begin{equation}\label{for ftransBSI itself}
\Big\| ((\mathbb{K}_{A_1}^\tau)-(\mathbb{K}_{A_2}^\tau))g\Big\|_{L^2(\partial \Omega)}\lesssim \Vert A_1 -A_2\Vert^{1/2}_{\infty} \| g\|_{L^2(\partial \Omega)},
\end{equation}
and
\begin{equation}\label{for ftransBSI adjoint}
\Big\| ((\mathbb{K}_{A_1}^\tau)^*-(\mathbb{K}_{A_2}^\tau)^*)g\Big\|_{L^2(\partial \Omega)}\lesssim \Vert A_1 -A_2\Vert^{1/2}_{\infty} \| g\|_{L^2(\partial \Omega)},
\end{equation}
for $g \in L^2(\partial \Omega).$
\end{Th}

\begin{proof}
We only prove \eqref{for ftransBSI adjoint}, since it will be the one we will be using later. Moreover \eqref{for ftransBSI itself} follows at once from \eqref{for ftransBSI adjoint}. The kernel of $(\mathbb{K}_{A_1}^\tau)^*-(\mathbb{K}_{A_2}^\tau)^*$ is given by $$\sum_{i,j=1}^n (a^{1}_{ij}(P) n_j(P) \partial_{x_i} \widehat{\G}_{A_1}(X,Q,\tau)_{|_{X=P}}-a^{2}_{ij}(P) n_j(P) \partial_{x_i} \widehat{\G}_{A_2}((X,Q,\tau)_{|_{X=P}}).$$

Using the fact that $a^{1}_{ij}(P)=a^{2}_{ij}(P)$, Lemma \ref{difference of fundsolutions to fourierparabolic}, Lemma \ref{Lem:GXYFourier} $(ii)$ and $(i)$, and \eqref{Holder cont of difference}, we see that the aforementioned kernel is bounded by
\begin{align*}
  |K(P,Q, \tau)|
  	& \lesssim \Vert A_{1} - A_{2}\Vert^{1/2}_{\infty} \sum_{i,j,k=1}^{n} \int_{\mathbb{R}^n}| \partial_{x_k}\partial_{z_i} \widehat{\Gamma}_{A_2}( X, Z, \tau)_{|_{X=P}}\,\partial_{z_j} \widehat{\Gamma}_{A_1}( Z, Q, \tau)||A_2 (Z) -A_1 (Z)|\, dZ\\
	& \lesssim \Vert A_{1} - A_{2}\Vert^{1/2}_{\infty} \int_{\mathbb{R}^n} \frac{1}{|P-Z|^n}\frac{dZ}{|Q-Z|^{n-1-\alpha/2}}
	\lesssim  \frac{\Vert A_{1} - A_{2}\Vert^{1/2}_{\infty}}{|P-Q|^{n-1-\alpha/2}}.
\end{align*}
Therefore since $\sup_{Q\in \partial \Omega}\int_{\partial \Omega} |K(P,Q, \tau)|\, dP$ and $\sup_{P\in \partial \Omega}\int_{\partial \Omega} |K(P,Q, \tau)|\, dQ$ are both bounded by $\Vert A_{1} - A_{2}\Vert^{1/2}_{\infty}$, Schur's Lemma implies the desired $L^2$ boundedness.
\end{proof}

\subsection{The Fourier-transformed jump relations}
Here we shall consider the jump relations that are valid for the Fourier-transformed layer potential operators. We start first with the jump relation for the Fourier-transformed DLP.
\begin{Prop}\label{Prop:Jumptau}
  Assume that \hyperref[A1]{\small{$(A1)$}} -- \hyperref[A4]{\small{$(A4)$}} hold.
  Let $g \in L^2(\partial \Omega)$.
  Then,
  $$\lim_{\substack{X \to P \\ X \in \g_{\pm}(P)}} \mathbb{D}^\tau(g)(X)
    = \mp \frac{1}{2} g(P) + \mathbb{K}^\tau(g)(P), \quad \text{a.e. } P \in \partial \Omega.$$
\end{Prop}

\begin{proof}
From Proposition \ref{Prop:NMFtau}, by standard techniques and
    density arguments it suffices to consider the case of $g \in C^\infty_c(\partial \Omega)$.

    Let $i,j=1,\dots,n$, $P \in \partial \Omega$, $X \in \g_\pm(P)$ and $\tau>0$. We have that
  \begin{align*}
    (\mathbb{D}^\tau)^{i,j}(g)(X)
       = \ & \int_{\partial \Omega}  a_{ij}(Q) n_j(Q) g(Q) \partial_{y_i} \widetilde{\G}(X,Y)_{|_{Y=Q}}\, dQ \\
         & + \int_{\partial \Omega}  a_{ij}(Q) n_j(Q) g(Q) \Big[ \int_0^\infty \partial_{y_i} \G(X,Y,\zeta)_{|_{Y=Q}}  (e^{-i \tau \zeta }-1) d\zeta \Big]\, dQ \\
       =: & A^{i,j}_1(f)(g)(X) + A^{i,j}_2(f)(g)(X).	
  \end{align*}

  As in the proof of Proposition \ref{Prop:Jump} we get
  \begin{align*}
      \lim_{\substack{X \to P \\ X \in \g_{\pm}(P)}} \sum_{i,j=1}^n A^{i,j}_1(g)(X)
	  = & \mp \frac{1}{2} g(P)
	     + \lim_{\eps \to 0} \int_{\partial \Omega \cap\{|P-Q|> \eps\}}  \partial_\nu \widetilde{\G}(P,Q) g(Q)\, dQ.
  \end{align*}
  and
  \begin{equation*}
   \lim_{\substack{X \to P \\ X \in \g_{\pm}(P)}} A^{i,j}_2(g)(X)
      = A^{i,j}_2(g)(P).
  \end{equation*}
  Moreover, it is straightforward to see that
  \begin{align*}
    & \lim_{\eps \to 0} \int_{\partial \Omega \cap\{|P-Q|> \eps\}} \partial_\nu \widetilde{\G}(P,Q) g(Q)\, dQ
	+ \sum_{i,j=1}^n A^{i,j}_2(g)(P)
	= \mathbb{K}^\tau(g)(P),
  \end{align*}
  which allows us to conclude this proof.
\end{proof}
For the normal derivative of the Fourier-transformed single-layer potential, we have the following jump relation which will be used in proof of the invertibility of boundary singular integrals.
\begin{Prop}\label{Prop:Jumptau2}
  Assume that \hyperref[A1]{\small{$(A1)$}} -- \hyperref[A4]{\small{$(A4)$}} hold.
  Let $g \in L^2(\partial \Omega)$.
  Then
  $$\partial_\nu (\mathbb{S}^\tau(g))^{\pm}(P)
    = \pm \frac{1}{2} g(P) + (\mathbb{K}^\tau)^*(g)(P), \quad \text{a.e. } P \in \partial \Omega.$$
\end{Prop}

\begin{proof}
Since the integral kernel of $\partial_\nu (\mathbb{S}^\tau)$ is similar to that of the Fourier-transformed DLP and therefore verifies the same estimates, proceeding as in the proof of Proposition \ref{Prop:Jumptau}, we will obtain the desired jump relation.
\end{proof}

As a consequence of the jump relation above, and the definition of $\nabla_T$ one sees as before that

\begin{equation*}
 \nabla_T (\mathbb{S}^\tau g)^+ = \nabla_T (\mathbb{S}^\tau g)^- .
\end{equation*}

Furthermore, taking the limit of the integral defining the Fourier-transformed SLP, using estimate \eqref{Lem:GFourier} and the Lebesgue dominated convergence theorem we obtain
\begin{equation*}
 (\mathbb{S}^\tau g)^+ = (\mathbb{S}^\tau g)^- .
\end{equation*}

\section{Parabolic Rellich estimates}\label{sec:Parabolic Rellich}

\subsection{Rellich estimates for the elliptic Fourier-transformed equation}
\quad\\
Rellich estimates are the most basic tools in proving the invertibility of boundary singular integral operators appearing in the study of elliptic and parabolic DNR problems. In this section we produce a systematic way to derive a family of  Rellich-type estimates which will be enough for proving the invertibility of our boundary singular integrals. This is done by a rather general transference method which transfers estimates valid for operators with smooth diffusion coefficients to those for operators with H\"older-continuous diffusion coefficients.\\

We begin this section by recalling that $r_0<\infty$ will denote the diameter of $\Omega $ and $\mathcal{Q}_\Omega$ is a cube that contains $\Omega$. Moreover $2\mathcal{Q}_\Omega$ will denote the cube with the same centre as $\mathcal{Q}_\Omega$ but with the double size.\\

In what follows, it is important to note that, by multiplying the Fourier-transformed equation \eqref{eq:FourierEq} by $\overline{u}(X,\tau)$ and integrating by parts one obtains the following useful identity

\begin{equation}\label{Divergence thm for Fouriertransformed eq}
  \int_{\partial \Omega} \overline{u (P, \tau)}\, \partial_{\nu_{A}} u^{\pm}(P, \tau)\,dP
    =\pm \int_{\Omega^\pm} \left( \sum_{i, j} a_{ij}(X)\, \partial_{i} \overline{u(X, \tau)}\, \partial_{j} u(X, \tau) + i \tau \,|u(X, \tau)|^2 \right)\,dX.
\end{equation}

Now we start by proving some basic estimates concerning the solutions of the Fourier-transformed equations. It is important to note that in all the estimates that are established here (until the end of this subsection), the constants hidden in the right hand sides of the estimates are all independent of parameter $\tau.$

\begin{Lem}\label{Lem:intboundGAO}
    Let $A$ be a matrix satisfying properties \hyperref[A1]{\small{$(A1)$}} -- \hyperref[A4]{\small{$(A4)$}}.
    Assume that $\widehat{\LL}_A u=0$ in $\Omega$. Then for every $\eps>0$
    $$\int_{\Omega^{\pm}} \Big( |\nabla u(Y,\tau)|^2 + |\tau| |u(Y,\tau)|^2\Big)\,dY
	\lesssim \int_{\partial \Omega} \Big(\eps |\partial_{\nu_A} u^{\pm}(P,\tau)|^2 + \frac{|u(P,\tau)|^2}{\eps} \Big)\,dP.$$
\end{Lem}

\begin{proof}
  The estimate follows easily by using \eqref{Divergence thm for Fouriertransformed eq} and Cauchy's inequality $ab \leq \eps a^2 + b^2/\eps$, $\eps>0$.
\end{proof}

\begin{Lem}\label{Lem:intboundGAO3.2}
    Let $A$ be a matrix satisfying properties \hyperref[A1]{\small{$(A1)$}} -- \hyperref[A4]{\small{$(A4)$}} and
    assume that $\widehat{\LL}_A u=0$ in $\Omega$. Then
    $$\int_{\partial \Omega}  |u(P,\tau)|^2 dP
	\lesssim \int_{\Omega^\pm} \Big( |\nabla u(X,\tau)|^2 + |u(X,\tau)|^2   \Big) \, dX.$$
\end{Lem}

\begin{proof}
 Given a $C^1$ vector field $H$ Integration by parts (i.e. divergence theorem) yields
\begin{equation*}%
\int_{\partial \Omega}  \langle H, N(P)\rangle |u(P,\tau)|^2 dP=\pm \int_{\Omega^{\pm}} |u(X,\tau)|^2 \, (\nabla\cdot H)\, dX \pm \mathrm{Re}\,\sum_{k} \int_{\Omega^{\pm}} H^{k} (X) \, \partial_{k} u (X, \tau)\, \overline{u (X,\tau)}\,  dX
\end{equation*}
 From this one deduces that
 \begin{equation}\label{divergence theorem fact}
   \int_{\partial \Omega}  |u(P,\tau)|^2 dP
	\lesssim \int_{\Omega^{\pm}} \Big( |u(X,\tau)||\nabla u(X,\tau)| + |u(X,\tau)|^2   \Big) \, dX.
 \end{equation}
Therefore \eqref{divergence theorem fact} and an application of Cauchy's inequality to the term $|u(X,\tau)||\nabla u(X,\tau)| $ yield the desired estimate.
 \end{proof}

\begin{Lem}\label{Lem:umax}
Let $A\in C^\infty(2 \mathcal{Q}_\Omega \setminus \partial \Omega)$ be a matrix verifying \eqref{gradient of A}, then for each $\delta >0$ one has

\begin{equation*}
    \int_{\Omega^{\pm}}|\nabla A(X)||\nabla u(X,\tau)|^2\, dX
        \lesssim \delta^{{n-1+\alpha_0}}  \int_{\partial \Omega} \sup_{X\in \gamma_{\delta}(P)}|\nabla u(X,\tau)|^2\, dP + \frac{1}{\delta^{(1-\alpha_0)}}\int_{\Omega^{\pm}}|\nabla u(X,\tau)|^2\, dX,
\end{equation*}
where $\gamma_{\delta}(P):= \gamma_+(P)\cap \{X \in \Omega : \mathrm{dist}(X,\partial \Omega) \leq \delta \},$ is a truncated cone with the vertex $P\in \partial \Omega$, and $\alpha_0$ is the parameter appearing in \eqref{gradient of A}.
\end{Lem}

\begin{proof}
    We confine ourselves to the proof for $\Omega^+$, as the one for $\Omega^-$ is similar. Set $\Omega^+:=\Omega$ and write it as a union of the sets $E_{1}$ and $E_{2}$ where $E_{1}:=\{X \in \Omega : \mathrm{dist}(X,\partial \Omega) \leq \delta \}$ and $E_{2}:=\{X \in \Omega : \mathrm{dist}(X,\partial \Omega) > \delta \}$. Then using \eqref{gradient of A} we have
\begin{align*}
	& \int_{\Omega}|\nabla A(X)||\nabla u(X,\tau)|^2\, dX
            = \int_{E_{1}}|\nabla A(X)||\nabla u(X,\tau)|^2\, dX
                        + \int_{E_{2}}|\nabla A(X)||\nabla u(X,\tau)|^2\, dX \\
     & \qquad \lesssim  \int_{E_{1}} \frac{|\nabla u(X,\tau)|^2}{\mathrm{dist}(X,\partial \Omega)^{1-\alpha_0}}\, dX
                        + \frac{1}{\delta^{(1-\alpha_0)}}\int_{E_{2}}|\nabla u(X,\tau)|^2\, dX \\
      & \qquad  \lesssim  \int_{\partial \Omega}\int_{\{|X-P|\leq (1+a)\mathrm{dist}(X,\partial \Omega);\,\mathrm{dist}(X,\partial \Omega) <\delta \}} \frac{|\nabla u(X,\tau)|^2}{\mathrm{dist}(X,\partial \Omega)^{1-\alpha_0}}\, dX\, dP
                        +  \frac{1}{\delta^{(1-\alpha_0)}}\int_{\Omega}|\nabla u(X,\tau)|^2\, dX\\
      & \qquad \lesssim   \int_{\partial \Omega}  \sup_{X\in \gamma_{\delta}(P)}|\nabla u(X,\tau)|^2\,\Big( \int_{|X-P|\lesssim \delta}\frac{dX}{|X-P|^{1-\alpha_0}}\Big)\, dP + \frac{1}{\delta^{(1-\alpha_0)}}\int_{\Omega}|\nabla u(X,\tau)|^2 \, dX\\
      & \qquad \lesssim \delta^{n-1+\alpha_0} \int_{\partial \Omega}\sup_{X\in \gamma_{\delta}(P)}|\nabla u(X,\tau)|^2\, dP
                        + \frac{1}{\delta^{(1-\alpha_0)}}\int_{\Omega}|\nabla u(X,\tau)|^2 \, dX.
    \end{align*}
\end{proof}

For operators with smooth diffusion coefficients one has the following Rellich-type estimate.
\begin{Prop}\label{Lem:Rellsmooth}
  Let $\tilde{A}$ be as in $\mathrm{Lemma}$ $\ref{Lem:Atilde}$ and assume that $\widehat{\LL}_{\tilde{A}} u=0$ in $\Omega$. Then,
  \begin{align}\label{eq:Rell2}
    \int_{\partial \Omega} \Big(|\nabla u^{\pm}(P,\tau)|^2 + |\tau| |u(P,\tau)|^2 \Big)\, dP
      \lesssim & \int_{\partial \Omega} \Big(|\partial_{\nu_A} u^{\pm}(P,\tau)|^2 + |u(P,\tau)|^2\Big)\, dP
  \end{align}
  and
  \begin{align}\label{eq:Rell1}
    \int_{\partial \Omega} |\nabla u^{\pm}(P,\tau)|^2 dP
      \lesssim & \int_{\partial \Omega} \Big(|\nabla_T u(P,\tau)|^2 + |\tau| |u(P,\tau)|^2 + \frac{|u(P,\tau)|^2}{\varepsilon^2}\Big)\, dP\\&+ \delta^{n-1+\alpha_0} \int_{\partial \Omega}\sup_{X\in \gamma_{\delta}(P)}|\nabla u(X,\tau)|^2\, dP , \nonumber
  \end{align}
  for any $\delta>0$ and $\eps \in (0,\eps_0(\delta))$.
\end{Prop}

\begin{proof}
  First of all we observe that \eqref{Divergence thm for Fouriertransformed eq} yields

\begin{equation}\label{eq:baby rellich 1}
 \int_{\Omega^{\pm}}\Big(|\nabla u(X,\tau)|^2 + |\tau| \, |u(X,\tau)|^2\Big) \, dX\lesssim \int_{\partial \Omega} |\partial_{\nu_A} u^{\pm}(P,\tau)|\, |u(P,\tau)| \, dP.
\end{equation}

Applying Cauchy's inequality to the left hand side of \eqref{eq:baby rellich 1} we also get
\begin{equation}\label{Divergence and Schwarz}
  \int_{\Omega^{\pm}}  |\tau|^{\frac{1}{2}}  \, |\nabla u(X,\tau)| \, |u(X,\tau)| \, dX \lesssim \int_{\partial \Omega}\, |\partial_{\nu_A} u^{\pm}(P,\tau)|\, |u(P,\tau)| \, dP.
\end{equation}

Now, given a $C^1$ vector field $H$ and using the Fourier-transformed equation \eqref{eq:FourierEq}, an application of the divergence theorem yields
\begin{align}\label{eq:divergence theorem fourier transform eq 3}
   & \sum_{i,j=1}^n \int_{\partial \Omega} \langle N(P), H(P)\rangle \,\tilde{a}_{ij}(P)\, \partial_i u^{\pm}(P,\tau)  \,\partial_j \bar{u}^{\pm}(P,\tau) \, dP \\
   & \qquad = 2\,\mathrm{Re}\, \sum_{i,j,k=1}^n  \int_{\Omega^{\pm}} H_k(X)\Big[\partial_i \Big(\tilde{a}_{ij}(X) \,\partial_k u(X,\tau) \, \partial_j \bar{u}(X,\tau) \Big) -i \tau\, u(X,\tau) \,\partial_k \, u(X,\tau)\Big]\, dX \nonumber \\
   & \qquad \qquad +\sum_{i,j=1}^n \int_{\Omega^{\pm}} (\nabla\cdot H)\, \tilde{a}_{ij}(X) \, \partial_i u(X,\tau) \, \partial_j \bar{u}(X,\tau) \Big] \, dX, \nonumber
   \end{align}
where $N$ denotes, as before, the unit normal vector to the boundary. Applying first the divergence theorem and then \eqref{eq:baby rellich 1} to the first term on the right hand side of \eqref{eq:divergence theorem fourier transform eq 3}, we have

\begin{align}\label{eq:baby rellich 2}
& |2\,\mathrm{Re}\, \sum_{i,j,k=1}^n  \int_{\Omega^{\pm}} H_k(X)\partial_i \Big(\tilde{a}_{ij}(X) \,\partial_k u(X,\tau) \, \partial_j \bar{u}(X,\tau) \Big)\, dX|\\
& \qquad = |2\,\mathrm{Re}\, \sum_{i,j,k=1}^n  \int_{\Omega^{\pm}} \partial_i H_k(X) \, \tilde{a}_{ij}(X) \,\partial_k u(X,\tau) \, \partial_j \bar{u}(X,\tau) \, dX \nonumber\\
& \qquad \qquad + \sum_{i,j,k=1}^n \int_{\partial \Omega}  H_k(P)\, n_i (P) \,\tilde{a}_{ij}(P) \,\partial_j \bar{u}(P,\tau)\, \partial_k u(P,\tau) \, dP | \nonumber\\
& \qquad \lesssim
\int_{\partial \Omega} \Big( |\partial_{\nu_A} u^{\pm}(P,\tau)|\, |\nabla u(P,\tau)|+|\partial_{\nu_A} u^{\pm}(P,\tau)|\,|u(P,\tau)|\Big)\, dP. \nonumber
\end{align}

For the second and the third terms of \eqref{eq:divergence theorem fourier transform eq 3}, we use \eqref{Divergence and Schwarz} and \eqref{eq:baby rellich 1} and Cauchy's inequality which yield
\begin{align}\label{eq:baby rellich 3}
\Big|i \tau\, \int_{\Omega^{\pm}} u(X,\tau) \,\partial_k \, u(X,\tau)\, dX +\sum_{i,j=1}^n \int_{\Omega^{\pm}} (\nabla\cdot H)\, \tilde{a}_{ij}(X) \, \partial_i u(X,\tau) \, \partial_j \bar{u}(X,\tau) \, dX\Big|  \\
\lesssim \int_{\partial \Omega} \Big( |\partial_{\nu_A} u^{\pm}(P,\tau)|^2 + |u(P,\tau)|^2 + |\tau|^{1/2}\, |u|\,|\partial_{\nu_A} u^{\pm}(P,\tau)|\Big)\, dP. \nonumber
\end{align}

Observe that we can pick the vector field $H$ such that $\langle H(P), N(P) \rangle \gtrsim  1$ independent of $P \in \partial \Omega$. Hence applying the ellipticity condition (A3) to the left hand side of \eqref{eq:divergence theorem fourier transform eq 3}, the estimates \eqref{eq:baby rellich 2} and \eqref{eq:baby rellich 3} to the right hand side of \eqref{eq:divergence theorem fourier transform eq 3} and also Cauchy's inequality, we will finally obtain
  \begin{equation*}%
    \int_{\partial \Omega} \Big(|\nabla u^{\pm}(P,\tau)|^2 + |\tau| \, |u(P,\tau)|^2 \Big)\, dP
      \lesssim \int_{\partial \Omega} \Big(|\partial_{\nu_A} u^{\pm}(P,\tau)|^2 + |u(P,\tau)|^2\Big)\, dP,
    \end{equation*}
which proves \eqref{eq:Rell2}.\\

The proof of \eqref{eq:Rell1} is a bit more involved and goes as follows. Given a $C^1$ vector field $H$ and using \eqref{eq:FourierEq}, an elementary calculation yields
   $$\sum_{i,j,k=1}^n \partial_k \,[(H_k \tilde{a}_{ij} - H_i \tilde{a}_{kj} - H_j \tilde{a}_{ik})\, \partial_i u  \,\partial_j \bar{u}]
        = \sum_{i,j=1}^n b_{ij} \, \partial_i u \, \partial_j \bar{u}-2 \tau \sum_{i=1}^n \mathrm{Im} \, (H_i \,\partial_i \bar{u} \, u),$$
         where $b_{ij}= \partial_k (H_k \tilde{a}_{ij} - H_i \tilde{a}_{kj} - H_j \tilde{a}_{ik}).$
Therefore, divergence theorem implies that
   \begin{align}\label{eq:divergence theorem fourier transform eq}
   & \sum_{i,j,k=1}^n \int_{\partial \Omega} n_k(P)\,\Big(H_k(P) \tilde{a}_{ij}(P) - H_i(P) \tilde{a}_{kj}(P) - H_j(P) \tilde{a}_{ik}(P)\Big)\, \partial_i u^{\pm}(P,\tau)  \,\partial_j \bar{u}^{\pm}(P,\tau) \, dP \\
   & \qquad   = \pm  \int_{\Omega^{\pm}} \Big[2 \tau\, \sum_{i=1}^n \mathrm{Im} \, \Big(H_i(X) \,\partial_i \bar{u}(X,\tau) \, u(X,\tau)\Big)
            + \sum_{i,j=1}^n b_{ij}(X) \, \partial_i u(X,\tau) \, \partial_j \bar{u}(X,\tau) \Big] \, dX. \nonumber
   \end{align}
Now since
\begin{align*}
& \sum_{i,j,k=1}^n \int_{\partial \Omega} n_k(P)\,\Big(H_k(P) \tilde{a}_{ij}(P) - H_i(P) \tilde{a}_{kj}(P) - H_j(P) \tilde{a}_{ik}(P)\Big)\, \partial_i u^{\pm}(P,\tau)  \,\partial_j \bar{u}^{\pm}(P,\tau) \, dP \\
& \qquad = \sum_{i,j,k=1}^n \Big(-\int_{\partial \Omega} n_k(P)\, \Big[ H_i(P)\,  \tilde{a}_{kj}(P)-H_k(P)\,  \tilde{a}_{ij}(P) \Big] \, \partial_i u^{\pm}(P,\tau)  \,\partial_j \bar{u}^{\pm}(P,\tau) \, dP \\
& \qquad \quad - \int_{\partial \Omega} n_k(P)\, \Big[H_j(P)\,  \tilde{a}_{ik}(P) - H_k(P)\,  \tilde{a}_{ij}(P)\Big] \, \partial_i u^{\pm}(P,\tau)  \,\partial_j \bar{u}^{\pm}(P,\tau) \, dP \\
& \qquad \quad -  \int_{\partial \Omega} n_k(P)\, H_k(P)\,  \tilde{a}_{ij}(P)\, \partial_i u^{\pm}(P,\tau)  \,\partial_j \bar{u}^{\pm}(P,\tau) \, dP\Big),
\end{align*}

equation \eqref{eq:divergence theorem fourier transform eq} yields that
\begin{align}\label{eq:a bloody mess}
& \sum_{i,j,k=1}^n \int_{\partial \Omega} n_k(P)\, H_k(P)\,  \tilde{a}_{ij}(P)\, \partial_i u^{\pm}(P,\tau)  \,\partial_j \bar{u}^{\pm}(P,\tau) \, dP \\
    & \quad = -\sum_{i,j,k=1}^n \int_{\partial \Omega} n_k(P)\, \Big[ H_i(P)\,  a_{kj}(P)-H_k(P)\,  \tilde{a}_{ij}(P) \Big] \, \partial_i u^{\pm}(P,\tau)  \,\partial_j \bar{u}^{\pm}(P,\tau) \, dP \nonumber\\
    & \quad \quad - \sum_{i,j,k=1}^n \int_{\partial \Omega} n_k(P)\, \Big[H_j(P)\,  \tilde{a}_{ik}(P) - H_k(P)\,  \tilde{a}_{ij}(P)\Big] \, \partial_i u^{\pm}(P,\tau)  \,\partial_j \bar{u}^{\pm}(P,\tau) \, dP \nonumber\\
    & \quad \quad \mp  \int_{\Omega^{\pm}} \Big[2 \tau\, \sum_{i=1}^n \mathrm{Im} \, \Big(H_i(X) \,\partial_i \bar{u}(X,\tau) \, u(X,\tau)\Big)
            + \sum_{i,j=1}^n b_{ij}(X) \, \partial_i u(X,\tau) \, \partial_j \bar{u}(X,\tau) \Big] \, dX. \nonumber
\end{align}

Now we observe that for fixed $j$ (resp. fixed $i$) the vector $ n_k(P)[H_i(P)\,  \tilde{a}_{kj}(P)-H_k(P)\,  \tilde{a}_{ij}(P)]$  (resp. $n_k(P) [H_j(P)\,  \tilde{a}_{ik}(P) - H_k(P)\,  \tilde{a}_{ij}(P)]$ is orthogonal to the normal vector $N$. Therefore the first two terms on the right hand side of \eqref{eq:a bloody mess} can be estimated by $\int_{\partial \Omega} |\nabla  u^{\pm}(P,\tau)|  \,|\nabla_{T} u(P,\tau)| \, dP .$  Moreover, if we once again use the fact that $\langle H(P), N(P) \rangle \gtrsim  1,$ then using the ellipticity condition (A3) and \eqref{eq:a bloody mess} we obtain
\begin{align}\label{eq:a bloody mess 2}
& \int_{\partial \Omega} |\nabla u^{\pm}(P,\tau)|^2 dP
    \lesssim \int_{\partial \Omega} |\nabla  u^{\pm}(P,\tau)|  \,|\nabla_{T} u(P,\tau)| \, dP  \\
    & \qquad \qquad +  \int_{\Omega^{\pm}} 2 |\tau| \, |\nabla u(X,\tau)| \, |u(X,\tau)| \, dX+ \int_{\Omega^{\pm}}|\nabla \tilde{A}(X)| \,|\nabla u(X,\tau)|^2\, dX. \nonumber
\end{align}

Now since \eqref{Divergence and Schwarz} yields that

\begin{equation*}
  \int_{\Omega^{\pm}}  |\tau| \, |\nabla u(X,\tau)| \, |u(X,\tau)|\, dX \lesssim \int_{\partial \Omega}|\tau|^{\frac{1}{2}}\, |\nabla u^{\pm}(P,\tau)|\, |u(P,\tau)| \, dP,
\end{equation*}

estimate \eqref{eq:a bloody mess 2} and Cauchy's inequality imply that

  \begin{equation}\label{eq:Rellich with grad A 2}
    \int_{\partial \Omega} |\nabla u^{\pm}(P,\tau)|^2 dP
      \lesssim  \int_{\partial \Omega} \Big(|\nabla_T u(P,\tau)|^2 + |\tau| |u(P,\tau)|^2 \Big)\, dP+\int_{\Omega^{\pm}}|\nabla \tilde{A}(X)| \,|\nabla u(X,\tau)|^2\, dX.
  \end{equation}
As the reader can see, the estimate above contains an undesired term $\int_{\Omega^{\pm}}|\nabla \tilde{A}(X)| \,|\nabla u(X,\tau)|^2\, dX$ which will be removed in what follows. To this end since $\tilde{A},$ verifies all the assumptions of Lemmas \ref{Lem:intboundGAO} and \ref{Lem:umax}, using these lemmas we have
  \eqref{eq:Rellich with grad A 2}
  \begin{align*}
    & \int_{\partial \Omega} |\nabla u^{\pm}(P,\tau)|^2 dP
      \lesssim  \int_{\partial \Omega} \Big(|\nabla_T u(P,\tau)|^2 + |\tau| |u(P,\tau)|^2 \Big)\, dP\\
    & \qquad \qquad  +  \delta^{n-1+\alpha_0} \int_{\partial \Omega} \sup_{X\in \gamma_{\delta}(P)}|\nabla u(X,\tau)|^2\, dP
                        + \frac{1}{\delta^{(1-\alpha_0)}}\int_{\Omega^{\pm}}|\nabla u(X,\tau)|^2\, dX \\
    & \qquad \lesssim  \int_{\partial \Omega} \Big(|\nabla_T u(P,\tau)|^2 + |\tau| |u(P,\tau)|^2 \Big)\, dP
        + \delta^{n-1+\alpha_0} \int_{\partial \Omega} \sup_{X\in \gamma_{\delta}(P)}|\nabla u(X,\tau)|^2\, dP  \\
    & \qquad \qquad  + \frac{1}{\delta^{(1-\alpha_0)}}
    \int_{\partial \Omega} \Big(\eps |\nabla u^{\pm}(P,\tau)|^2 + \frac{|u(P,\tau)|^2}{\eps} \Big)\,dP,
  \end{align*}
  where we have used the fact that $|\partial_{\nu_{\tilde{A}}} u^{\pm}| \lesssim |\nabla u^{\pm}|.$ Next picking $\varepsilon$ sufficiently small, say $\eps<\eps_0(\delta)<1$, we can absorb the term $ \frac{\eps }{\delta^{(1-\alpha_0)}}
    \int_{\partial \Omega} |\nabla u^{\pm}(P,\tau)|^2 \,dP $ in the left hand side of the estimate above. This establishes \eqref{eq:Rell1}.
\end{proof}

Now we are ready to state and prove our first Rellich-type estimates for the Fourier-transformed equation associated to the diffusion matrix $\tilde{A} + A-A^{(r)}$.

\begin{Th}\label{Thm:Transferencetau}
    Let $A$ be a matrix satisfying properties \hyperref[A1]{\small{$(A1)$}} -- \hyperref[A4]{\small{$(A4)$}} and assume that $\tilde{A}$ and $A^{(r)}$, with $0<r<1$, are defined as is $\mathrm{Lemmas}$ $\ref{Lem:Atilde}$ and $\ref{Lem:Ar}.$
    Set $B:= \tilde{A} + A-A^{(r)}$ and assume that $\widehat{\LL}_B u=0$ in $\Omega$.
    Then for all $\delta\in (0, \delta_0 (r))$ and $\eps' \in (0,\eps_0'(\delta,r))$, we have the following Rellich-type estimates
    \begin{align}\label{eq:RellB1}
    \int_{\partial \Omega} \Big(|\nabla u^{\pm}(P,\tau)|^2 + |\tau| |u(P,\tau)|^2 \Big)\, dP
      \lesssim & \int_{\partial \Omega} \Big(|\partial_{\nu_B} u^{\pm}(P,\tau)|^2 + |u(P,\tau)|^2\Big)\, dP ,
    \end{align}
 \begin{align}\label{eq:RellB2}
  \int_{\partial \Omega} |\nabla u^{\pm}(P,\tau)|^2 dP
      &\lesssim \int_{\partial \Omega} \Big(|\nabla_T u(P,\tau)|^2 + |\tau| |u(P,\tau)|^2 + \frac{|u(P,\tau)|^2}{\eps'}\Big)\, dP\\
      & \quad + \delta^{n-1+\alpha_0} \int_{\partial \Omega}|(\nabla u)^+_* (P,\tau)|^2\, dP\nonumber,
  \end{align}

   \begin{equation}\label{eq:RellB3}
  \int_{\partial \Omega} |\partial_{\nu_B} u^-(P,\tau)|^2 dP
  	\lesssim \int_{\partial \Omega} \Big(\frac{|u(P,\tau)|^2}{\eps'} + |\partial_{\nu_B} u^+(P,\tau)|^2 + \delta^{n-1+\alpha_0}|(\nabla u)^+_*(P,\tau)|^2\Big)\, dP.
 \end{equation}
 where {$(\nabla u)^+_*$} denotes the non-tangential maximal function of $\nabla u.$

\end{Th}

\begin{proof}
By definition, we have that $B:= \tilde{A} + A-A^{(r)}$ and $B$ satisfies properties \hyperref[A1]{\small{$(A1)$}} -- \hyperref[A4]{\small{$(A4)$}}. The idea is to transfer the Rellich inequalities \eqref{eq:Rell2} and \eqref{eq:Rell1}, which are valid for the smooth matrix $\tilde{A}$, to the corresponding ones for $B$.

Let $m_{ij}$ be the matrix associated with $ A-A^{(r)}$, and $b_{ij}$ be that of $B$'s. Observe that by the construction of $A^{(r)}$
 \begin{equation*}
    m_{ij}(X)=0 \quad \mathrm{when}\quad  \mathrm{dist}(X,\partial \Omega) < r_0 r,
  \end{equation*}
where $r_0$ is the diameter of $\Omega$. Since $\widehat{\LL}_ B u=0$ in $\Omega$, it is clear that
  \begin{align*}
    \widehat{\LL}_{\tilde{A}} u
	= - i \tau u - \sum_{i,j=1}^n \partial_i \big[ \big(b_{ij}-m_{ij}\big)\, \partial_j u \big]
	= \sum_{i,j=1}^n \partial_i(m_{ij}\,\partial_j u )
	\quad \text{in } \Omega.
  \end{align*}
Hence, we can decompose $u=w+v;$
  where $\widehat{\LL}_{\tilde{A}} w=0$ in $\Omega$, and
  $$v(X,\tau)
	:= -\sum_{i,j=1}^n \int_\Omega \partial_{y_i}\widehat{\G}_{\tilde{A}}(X,Y,\tau)\, m_{ij}(Y)\partial_{y_j} u(Y,\tau)\, dY=: \sum_{i,j=1}^n v_{ij}(X,\tau) ,$$
  where $\G_{\tilde{A}}$ denotes the fundamental solution $\G$ considered in Section \ref{sec:FS},
  but here we emphasise its dependence on $\tilde{A}$.

  First we prove \eqref{eq:RellB1}. We observe that applying \eqref{eq:Rell2} to $w$ yields
  \begin{align}\label{starting rellich}
    & \int_{\partial \Omega} \Big(|\nabla u^{\pm}(P,\tau)|^2 + |\tau||u(P,\tau)|^2 \Big)\, dP  \\
    & \qquad \lesssim \int_{\partial \Omega} \Big(|\nabla w^{\pm}(P,\tau)|^2 + |\tau| |w(P,\tau)|^2 \Big)\,dP
		    + \int_{\partial \Omega} \Big(|\nabla v^{\pm}(P,\tau)|^2 + |\tau| |v(P,\tau)|^2 \Big)\, dP \nonumber \\
    & \qquad \lesssim \int_{\partial \Omega} \Big(|\partial_{\nu_{\tilde{A}}} w^{\pm}(P,\tau)|^2 + |w(P,\tau)|^2\Big)\,dP
		    + \int_{\partial \Omega} \Big(|\nabla v^{\pm}(P,\tau)|^2 + |\tau| |v(P,\tau)|^2 \Big)\,dP \nonumber \\
    & \qquad \lesssim \int_{\partial \Omega} \Big(|\partial_{\nu_{B}} u^{\pm}(P,\tau)|^2 + |u(P,\tau)|^2\Big)\,dP
		    + \int_{\partial \Omega} \Big(|\nabla v^{\pm}(P,\tau)|^2 + |\tau||v(P,\tau)|^2 + |v(P,\tau)|^2\Big)\,dP,\nonumber
  \end{align}
  where we have used that $\tilde{A}=B$ on $\partial \Omega$, and
  $|\partial_{\nu_{\tilde{A}}} v^\pm| \lesssim |\nabla v^\pm|$. It remains to control the integrals involving the function $v$.

 Applying Lemma \ref{Lem:GXYFourier} $(i)$, $(ii)$ and $(iii)$ (taken with $q=1/2$), we have that
  \begin{align*}
      |\nabla v_{ij}^\pm(P,\tau)| + |\tau|^{1/2} |v_{ij}(P,\tau)| + |v_{ij}(P,\tau)|
	& \lesssim \int_{\Omega\cap \{|Y-P| \geq r_0 r\}} \frac{|\nabla u(Y,\tau)|}{|Y-P|^n} \, dY \\
	& \lesssim \frac{1}{r^n} \int_{\Omega} |\nabla u(Y,\tau)|\, dY, \quad P \in \partial  \Omega.
  \end{align*}
  Thus,
  \begin{align}\label{estim:vijs great estim}
    \int_{\partial \Omega} \Big(|\nabla v_{ij}^\pm(P,\tau)|^2 + |\tau| |v_{ij}(P,\tau)|^2 + |v_{ij}(P,\tau)|^2\Big)\,dP
      & \lesssim \frac{1}{r^{2n}} \int_{\Omega} |\nabla u(Y,\tau)|^2 \, dY.
  \end{align}

Now applying Lemma \ref{Lem:intboundGAO} to the term $\int_{\Omega} |\nabla u(Y,\tau)|^2 \, dY$ in the above estimate and inserting the resulting estimate in \eqref{starting rellich}, we arrive at
 \begin{equation*}
     \int_{\partial \Omega} \Big(|\nabla u^\pm(P,\tau)|^2 + |\tau ||u(P,\tau)|^2 \Big)\, dP \lesssim \Big(1 + \frac{1}{r^{2n}} \Big) \int_{\partial \Omega} \Big(|\partial_{\nu_{B}} u^\pm (P,\tau)|^2 + |u(P,\tau)|^2\Big)\, dP.
  \end{equation*}
  This proves \eqref{eq:RellB1}.

  We proceed with the proof of \eqref{eq:RellB2}. We observe that applying \eqref{eq:Rell1} to $w$ yields
\begin{align}\label{ending rellich}
& \int_{\partial \Omega}|\nabla u^{\pm}(P,\tau)|^2 \, dP
    	\lesssim \int_{\partial \Omega} |\nabla w^{\pm}(P,\tau)|^2\,dP + \int_{\partial \Omega} |\nabla v^{\pm}(P,\tau)|^2 \, dP\\
& \qquad  \lesssim \int_{\partial \Omega} \Big(|\nabla_T w(P,\tau)|^2 + |\tau| |w(P,\tau)|^2 + \frac{|w(P,\tau)|^2}{\eps^2}\Big)\, dP \nonumber \\
& \qquad \quad + \delta^{n-1+\alpha_0} \int_{\partial \Omega}\sup_{X\in \gamma_{\delta}(P)}|\nabla w(X,\tau)|^2\, dP+  \int_{\partial \Omega} |\nabla v^{\pm}(P,\tau)|^2 \, dP \nonumber\\
&\qquad \lesssim \int_{\partial \Omega} \Big(|\nabla_T u(P,\tau)|^2 + |\tau| |u(P,\tau)|^2 + \frac{|u(P,\tau)|^2}{\eps^2}\Big)\, dP+ \delta^{n-1+\alpha_0} \int_{\partial \Omega}|(\nabla u)^+_* (P,\tau)|^2\, dP \nonumber\\
&\qquad \quad  +\int_{\partial \Omega} \Big(|\nabla v^{\pm}(P,\tau)|^2 +|\nabla_T v(P,\tau)|^2 + |\tau| |v(P,\tau)|^2 +
\frac{|v(P,\tau)|^2}{\eps^2}\Big)\, dP \nonumber\\
&\qquad \quad + \delta^{n-1+\alpha_0} \int_{\partial \Omega}\sup_{X\in \gamma_{\delta}(P)}|\nabla v(X,\tau)|^2\, dP. \nonumber
  \end{align}
Now using \eqref{estim:vijs great estim} and Lemma \ref{Lem:intboundGAO}, we once again see that

\begin{align}\label{estim: rellich fur v2}
& \int_{\partial \Omega} \Big(|\nabla v^{\pm}(P,\tau)|^2 +|\nabla_T v(P,\tau)|^2 + |\tau| |v(P,\tau)|^2 + \frac{|v(P,\tau)|^2}{\eps^2}\Big)\, dP\\
& \qquad \lesssim \int_{\partial \Omega} \Big(|\nabla v^{\pm}(P,\tau)|^2 + |\tau| |v(P,\tau)|^2 + \frac{|v(P,\tau)|^2}{\eps^2}\Big)\, dP
\lesssim \frac{1}{\eps^2 r^{2n}} \int_{\Omega} |\nabla u(Y,\tau)|^2 \, dY \nonumber \\
& \qquad \lesssim \frac{1}{\eps^2 r^{2n}}  \int_{\partial \Omega} \Big(\eps' |\nabla u^\pm (P,\tau)|^2 + \frac{|u(P,\tau)|^2}{\eps'}\Big)\, dP. \nonumber
\end{align}
Moreover using the support property of $m_{ij}$ one has
\begin{equation}\label{nontangmax for vij}
   |\nabla v_{ij}(X,\tau)| \lesssim \int_{\Omega\cap \{Y;\,  \mathrm{dist}(Y,\partial \Omega) \geq r_0 r\}} \frac{|\nabla u(Y,\tau)|}{|X-Y|^n} \, dY.
\end{equation}
If now $P\in\partial \Omega$,  $X\in \gamma_{\delta}(P)$ and $\mathrm{dist}(Y,\partial \Omega) \geq r_0 r$, then triangle inequality yields that $$|X-Y|\geq r_0 r -(1+a) \delta\geq C>0, $$ provided that we choose $\delta <\frac{r_0 r}{1+a}=:\delta_0 (r)$. Therefore choosing $\delta<\delta_0 (r)$, estimate \eqref{nontangmax for vij}, Lemma \ref{Lem:intboundGAO} and the Cauchy-Schwarz inequality yield
\begin{equation}\label{nontangmax for v}
   \int_{\partial \Omega}\sup_{X\in \gamma_{\delta}(P)}|\nabla v(X,\tau)|^2\, dP \lesssim \int_{\Omega} |\nabla u(Y,\tau)|^2 \, dY \lesssim  \int_{\partial \Omega} \Big(\eps' |\nabla u^\pm (P,\tau)|^2 + \frac{|u(P,\tau)|^2}{\eps'}\Big)\, dP.
\end{equation}
Hence choosing $\eps'$  small enough, say $\eps'<\eps_0'(\delta,r,\eps)$, we can absorb the term $\int_{\partial \Omega} \eps' |\nabla u^\pm (P,\tau)|^2\, dP$ on the right hand side of \eqref{estim: rellich fur v2} and \eqref{nontangmax for v} in the left hand side of \eqref{ending rellich}. Thus, \eqref{estim: rellich fur v2}, \eqref{nontangmax for v} and \eqref{ending rellich} together yield
 \begin{align*}
\int_{\partial \Omega} |\nabla u^{\pm}(P,\tau)|^2 dP
      & \lesssim \int_{\partial \Omega} \Big(|\nabla_T u(P,\tau)|^2 + |\tau| |u(P,\tau)|^2 + \frac{|u(P,\tau)|^2}{\eps'}\Big)\, dP \\
      & \quad + \delta^{n-1+\alpha_0} \int_{\partial \Omega}|(\nabla u)^+_* (P,\tau)|^2\, dP.
\end{align*}

Finally, \eqref{eq:RellB3} follows by using the fact that $|\partial_{\nu_{B}} u^{\pm}| \lesssim |\nabla u^{\pm}|$ and applying first \eqref{eq:RellB2} and then \eqref{eq:RellB1}.
\end{proof}

\subsection{Rellich estimates for the parabolic equation}\label{subsec:rellich for parab}
Once again we start by some basic estimates for operators with smooth diffusion coefficients. However, for a H\"older-continuous $A$ we have the following lemma
\begin{Lem}\label{Lem:intboundGAO4.1}
    Let $A$ be a matrix satisfying properties \hyperref[A1]{\small{$(A1)$}} -- \hyperref[A4]{\small{$(A4)$}}.
    Assume that $\LL_A u=0$ in $\Omega \times (0,T)$ and $u(X,0)=0$. Then for every $\eps>0$
    $$\int_0^T \int_{\Omega^{\pm}} |\nabla u(X,t)|^2 \, dX\, dt
	\lesssim \int_0^T \int_{\partial \Omega} \Big(\eps |\nabla u^{\pm}(P,t)|^2 + \frac{|u(P,t)|^2}{\eps} \Big)\,dP \,dt.$$
\end{Lem}

\begin{proof}
 We multiply equation $\LL_A u=0$ by $u$ and integrate by parts. Thereafter we use the ellipticity assumption ($A_3$) and the assumption $u(X,0)=0,$ which yield

  \begin{align}\label{estim: energy related}
    \pm \int_{0}^{T}\int_{\partial \Omega} u\, \partial_{\nu_{A}} u^{\pm}\, dP \, dt
    &=\int_{0}^{T}\int_{\Omega^\pm} \left( \sum_{ij} a_{ij}\, \partial_{i} u\, \partial_{j} u +  u\, \partial_{t} u\right)\, dX \, dt \\
    &\geq \mu \int_{0}^{T}\int_{\Omega^\pm} |\nabla u|^2 \, dX \, dt + \frac{1}{2} \int_{\Omega^\pm} |u (X, T) |^2\, dX. \nonumber
  \end{align}
Now, neglecting the last term in the right hand side of \eqref{estim: energy related} and applying Cauchy's inequality we obtain the estimate claimed in the lemma.
\end{proof}

\begin{Lem}\label{Lem:umaxparab}
Let $A\in C^\infty(2 \mathcal{Q}_\Omega \setminus \partial \Omega)$ be a matrix verifying \eqref{gradient of A}, then for each $\delta >0$ one has
\begin{align*}
& \int_{0}^{T}\int_{\Omega^{\pm}}|\nabla A(X)||\nabla u(X,t)|^2\, dX\, dt
 \lesssim \delta^{n-1+\alpha_0}  \int_{0}^{T}\int_{\partial \Omega} \sup_{X\in \gamma_{\delta}(P)}|\nabla u(X,t)|^2\, dP \, dt \\
 & \qquad \qquad + \frac{1}{\delta^{(1-\alpha_0)}}\int_{0}^{T}\int_{\Omega^{\pm}}|\nabla u(X,t)|^2\, dX \, dt,
\end{align*}
where $\gamma_{\delta}(P)$ was defined in $\mathrm{Lemma}$ $\ref{Lem:umax}$ and $\alpha_0$ is the parameter appearing in \eqref{gradient of A}.
\end{Lem}

\begin{proof}
The proof is practically identical to that of Lemma \ref{Lem:umax}, hence omitted.
\end{proof}
To handle the Rellich estimates involving fractional derivatives of the solution of the parabolic equation, we would also need the following lemma
\begin{Lem}\label{Lem:Dt12}
    Let $A$ be a matrix satisfying properties \hyperref[A1]{\small{$(A1)$}} -- \hyperref[A4]{\small{$(A4)$}}.
    Assume that $\LL_A u=0$ in $\Omega \times (0,T)$ and $u(X,0)=0$. Assume that $H$ is a $C^1$ vector field on $\partial \Omega$ with $\langle H, N\rangle >0$. Then, for every $\eps>0$, we have
    \begin{equation}\label{eq:PITA}
     \int_0^T \int_{\Omega^\pm} \langle H, \nabla u(X,t)\rangle\, \partial_t u(X,t)\, dX\, dt
      \lesssim \int_0^T \int_{\partial \Omega} \Big(\eps |\nabla u^{\pm}(P,t)|^2 + \frac{|D_t^{1/2}u(P,t)|^2}{\eps} \Big)\,dP\, dt.
    \end{equation}
\end{Lem}
\begin{proof}
    Let $v$ and $w$ be solutions of $\LL_A u=0$ in $\Omega \times (0,T)$ and $v(X,0)=w(X,0)=0$. Then, using the equation for $v$ and $w$ and the divergence theorem yield
    \begin{align}\label{eq:WZ}
        & \sum_{i,j=1}^n \int_0^T \int_{\Omega^\pm} a_{ij}\, \partial_i v(X,t)\, \partial_j w(X,t)\, dX \, dt
               + \int_0^T \int_{\Omega^\pm} v(X,t)\, \partial_t w(X,t)\, dX\, dt\\
        & \qquad \qquad = \pm \int_0^T \int_{\partial \Omega} v(P,t)\, \partial_{\nu_A} w^\pm(P,t)\, dP\, dt . \nonumber
    \end{align}
    Now if we take $v=w=D_t^{1/4}u$ (noticing that $D_t^{1/4}u$ is also a solution of the equation) in \eqref{eq:WZ} and use the ellipticity condition (A3), Lemma \ref{Browns fractional} and the fact that $v(X,0)=w(X,0)=0$ (to be able to eventually omit the second term on the left hand side), we obtain
    \begin{align}\label{eq:WTF1}
        \int_0^T \int_{\Omega^\pm} |D_t^{1/4} \nabla u(X,t)|^2 \,dX\, dt
            \leq & \int_0^T \int_{\partial \Omega} D_t^{1/4} \Big( \partial_{\nu_A} u^\pm(P,t)\Big) D_t^{1/4}\Big(u(P,t)\Big)\, dP\, dt \\
            \leq & \int_{\partial \Omega} \Big( \int_0^T | \partial_{\nu_A} u^\pm(P,t)|^2 \,dt \Big)^{1/2}\Big( \int_0^T | D_t^{1/2} u(P,t)|^2 dt \Big)^{1/2}\, dP. \nonumber
    \end{align}

Using \eqref{defn:fractional derivative} we observe that $\partial_t I_{1/4}=D_t^{3/4}$. Furthermore, if we take $v=D_t^{3/4} u$ and $w=I_{1/4}u$ we observe that these are still solutions of the equation and $I_{1/4}u(X,0)=0$. Therefore if we apply \eqref{eq:WZ} to these solutions and use the ellipticity again we get
\begin{equation}\label{eq:WTF2}
\int_0^T \int_{\Omega^\pm} |D_t^{3/4} u(X,t)|^2 dX dt
    \lesssim \int_{\partial \Omega} \Big( \int_0^T | \partial_{\nu_A} u^\pm(P,t)|^2 dt \Big)^{1/2}\Big( \int_0^T | D_t^{1/2} u(P,t)|^2 dt \Big)^{1/2} dP.
\end{equation}

Thus \eqref{eq:WTF1}, \eqref{eq:WTF2} and Cauchy's inequality yield
\begin{align}\label{eq:WTF3}
&\int_0^T \int_{\Omega^\pm} \Big(|D_t^{1/4} \nabla u(X,t)|^2 + |D_t^{3/4} u(X,t)|^2 \Big) dX dt \\
& \qquad \qquad \lesssim \int_0^T \int_{\partial \Omega} \Big(\eps |\nabla u^{\pm}(P,t)|^2 + \frac{|D_t^{1/2}u(P,t)|^2}{\eps} \Big)\,dP\, dt. \nonumber
\end{align}
Finally using the fact that $\partial_t u = D_t^{1/4} D_t^{3/4} u$ on the left hand side of \eqref{eq:PITA}, Lemma \ref{Browns fractional} and  estimate \eqref{eq:WTF3}, we can conclude the proof of this lemma.
\end{proof}

\begin{Prop}\label{Lem:Rellsmoothparab}
   Let $\tilde{A}$ be as in $\mathrm{Lemmas}$ $\ref{Lem:Atilde}$ and assume that $\LL_{\tilde{A}}u=0$ in $\Omega \times (0,T)$ and $u(X,0)=0$. Then
\begin{align*}
     \int_0^T \int_{\partial \Omega} |\nabla u^{\pm}(P,t)|^2 \,dP\, dt
     & \lesssim \int_0^T \int_{\partial \Omega} \Big(|\nabla_T u(P,t)|^2 +  |D_t^{1/2}u(P,t)|^2 + \frac{|u(P,t)|^2}{\eps}\Big)\,dP\, dt \\
    & \qquad  +
      \delta^{n-1+\alpha_0}  \int_{0}^{T}\int_{\partial \Omega} \sup_{X\in \gamma_{\delta}(P)}|\nabla u(X,t)|^2\, dP \, dt ,
  \end{align*}
  for any $\delta>0$ and small enough $\eps \in (0,\eps_0(\delta))$.
\end{Prop}
\begin{proof}
Following the proof of \eqref{eq:Rell1} we have
\begin{align*}
   & \sum_{i,j,k=1}^n \int_0^T\int_{\partial \Omega} n_k(P)\,\Big(H_k(P) \tilde{a}_{ij}(P) - H_i(P) \tilde{a}_{jk}(P) - H_j(P)\tilde{a}_{ik}(P)\Big)\, \partial_i u^{\pm}(P,t)  \,\partial_j {u}^{\pm}(P,t) \, dP\, dt \\
   & \qquad   = \pm  \int_0^T \int_{\Omega^{\pm}} \Big[2\, \sum_{i=1}^n \Big(H_i(X) \,\partial_i  {u}(X,t) \, \partial_t u(X,t)\Big)
            + \sum_{i,j=1}^n b_{ij}(X) \, \partial_i u(X,t) \, \partial_j {u}(X,t) \Big] \, dX\, dt. \nonumber
   \end{align*}
Hence, the same argument as in the proof of \eqref{eq:Rell1} yields
\begin{align*}
    \int_0^T \int_{\partial \Omega} |\nabla u(P,t)|^2 \,dP\, dt
        & \lesssim \int_0^T \int_{\partial \Omega} |\nabla_T u(P,t)| |u(P,t)| \,dP\, dt
                    + \int_0^T \int_{\Omega^\pm} \langle H, \nabla u(X,t)\rangle\, \partial_t u(X,t)\, dX\, dt \\
        & \qquad + \int_0^T\int_{\Omega^\pm} |\nabla A(X)| |\nabla u(X,t)|^2 \, dX \, dt .
\end{align*}
Now applying the Cauchy inequality to the first term on the right hand side of the above estimate, Lemma \ref{Lem:umaxparab} and Lemma \ref{Lem:intboundGAO4.1} to the second and Lemma \ref{Lem:Dt12} to the third, we obtain
for every $\delta, \eps , \eps'>0$
\begin{align*}
\int_0^T \int_{\partial \Omega} |\nabla u^{\pm}(P,t)|^2 dP dt
&  \lesssim \int_0^T \int_{\partial \Omega} \Big(|\nabla_T u(P,t)|^2 +  |u(P,t)|^2 \Big) dP dt\\
& \qquad +\delta^{n-1+\alpha_0}  \int_{0}^{T}\int_{\partial \Omega} \sup_{X\in \gamma_{\delta}(P)}|\nabla u(X,t)|^2\, dP \, dt \\
& \qquad         + \frac{1}{\delta^{(1-\alpha_0)}} \int_0^T \int_{\partial \Omega} \Big(\eps |\nabla u^{\pm}(P,t)|^2 + \frac{|u(P,t)|^2}{\eps} \Big)\,dP \,dt \\
& \qquad + \int_0^T \int_{\partial \Omega} \Big(\eps' |\nabla u^{\pm}(P,t)|^2 + \frac{|D_t^{1/2}u(P,t)|^2}{\eps'} \Big)\,dP\, dt
\end{align*}
Taking $\eps$ and $\eps'$ sufficiently small, say $\eps<\eps_0(\delta)$ and $\eps'<\eps_0'$, we can absorb the terms containing $|\nabla u^{\pm}|$ in the left hand side of the estimate. This concludes the proof of the proposition.
\end{proof}

At this point we state and prove our second Rellich-type estimates for the parabolic operator associated the the diffusion matrix $B$.
\begin{Th}\label{Thm:Transferenceparab}
    Let $B$ be as in the statement of $\mathrm{Theorem}$ $\ref{Thm:Transferencetau}$ and assume that $\LL_B u=0$ in $\Omega \times (0,T)$ and $u(X,0)=0$. Then for all sufficiently small $\delta \in (0,\delta_0(r))$ and $\eps \in (0,\eps_0(\delta))$, we have
\begin{align*}
\int_0^T \int_{\partial \Omega} |\nabla u^{\pm}(P,t)|^2 \,dP\, dt
	  & \lesssim \int_0^T \int_{\partial \Omega} \Big(|\nabla_T u(P,t)|^2 +  |D_t^{1/2}u(P,t)|^2 + \frac{|u(P,t)|^2}{\eps}\Big)\,dP\, dt \\
	  & \qquad +\delta^{n-1+\alpha_0}  \int_{0}^{T}\int_{\partial \Omega} |(\nabla u)^+_*(P,t)|^2\, dP \, dt. \nonumber
    \end{align*}
   \end{Th}
\begin{proof}
    We are going to follow the scheme of the proof of  Theorem \ref{Thm:Transferencetau}.
    So unless otherwise stated, we use the same notation as there.
    We decompose $u=w+v$;
    where this time $\LL_{\tilde{A}} w=0$ in $\Omega \times (0,T)$, and
    \begin{equation*}
      v(X,t)=-\sum_{i,j=1}^n \int_0^T \int_\Omega \partial_{y_i} \G_{\tilde{A}}(X,Y,t-s)\,m_{ij}(Y)\, \partial_{y_j} u(Y,s) \,dY\, ds=: \sum_{i,j=1}^n  v_{ij}(X,t).
    \end{equation*}
    First we apply Proposition \ref{Lem:Rellsmoothparab} to the function $w$ to get
\begin{align} \label{eq:expandu}
&	\int_0^T \int_{\partial \Omega} |\nabla u^{\pm}(P,t)|^2 \,dP \,dt
	    \lesssim \int_0^T \int_{\partial \Omega} |\nabla w^{\pm}(P,t)|^2 \,dP\, dt
		   + \int_0^T \int_{\partial \Omega} |\nabla v^{\pm}(P,t)|^2 \, dP\, dt  \\
& \qquad	\lesssim \int_0^T \int_{\partial \Omega} \Big(|\nabla_T w(P,t)|^2 +  |D_t^{1/2}w(P,t)|^2 + |w(P,t)|^2\Big)\,dP \,dt \nonumber \\
& \qquad \quad+\delta^{n-1+\alpha_0} \int_0^T\int_{\partial \Omega}\sup_{X\in \gamma_{\delta}(P)}|\nabla w(X,t)|^2\, dP\, dt  + \int_0^T \int_{\partial \Omega} |\nabla v^{\pm}(P,t)|^2 \,dP\, dt \nonumber\\
& \qquad 	\lesssim \int_0^T \int_{\partial \Omega} \Big(|\nabla_T u(P,t)|^2 +  |D_t^{1/2}u(P,t)|^2 + |u(P,t)|^2\Big)\,dP \,dt \nonumber \\
& \qquad \quad + \delta^{n-1+\alpha_0} \int_0^T\int_{\partial \Omega}|(\nabla u)_*(X,t)|^2\, dP\, dt \nonumber \\
& \qquad \quad+ \int_0^T \int_{\partial \Omega}
	    \Big(|\nabla v^{\pm}(P,t)|^2 +  |D_t^{1/2}v(P,t)|^2 + |v(P,t)|^2 \Big) \, dP\, dt \nonumber \\
& \qquad \quad+ \delta^{n-1+\alpha_0} \int_0^T\int_{\partial \Omega}\sup_{X\in \gamma_{\delta}(P)}|\nabla v(X,t)|^2\, dP\, dt. \nonumber
  \end{align}

    Let $i,j=1, \dots, n$. Now Lemmas \ref{Lem:GXYm}; \ref{Lem:GXYt} $(i)$; and \ref{Lem:GDt} $(ii)$ yield
    \begin{align*}
	|\nabla v_{ij}^{\pm}(P,t)| +  |D_t^{1/2}v_{ij}(P,t)| + |v_{ij}(P,t)|
	    & \lesssim \int_0^T \int_{\Omega\cap \{|Y-P| \geq r_0 r\}} \frac{|\nabla u(Y,s)|}{|Y-P|^{n+2}} \,dY\, ds \\
	    & \lesssim \frac{1}{r^{n+2}}\int_0^T \int_\Omega |\nabla u(Y,s)| \, dY \, ds, \quad (P,t) \in S_T.
    \end{align*}
    Hence using Lemma \ref{Lem:intboundGAO4.1} we obtain
\begin{align*}
&      \int_0^T \int_{\partial \Omega}  \Big(|\nabla v^{\pm}(P,t)|^2 +  |D_t^{1/2}v(P,t)|^2 + |v(P,t)|^2 \Big)\, dP\, dt\\
& \qquad \lesssim \frac{1}{r^{2n+4}}\int_0^T \int_\Omega |\nabla u(Y,s)|^2 \,dY \,ds
 \lesssim  \frac{1}{r^{2n+4}} \int_0^T \int_{\partial \Omega} \Big(\eps |\nabla u^\pm (P,t)|^2 + \frac{|u(P,t)|^2}{\eps}\Big)\, dP\, dt.
	\end{align*}

Moreover, the same reasoning as in the proof of Theorem \ref{Thm:Transferencetau} and Lemma \ref{Lem:intboundGAO4.1} yields for $\delta< \delta_{0}(r)$ that
\begin{align*}
    \int_0^T\int_{\partial \Omega}\sup_{X\in \gamma_{\delta}(P)}|\nabla v(X,t)|^2\, dP\, dt
    & \lesssim \int_0^T\int_{\Omega} |\nabla u(Y,t)|^2 \, dY\, dt \\
   & \lesssim  \int_0^T \int_{\partial \Omega} \Big(\eps |\nabla u^\pm (P,t)|^2 + \frac{|u(P,t)|^2}{\eps}\Big)\, dP\, dt.
\end{align*}

To conclude we only need to pick $\eps$ small enough such that we can absorb the terms corresponding to $\eps|\nabla u^{\pm}|^2$  in the left hand side of \eqref{eq:expandu}, which finally yields
\begin{align*}
     \int_0^T \int_{\partial \Omega} |\nabla u^{\pm}(P,t)|^2 \, dP \, dt
	& \lesssim \int_0^T \int_{\partial \Omega} \Big(|\nabla_T u(P,t)|^2 +  |D_t^{1/2}u(P,t)|^2 + |u(P,t)|^2\Big)\, dP\, dt  \\
	& \qquad + \delta^{n-1+\alpha_0} \int_0^T\int_{\partial \Omega}|(\nabla u)^+_*(X,t)|^2\, dP\, dt. \nonumber
	\end{align*}
\end{proof}

\section{Invertibility of operators associated to the layer potentials}\label{sec:Invertibility}

\subsection{Invertibility of BSI}
The main result of this Section is the proof of the invertibility of BSI. To achieve this we shall start by showing the invertibility of the Fourier-transformed layer potentials associated to suitable pieces of the diffusion matrix $A$.
\begin{Lem}\label{Lem:DLPinverttau}
 Let $B$ be the diffusion matrix in $\mathrm{Theorem}$ $\ref{Thm:Transferencetau}.$ Then for $\tau\neq 0$ the operators $\pm 1/2 + (\mathbb{K}_{B}^\tau)^*$, defined in $\mathrm{Remark}$ $\ref{Rem:Kadjoint},$
  verify the estimate
\begin{equation*}
  \|g\|_{L^2(\partial \Omega)}
      \lesssim \Big(1 + \frac{1}{|\tau|} \Big) \Big\| \Big( \pm \frac{1}{2} + (\mathbb{K}_{B}^\tau)^* \Big)g\Big\|_{L^2(\partial \Omega)},
      \quad g \in L^2(\partial \Omega).
\end{equation*}
\end{Lem}

\begin{proof}
  We only give the proof for $1/2 + (\mathbb{K}_{B}^\tau)^*$,  since the proof
  for $-1/2 + (\mathbb{K}_B^\tau)^*$ is almost identical.

  Let $g \in L^2(\partial \Omega)$ and set $u(X,\tau):= \mathbb{S}_{B}^\tau(g)(X)$.
  By the jump relation in Proposition \ref{Prop:Jumptau2}, we can write
  \begin{equation}\label{jump for invertibility 1}
  g = \Big(\frac{1}{2} + (\mathbb{K}_{B}^\tau)^*\Big)g - \Big(-\frac{1}{2} + (\mathbb{K}_{B}^\tau)^* \Big)g
      = \partial_{\nu_{B}} u^+ - \partial_{\nu_{B}} u^-.
  \end{equation}

  From now on we drop the subscript $B$ in the normal derivatives.\\

Next, since $u(X,\tau)= \mathbb{S}_{B}^\tau(g)(X)$, we have that $\widehat{\LL}_B u=0$ in $\Omega$ and for $\delta< \delta_{0}(r)$ (see the proof of the second part of Theorem \ref{Thm:Transferencetau}), Rellich estimate \eqref{eq:RellB3} and estimate \eqref{estim:grad of SLP} (i.e. the boundedness of $(\nabla\mathbb{S}_{B}^{\tau})^\pm_*$) altogether imply that
  \begin{equation}\label{estim:pure joy 0}
   \int_{\partial \Omega} |\partial_\nu u^-(P,\tau)|^2 dP \lesssim \int_{\partial \Omega} \Big(|u(P,\tau)|^2 + |\partial_{\nu} u^+(P,\tau)|^2 + \delta^{n-1+\alpha_0}|g(P)|^2\Big)\, dP.
  \end{equation}
 Now Lemma \ref{Lem:intboundGAO} yields (observe that $\tau\neq 0$)

  $$\int_{\Omega^{\pm}} |u(Y,\tau)|^2\,dY \lesssim \int_{\partial \Omega} \Big(\frac{\eps}{|\tau|} |\partial_{\nu_A} u^{\pm}(P,\tau)|^2 + \frac{|u(P,\tau)|^2}{\eps|\tau|} \Big)dP,$$
  and
  $$\int_{\Omega^{\pm}} |\nabla u(X,\tau))|^2\,dY \lesssim \int_{\partial \Omega} \Big(\eps|\partial_{\nu_A} u^{\pm}(P,\tau)|^2 + \frac{|u(P,\tau)|^2}{\eps} \Big)dP.$$

  Therefore adding these two estimates we obtain
   \begin{equation}\label{estim:pure joy}
   \int_{\Omega^\pm} \Big( |\nabla u(X,\tau)|^2 + |u(X,\tau)|^2   \Big) \, dX \lesssim \int_{\partial \Omega} \Big(\eps(1+\frac{1}{|\tau|}) |\partial_{\nu_A} u^{\pm}(P,\tau)|^2 + \frac{1}{\eps}(1+\frac{1}{|\tau|})|u(P,\tau)|^2 \Big)\,dP.
  \end{equation}
  On the other hand \eqref{estim:pure joy} and Lemma \ref{Lem:intboundGAO3.2} yield
\begin{align*}
      \int_{\partial \Omega} |u(P,\tau)|^2 \, dP
      &\leq C \int_{\Omega^\pm} \Big( |\nabla u(X,\tau)|^2 + |u(X,\tau)|^2   \Big) \, dX\\
	&\leq C \eps \Big(1 + \frac{1}{|\tau| } \Big) \int_{\partial \Omega} |\partial_{\nu} u^+(P,\tau)|^2 \, dP
		  + \frac{C}{\eps}(1+\frac{1}{|\tau|})\int_{\partial \Omega} |u(P,\tau)|^2 \, dP.
  \end{align*}
  Thus taking $\eps=2C(1+1/|\tau|)$) we obtain
  \begin{align}\label{estim:pure jopy 2}
      \int_{\partial \Omega} |u(P,\tau)|^2 \, dP
	\lesssim \Big(1 + \frac{1}{|\tau|} \Big)^2 \int_{\partial \Omega} |\partial_{\nu} u^+(P,\tau)|^2 \, dP.
  \end{align}
  Thus combining \eqref{estim:pure joy 0} and \eqref{estim:pure jopy 2} we have
\begin{align}\label{estim:pure joy 3}
   \int_{\partial \Omega} |\partial_\nu u^+(P,\tau)|^2 dP + \int_{\partial \Omega} |\partial_\nu u^-(P,\tau)|^2 dP
   	 &\lesssim \Big(1 + \frac{1}{|\tau|} \Big)^2 \int_{\partial \Omega} |\partial_{\nu} u^+(P,\tau)|^2 \, dP \\
	 & \qquad + \delta^{n-1+\alpha_0}\int_{\partial \Omega} |g(P)|^2\, dP. \nonumber
\end{align}

Therefore, using \eqref{jump for invertibility 1} and \eqref{estim:pure joy 3}, we obtain for $\delta< \delta_{0}(r)$
\begin{align*}
      \|g\|_{L^2(\partial \Omega)}
	  & \lesssim \| \partial_\nu u^+ \|_{L^2(\partial \Omega)} + \| \partial_\nu u^- \|_{L^2(\partial \Omega)} \\
	  &  \leq C \Big(1 + \frac{1}{|\tau|} \Big) \Big\| \Big(\frac{1}{2} + (\mathbb{K}_B^\tau)^*\Big)g \Big\|_{L^2(\partial \Omega)} +C\delta^{\frac{n-1+\alpha_0}{2}} \|g\|_{L^2(\partial \Omega)}.
  \end{align*}
 Now if we take $\delta< \mathrm{min}\{ \delta_{0}(r), (\frac{1}{2C})^{2/(n-1+\alpha_0)}\}$ then
   \begin{align*}
   \|g\|_{L^2(\partial \Omega)}
	  & \lesssim \| \partial_\nu u^+ \|_{L^2(\partial \Omega)} + \| \partial_\nu u^- \|_{L^2(\partial \Omega)}
	    \lesssim \Big(1 + \frac{1}{|\tau|} \Big) \Big\| \Big(\frac{1}{2} + (\mathbb{K}_B^\tau)^*\Big)g \Big\|_{L^2(\partial \Omega)},
  \end{align*}
which is the desired estimate.
\end{proof}

The final step here is to show the invertibility of the Fourier-transformed BSI associated to the matrix $A$.

\begin{Lem}\label{invertability of hat BSI}
Let $A$ satisfy \hyperref[A1]{\small{$(A1)$}} -- \hyperref[A4]{\small{$(A4)$}} defined as in $\mathrm{Lemma}$ $\ref{Lem:Ar}.$ Then for $\tau\neq 0$ the operators $\pm 1/2 + (\mathbb{K}_{A}^\tau)^*$,
are invertible in $L^2(\partial \Omega)$ and satisfy the estimate

\begin{equation*}
  \|g\|_{L^2(\partial \Omega)}
      \lesssim \Big(1 + \frac{1}{|\tau|} \Big) \Big\| \Big( \pm \frac{1}{2} + (\mathbb{K}_{A}^\tau)^* \Big)g\Big\|_{L^2(\partial \Omega)},
      \quad g \in L^2(\partial \Omega).
\end{equation*}

\end{Lem}
\begin{proof}
As before, we write $A$ as $B+C$ where $B:= \tilde{A}+A-A^{(r)}$ and $C:=A^{(r)}-\tilde{A}$. Then we observe that

$$\pm \frac{1}{2} + (\mathbb{K}_{A}^\tau)^* = \pm \frac{1}{2} + (\mathbb{K}_{B}^\tau)^* + (\mathbb{K}_{A}^\tau)^* - (\mathbb{K}_{B}^\tau)^* ,$$

and Lemma \ref{Lem:DLPinverttau} yields that
\begin{equation*}
  \|g\|_{L^2(\partial \Omega)}
      \lesssim \Big(1 + \frac{1}{|\tau|} \Big) \Big\| \Big( \pm \frac{1}{2} + (\mathbb{K}_{B}^\tau)^* \Big)g\Big\|_{L^2(\partial \Omega)},
      \quad g \in L^2(\partial \Omega).
\end{equation*}

Moreover since $\pm \frac{1}{2} + (\mathbb{K}_{B}^\tau)^*= \pm \frac{1}{2} + (\mathbb{K}_{A}^\tau)^* - (\mathbb{K}_{A}^\tau)^* +(\mathbb{K}_{B}^\tau)^*$, we have
\begin{align*}
  \|g\|_{L^2(\partial \Omega)}
      & \lesssim \Big(1 + \frac{1}{|\tau|} \Big) \Big\| \Big( \pm \frac{1}{2} + (\mathbb{K}_{B}^\tau)^* \Big)g\Big\|_{L^2(\partial \Omega)}\\			 & \lesssim \Big(1 + \frac{1}{|\tau|} \Big) \Big\| \Big( \pm \frac{1}{2} + (\mathbb{K}_{A}^\tau)^* - (\mathbb{K}_{A}^\tau)^* +(\mathbb{K}_{B}^\tau)^* \Big)g\Big\|_{L^2(\partial \Omega)}\\
      & \lesssim \Big(1 + \frac{1}{|\tau|} \Big) \Big\| \Big( \pm \frac{1}{2} + (\mathbb{K}_{A}^\tau)^*\Big) g\|_{L^2(\partial \Omega)}+ \Big(1 + \frac{1}{|\tau|} \Big) \Big\| ((\mathbb{K}_{A}^\tau)^* -(\mathbb{K}_{B}^\tau)^*)g\Big\|_{L^2(\partial \Omega)}.
\end{align*}
Now using Theorem \ref{boundedness Ftrans BSI of the difference} (estimate \eqref{for ftransBSI adjoint}) and Lemma \ref{Lem:Ar}, we have
\begin{align*}
\Big\| ((\mathbb{K}_{A}^\tau)^* -(\mathbb{K}_{B}^\tau)^*)g\Big\|_{L^2(\partial \Omega)}
& \lesssim \Vert A-B\Vert_{\infty}^{1/2} \| g\|_{L^2(\partial \Omega)}
\lesssim \Vert C\Vert_{\infty}^{1/2} \| g\|_{L^2(\partial \Omega)}\\
&\lesssim \Vert A^{(r)}-\tilde{A}\Vert_{\infty}^{1/2} \| g\|_{L^2(\partial \Omega)}\leq c r^{\alpha_0/2} \| g\|_{L^2(\partial \Omega)},
\end{align*}
where the last estimate follows from Lemma \ref{Lem:Ar}. Therefore choosing $r$ so small that
$c (1+1/|\tau|) r^{\alpha_0 /2}<1/2$, we will have
\begin{equation*}
  \|g\|_{L^2(\partial \Omega)}
      \lesssim \Big(1 + \frac{1}{|\tau|} \Big) \Big\| \Big( \pm \frac{1}{2} + (\mathbb{K}_{A}^\tau)^* \Big)g\Big\|_{L^2(\partial \Omega)}.
\end{equation*}
This shows that the operator $1/2 + (\mathbb{K}_A^\tau)^* : L^2(\partial \Omega) \longrightarrow L^2(\partial \Omega)$ is injective and has closed range in $L^2(\partial \Omega)$.
To conclude, we need to show that the range of $1/2 + (\mathbb{K}_A^\tau)^*$
  is equal to $L^2(\partial \Omega)$. We write
  \begin{equation}\label{eq:decomp}
      \Big( \frac{1}{2} + \mathbb{K}_A^\tau \Big) - \Big(\frac{1}{2} + \mathbb{K}_A^\tau \Big)^*
	= \Big( \mathbb{K}_A^\tau - \mathbb{K}_A^0 \Big)
	+ \Big( \mathbb{K}_{A}^0 - (\mathbb{K}_A^0)^* \Big)
	+ \Big( \mathbb{K}_{A}^0 - \mathbb{K}_A^\tau \Big)^*.
  \end{equation}
  First we observe that Lemma \ref{Lem:GXYFourier} $(iv)$ and Lemma \ref{Lem:compact} yield that
  $\mathbb{K}_A^\tau - \mathbb{K}_{A}^0$ is a compact operator, hence $(\mathbb{K}_A^0 - \mathbb{K}_A^\tau)^*$ is also compact. Now we claim that the operator $\mathbb{K}_A^0 - (\mathbb{K}_A^0)^*$ is compact as well. Indeed, let $\widetilde{\Gamma}_{A}(X,Y)$ be the fundamental solution of the elliptic operator $\nabla\cdot(A\nabla).$ Then as was shown in \cite[Theorem 2.6]{MiTa} one has
  $$\widetilde{\Gamma}_{A}(X,Y)= E_0 (X,Y)+E_1(X,Y),$$
  where $$E_0 (X, Y) = \frac{C_n}{(\det {A}(Y))^{1/2} \langle A^{-1}(Y) (X-Y), X-Y\rangle ^{(n-2)/2}}$$
  and $|\nabla_{X} E_1(X,Y)|\lesssim |X-Y|^{1-n+\alpha/2},$ where $\alpha$ is the H\"older exponent of $A$.

Now since the integral kernel of the operator $(\mathbb{K}_A^0)^*$ is the normal derivative of $\widetilde{\Gamma}_{A}(X,Y)$, estimates above yield that the integral kernel of $\mathbb{K}^0 - (\mathbb{K}_A ^0)^*$ is of the form

$$\frac{\langle \nabla_{Y} (\langle A^{-1}(X) (Y-X), Y-X\rangle), A^{-1}(X) N(Y)\rangle}{(\det {A}(X))^{1/2}\langle A^{-1}(X) (Y-X), Y-X\rangle ^{n/2}}-\frac{\langle\nabla_{X} (\langle A^{-1}(Y) (X-Y), X-Y\rangle), A^{-1}(Y)N(X)\rangle}{(\det {A}(Y))^{1/2}\langle A^{-1}(Y) (X-Y), X-Y\rangle ^{n/2}},
$$ plus a kernel, which in light of the estimate for $\nabla_{X} E_1(X,Y)$ above, gives rise to a compact operator. This implies that the integral kernel of $\mathbb{K}_A^0 - (\mathbb{K}_A^0)^*$ can be estimated by $|X-Y|^{2-n}$  modulo a kernel which gives rise to a compact operator. Hence, using Lemma \ref{Lem:compact}, we conclude that  $\mathbb{K}^0 - (\mathbb{K}^0)^*$ is compact. Thus, the operator on the left hand side of \eqref{eq:decomp} is compact and since $1/2 + (\mathbb{K}_A^\tau)^*$ is injective, Lemma \ref{surjectivity of Fredholm} yields the surjectivity of $1/2 + (\mathbb{K}_A^\tau)^*$.

\end{proof}

Now we are ready to state and prove one of main results of this paper, which concerns the invertibility of the boundary singular integral operator which is the backbone of the solvability results in Section \ref{sec:SIBVP}.

\begin{Th}\label{Thm:DLPinvert}
  Assume that $A$ verifies \hyperref[A1]{\small{$(A1)$}} -- \hyperref[A4]{\small{$(A4)$}}.
  Then, the operator $\pm 1/2 + \KK_A^*$ defined in Remark $\ref{Rem:KKadjoint}$, is invertible in $L^2(S_T)$.
\end{Th}

\begin{proof}
  We only show the invertibility of $1/2 + \KK_A^*$ since the invertibility of $-1/2 + \KK_A^*$ is done in a similar way. First we show the injectivity of $1/2 + \KK_A^*$. To this end suppose $f\in L^2(S_T)$ is such that $(1/2 + \KK_A^*) f=0$. Then if $u(X,t) = \mathcal{S} f(X,t)$ then the jump relation in Proposition \ref{Prop:Jump2} yields that $\partial_{\nu} u^+ =0.$  Now using \eqref{estim: energy related}, it follows that $u=0$ in $\Omega^+  \times (0,T)$ and hence $u=0$ in $S_T$. Furthermore, \eqref{estim: energy related} also yields that $u=0$ in $\Omega^-  \times (0,T)$, and therefore $\partial_{\nu} u^- =0.$ Thus once again the jump relation in Proposition \ref{Prop:Jump2} yields that $f=0$ and hence $(1/2 + \KK_A^*)$ is injective. In order to prove the surjectivity of $(1/2 + \KK_A^*)$ we follow the strategy of Shen in \cite[Theorem 4.1.18]{Sh1}. Take an arbitrary $f\in L^2(S_T)$ and extend it to $\partial \Omega \times \R$ by setting
\begin{equation*}
h(P,t):=
  \begin{cases}
     f(P,t)\quad  &\textrm{if}\quad 0<t<T,\\
    -f(P, t-T)\quad &\text{if}\quad T<t<2T,\\
    0 \quad &\textrm{if}\quad 2T<t \quad  \textrm{or} \quad  t<0.
    \end{cases}
\end{equation*}
This function is in $L^2(\partial \Omega \times \R)\cap L^1(\partial \Omega \times \R)$, and therefore its Fourier transform in $t$, i.e. $\widehat{h}(P, \tau)$ is continuous in $\tau.$ Furthermore $\widehat{h}(\cdot, \tau)\in L^2(\partial \Omega)$ for each $\tau\in \R$ and $\widehat{h}(P, 0)=0.$ Using Plancherel's theorem it follows that $\Vert \widehat{h}\Vert_{L^2(\partial \Omega \times \R)}= C \Vert h\Vert_{L^2(\partial \Omega \times \R)}\lesssim \Vert f\Vert_{L^2(S_T)}.$ Now if $\tau \neq 0$ then using the invertibility of $1/2 + (\mathbb{K}_{A}^\tau)^*$ on $L^2(\partial \Omega)$ established in Lemma \ref{invertability of hat BSI}, there exists $\varphi(P, \tau)\in L^2(\partial \Omega)$ such that
\begin{equation}\label{equation for varphi}
  (1/2 + (\mathbb{K}_{A}^\tau)^*)\,\varphi( P,\tau)= \widehat{h}(P, \tau).
\end{equation}
For $\tau=0$, thanks to the fact that $\widehat{h}(P, 0)=0,$ we can set $\varphi(P, 0)=0$ and hence produce a solution to \eqref{equation for varphi} for all values of $\tau$.  We will show that
\begin{equation}\label{estim:L2 boundary estim varphi}
\Vert \varphi \Vert_{L^2(\partial \Omega \times \R)}\leq C_{T} \Vert f\Vert_{L^2(S_T)}.
\end{equation}
Indeed, Lemma \ref{Lem:DLPinverttau} yields that
 \begin{equation}\label{estim:varphi tau L2}
  \|\varphi(\cdot, \tau)\|_{L^2(\partial \Omega)}
      \lesssim \Big(1 + \frac{1}{|\tau|} \Big) \Vert \widehat{h}(\cdot, \tau)\Vert_{L^2(\partial \Omega)}.
 \end{equation}

 Therefore $\int_{|\tau|\geq 1}\int_{\partial \Omega} |\varphi (P, \tau)|^2\, dP\, d\tau \lesssim \Vert \widehat{h}\Vert^{2}_{L^2(\partial \Omega \times \R)} \lesssim \Vert f\Vert^{2}_{L^2(S_T)}. $ On the other hand for $|\tau|<1$, using the fact that $\widehat{h}(P,0)=0$ and applying the mean value theorem and the Cauchy-Schwarz inequality we deduce that

 \begin{equation*}
   |\widehat{h}(P, \tau)| \leq |\tau|\int_{0}^{2T} t\, |h(P,t)| \, dt \leq C_T \,|\tau|\, \{\int_{0}^{T} |f(P,t)|^2 \, dt\}^{1/2}.
 \end{equation*}
 This together with \eqref{estim:varphi tau L2} yield that $\int_{|\tau|< 1}\int_{\partial \Omega} |\varphi (P, \tau)|^2\, dP\, d\tau  \leq C_T \, \Vert f\Vert^{2}_{L^2(S_T)}. $ This establishes \eqref{estim:L2 boundary estim varphi}.
 Now let $\psi(P,t)\in L^2(\partial \Omega \times \R)$ be such that its Fourier transform in $t$ is equal to $\varphi(P, \tau)$. The for almost all $(P,t)\in \partial \Omega \times \R$ we have
  \begin{equation*}
   \frac{1}{2}\psi(P,t) + \int_{-\infty}^{t}\int_{\partial \Omega}\mathcal{K}_{A}^{*} (P,Q, t-s)\, \psi(Q, s) \, dQ\, ds= h(P, t).
 \end{equation*}
Furthermore, since $h(P,t)=0$ for $t<0$, an argument similar to that of the injectivity of $1/2 + \KK_{A}^*$ shows that $\psi(P,t)=0$ when $t<0.$ Now, since for $(P,t)\in S_T$, $h(P, t)=f(P,t)$, we have for almost every $(P,t)\in S_T$ that $$\frac{1}{2}\psi(P,t) + \int_{0}^{t}\int_{\partial \Omega}\mathcal{K}_{A}^{*} (P,Q, t-s)\, \psi(Q, s) \, dQ\, ds= (1/2 + \KK^*)\psi(P,t)= f(P, t).$$ Therefore $1/2 + \KK_{A}^*$ is also surjective and hence because of its injectivity, invertible.
\end{proof}

\begin{Cor}\label{Cor:DLPinvert2}
  Assume that $A$ satisfies the conditions \hyperref[A1]{\small{$(A1)$}} -- \hyperref[A4]{\small{$(A4)$}}. Then, the operators $\pm 1/2 + \KK_{A}$ are invertible in $L^2(S_T)$.
\end{Cor}

\subsection{Invertibility of the SLP}\label{subsec:Invertibility of the SLP}

For the solvability of the regularity problem we would also need the following invertibility result.
\begin{Th}\label{Thm:DLPinvert1}
  Assume that $A$ satisfies the conditions \hyperref[A1]{\small{$(A1)$}} -- \hyperref[A4]{\small{$(A4)$}}.
  Then, the single-layer operator $\Ss_A$ is invertible from $L^2(S_T)$ to $H^{1,1/2}(S_T)$.
\end{Th}
\begin{proof}
First we show that $\Ss_A$ is injective. To this end take $f\in L^2(S_T)$ and set $u:=\Ss_B f$, with b as in Theorem \ref{Thm:Transferenceparab}. Then $\LL_B u=0$ and $u(X,0)=0$. Moreover the jump relation in Proposition \ref{Prop:Jump2} yields

\begin{equation*}%
  f = \Big(\frac{1}{2} + (\mathcal{K}_{B})^* \Big)f - \Big(-\frac{1}{2} + (\mathcal{K}_{B})^* \Big)f
      = \partial_{\nu_{B}} u^+ - \partial_{\nu_{B}} u^-.
  \end{equation*}

Therefore the Rellich estimates in Theorem \ref{Thm:Transferenceparab} (for $\delta< \delta_{0}(r)$) and the boundedness result \eqref{estim:grad of SLP parabolic}, yield that
\begin{align*}
\Vert f\Vert_{L^2( S_T)}
    & \lesssim \Vert\partial_{\nu_{B}} u^{+} \Vert_{L^2( S_T)} + \Vert \partial_{\nu_{B}} u^{-}\Vert_{L^2( S_T)} \\
    & \lesssim \Vert \nabla_T u\Vert_{L^2(S_T)} +  \Vert D_t^{1/2}u\Vert_{L^2(S_T)} + \Vert u\Vert_{L^2(S_T)}+ \delta^{\frac{n-1+\alpha_0}{2}} \Vert(\nabla \Ss_B f)_*\Vert_{L^2(S_T)}\\
    & \lesssim \Vert u\Vert_{H^{1,1/2}(S_T)}+ \delta^{\frac{n-1+\alpha_0}{2}} \Vert f\Vert_{L^2(S_T)}.
\end{align*}							
So we arrive at the invertibility estimate
\begin{equation*}
\Vert f\Vert_{L^2( S_T)}\lesssim \Vert \Ss_B f\Vert_{H^{1,1/2}(S_T)}+\delta^{\frac{n-1+\alpha_0}{2}} \Vert f\Vert_{L^2(S_T)}.
\end{equation*}

However, since $\Ss_B= \Ss_A +\Ss_B -\Ss_A$, we also have

\begin{align}\label{almost invertibility of A SLP}
 \Vert f\Vert_{L^2( S_T)}
&\lesssim \Vert \Ss_A f\Vert_{H^{1,1/2}(S_T)} + \Vert (\Ss_B -\Ss_A) f\Vert_{H^{1,1/2}(S_T)}
+\delta^{\frac{n-1+\alpha_0}{2}} \Vert f\Vert_{L^2(S_T)}\\
&  = \Vert \Ss_A f\Vert_{H^{1,1/2}(S_T)}+\Vert \nabla (\Ss_B -\Ss_A) f\Vert_{L^2(S_T)}
 +  \Vert D_t^{1/2}(\Ss_B -\Ss_A ) f\Vert_{L^2(S_T)} \nonumber \\
& \qquad  + \Vert (\Ss_B -\Ss_A) f\Vert_{L^2(S_T)}+\delta^{\frac{n-1+\alpha_0}{2}} \Vert f\Vert_{L^2(S_T)} \nonumber.
\end{align}
Now using Theorem \ref{boundedness SLP of the difference} and Lemma \ref{Lem:Ar} we can show that the last three terms on the right hand side of \eqref{almost invertibility of A SLP} are bounded by $\Vert A-B\Vert_{\infty}^{1/2} \Vert f\Vert_{L^2(S_T)} \lesssim r^{\alpha_0/2} \Vert f\Vert_{L^2(S_T)}.$ So taking $r$ small enough we can absorb the second, third and fourth  terms on the right hand side of \eqref{almost invertibility of A SLP} into the left hand side, and thereafter choosing $\delta$ in a suitable manner, we can do the same with the fifth term on the right hand side of \eqref{almost invertibility of A SLP} and finally obtain

\begin{equation}\label{invertibility of A SLP}
\Vert f\Vert_{L^2( S_T)}\lesssim \Vert \Ss_A f\Vert_{H^{1,1/2}(S_T)}.
\end{equation}

Thus if $u=0$ then \eqref{invertibility of A SLP} yields that $f=0$, and therefore $\Ss_A$ is injective and its range is closed in $H^{1,1/2}(S_T)$.
The surjectivity of the single layer potential is quite standard and is handled in many papers, see e.g. \cite[Theorem 5.6 ]{BrownPhD} or \cite[Theorem 4.2.1]{Sh1}. We would like to mention that, the main part of the surjectivity-proof, hinges on two main steps; one is an approximation argument and the existence of regular solution to the Dirichlet problem for equations with H\"older-continuous coefficients on smooth domains, see \cite{LSU}. The other is the existence and uniqueness of the solution to the Neumann problem, which will be established in Theorem \ref{Th:Neumann}.

\end{proof}

\section{Solvability of initial boundary value problems}\label{sec:SIBVP}

Given all the information and estimates which have been gathered in the previous sections, one has the following three theorems which state the well-posedness of the DNR problems. The proofs of these statements are completely standard and can be done in the same way as e.g. \cite{FaRi}, hence we omit the details.
\subsection{Solvability of the Dirichlet problem}
\begin{Th}\label{Th:Dirichlet}
  Let $T>0$, $\Omega$ be a Lipschitz domain and
  $A$ a matrix verifying properties \hyperref[A1]{\small{$(A1)$}} -- \hyperref[A4]{\small{$(A4)$}}.
  If $f \in L^2(S_T)$, then the function $u:=\DD (-1/2+\KK_{A})^{-1}f$ is the unique solution of the following
  initial Dirichlet boundary value problem
  \begin{equation*} %
    \partial_t u(X,t) - \nabla_X\cdot\big(A(X) \nabla_X u(X,t)\big)
		= 0 \quad \text{in} \quad \Omega  \times (0,T),
  \end{equation*}

  \begin{equation*} %
    u(X,0)= 0, \quad \text{a.e. } X \in \Omega,
  \end{equation*}

  \begin{equation*} %
     \lim_{\substack{X \to P \\ X \in \g_{+}(P)}} u(X,t)
	= f(P,t), \quad \text{a.e. } (P,t) \in S_T,
  \end{equation*}
  and
  \begin{equation*} %
    \| u^{\pm}_*\|_{L^2(S_T)} \lesssim \| f \|_{L^2(S_T)}.
  \end{equation*}
\end{Th}

\begin{proof}
  The proof goes along the same lines as \cite[Theorem 2.3]{FaRi}.
\end{proof}

\subsection{Solvability of the Neumann problem}

\begin{Th}\label{Th:Neumann}
  Let $T>0$, $\Omega$ be a Lipschitz domain and
  $A$ a matrix verifying properties \hyperref[A1]{\small{$(A1)$}} -- \hyperref[A4]{\small{$(A4)$}}.
  If $f \in L^2(S_T)$ with $\int_{S_T} f =0,$
  then the function $u:=\Ss_{A} (1/2 +\KK_{A}^*)^{-1}f$ is the unique solution of the following
  initial Neumann boundary value problem
  \begin{equation*}\label{eq:Neumanneq}
    \partial_t u(X,t) - \nabla_X\cdot\big(A(X) \nabla_X u(X,t)\big)
		= 0 \quad \text{in} \quad \Omega  \times (0,T),
  \end{equation*}

  \begin{equation*}\label{eq:Neumanninit}
    u(X,0)= 0, \quad \text{a.e. } X \in \Omega,
  \end{equation*}

  \begin{equation*}\label{eq:Neumannconv}
     \partial_{\nu_A} u(P,t)= f(P,t), \quad \text{a.e. } (P,t) \in S_T,
  \end{equation*}
  and
  \begin{equation*}\label{eq:Neumannmax}
    \| u^\pm_*\|_{L^2(S_T)} + \| (\nabla u)^\pm_*\|_{L^2(S_T)} + \| (D_t^{1/2} u)^\pm_* \|_{L^2(S_T)} \lesssim \| f \|_{L^2(S_T)}.
  \end{equation*}

\end{Th}

\begin{proof}
  The proof goes along the same lines as \cite[Theorem 2.4]{FaRi}.
\end{proof}

\subsection{Solvability of the regularity problem}

\begin{Th}\label{Th:Regularity}
  Let $T>0$, $\Omega$ be a Lipschitz domain and
  $A$ a matrix verifying properties \hyperref[A1]{\small{$(A1)$}} -- \hyperref[A4]{\small{$(A4)$}}.
  If $f \in H^{1,1/2}(S_T)$, then the function $u:=\Ss_{A} (\Ss_{|_{S_T}})^{-1} f$ is the unique solution of the following
  initial regularity boundary value problem
  \begin{equation*}\label{eq:Regeq}
    \partial_t u(X,t) - \nabla_X\cdot\big(A(X) \nabla_X u(X,t)\big)
		= 0 \quad \text{in} \quad \Omega  \times (0,T),
  \end{equation*}

  \begin{equation*}\label{eq:Reginit}
    u(X,0)= 0, \quad \text{a.e. } X \in \Omega,
  \end{equation*}

  \begin{equation*}\label{eq:Regconv}
     \lim_{\substack{X \to P \\ X \in \g_{+}(P)}} u(X,t)
	= f(P,t), \quad \text{a.e. } (P,t) \in S_T,
  \end{equation*}
  and
  \begin{equation*}\label{eq:Neumannmax}
    \| u^\pm_*\|_{L^2(S_T)} + \| (\nabla u)^\pm_*\|_{L^2(S_T)} + \| (D_t^{1/2} u)^\pm_* \|_{L^2(S_T)} \lesssim \| f \|_{H^{1,1/2}(S_T)}.
  \end{equation*}
\end{Th}

\begin{proof}
The proof goes along the same lines as \cite[Theorem 3.4]{FaRi}.
\end{proof}

\end{document}